\documentclass[11pt]{article}

\usepackage{amsfonts}
\usepackage{amsmath,amsthm}
\usepackage{graphicx,epstopdf,epsfig,multirow,epic,bm}
\usepackage{color}
\usepackage{multicol}
\usepackage{algorithm}
\usepackage{algpseudocode}
\usepackage{paralist}

\usepackage{makeidx}
\usepackage{showidx}
\usepackage{textcomp}

\usepackage{CJK}

\theoremstyle{definition}
\newtheorem{thm}{Theorem}
\newtheorem{cor}[thm]{Corollary}

\newtheorem{prop}[thm]{Proposition}
\theoremstyle{definition}
\newtheorem{defn}{Definition}
\theoremstyle{definition}
\newtheorem{rem}{Remark}

\theoremstyle{definition}
\newtheorem{problem}{Problem}

\theoremstyle{definition}
\newtheorem{conj}{Conjecture}
\theoremstyle{definition}
\newtheorem{example}{Example}
\theoremstyle{definition}

\DeclareGraphicsExtensions{.eps,.eps.gz}
\normalsize \rm
\usepackage{titlesec}
\usepackage[nottoc]{tocbibind}

\newcommand{\qqed}{\hfill $\square$}

\begin{document}

\thispagestyle{empty}

\begin{center}
{\huge \textbf{Topological Structures Of Sets\\[12pt] And Their Subsets}}\\[12pt]
{\Large \textbf{Fei Ma \footnote{~School of Computer Science, Northwestern Polytechnical University, Xi'an, 710072, CHINA \\ email: mafei123987@163.com}~and ~Bing Yao \footnote{~College of Mathematics and Statistics, Northwest Normal University, Lanzhou, 730070 CHINA \\ email: yybb918@163.com}}}\\[12pt]
\end{center}

\begin{quote}
\textbf{Abstract:} For real application and theoretical investigation of ordinary hypergraphs and non-ordinary hypergraphs, researchers need to establish standard rules and feasible operating methods. We propose a visualization tool for investigating hypergraphs by means of the natural topological structure of finite sets and their subsets, so we are able to construct various non-ordinary hypergraphs, and to reveal topological properties (such as hamiltonian cycles, maximal planar graphs), colorings, connectivity, hypergraph group, isomorphism and homomorphism of hypergraphs.\\[6pt]
\end{quote}



\setcounter{page}{1}
\pagenumbering{arabic}

\vskip 0.8cm

\section{Introduction And Preliminary}

\subsection{Researching background}

Things in the world are divided into sets of different categories, each with its own subsets. According to certain rules, the subsets of a set form many relational groups, which can be described by hypergraphs in mathematics, as known, hypergraphs are an important branch of graph theory and information science.

\subsubsection{Hypergraphs are combinatorics of finite sets}

In the abstrct of the book of ``Hypergraphs: Combinatorics of Finite Sets'' by Claude Berge (Ref. \cite{Claude-Berge-North-Holland-1989}), Berge said: \emph{Graph Theory has proved to be an extremely useful tool for solving combinatorial problems in such diverse areas as Geometry, Algebra, Number Theory, Topology, Operations Research and Optimization. It is natural to attempt to generalize the concept of a graph, in order to attack additional combinatorial problems.} The idea of looking at a family of sets from this standpoint took shape around 1960. In regarding each set as a ``generalized edge'' and in calling the family itself a ``hypergraph'', the initial idea was to try to extend certain classical results of Graph Theory such as the theorems of Tur\'{a}n and K\"{o}nig. It was noticed that this generalization often led to simplification.

As a \emph{subset system} of a finite set, hypergraph is the most general \emph{discrete structure}, which is widely used in information science, data structure, life science and other fields. However, \emph{hypergraphs are more difficult to draw on paper than graphs}, there are methods for the visualization of hypergraphs, such as Venn diagram, PAOH \emph{etc}.

Professor Wang, in his book \cite{Jianfang-Wang-Hypergraphs-2008}, investigated ordinary hypergraphs (also, hypergraphs), and pointed: ``\emph{When computers become very powerful, the security theory of information science requires hypergraphs to procedure and protect information data}.''

\subsubsection{Applications of hypergraphs}

Since the hypergraph theory was proposed systematically by Claude Berge in 1989, more and more attention has been paid to the research of hypergraph theory and its application.

\textbf{Hypergraphs in the era of post-quantum encryption.} As known, all things of high-dimensional data sets are interrelated and interact on each other, we need to study the complex structures of high-dimensional data sets, and the interaction between high-dimensional data sets, one of research tools is \emph{hypergraph}.

Hypergraphs can tease out of big data sets proposed in ``\emph{How Big Data Carried Graph Theory Into New Dimensions}'' by Stephen Ornes \cite{how-big-data-graph-theory-20210819}. And large data sets in practical application show that the impact of groups often exceeds that of individuals, thereby, it is more and more important to study hypergraphs.

Undirected hypergraphs are useful in modelling such things as satisfiability problems \cite{Uriel-Kim-Han-Eran-2006}, databases \cite{Beeri-Fagin-Maier-Yannakakis-1983}, machine learning \cite{Huang-Zhang-Yu-Jeffrey-Xu-2015}, and Steiner tree problems \cite{Brazil-M-Zachariasen-2015}. They have been extensively used in machine learning tasks as the data model and classifier regularization (mathematics) \cite{Zhou-Dengyong-Huang-Jiayuan-Scholkopf-Bernhard-2006}. The applications include recommender system (communities as hyperedges) \cite{Tan-Bu-Chen-Xu-Wang-He-2013}, image retrieval (correlations as hyperedges) \cite{Liu-Qingshan-Huang-Metaxas-Dimitris-2013}, and bioinformatics (biochemical interactions as hyperedges) \cite{Patro-Rob-Kingsoford-Carl-2013}. Representative hypergraph learning techniques include hypergraph spectral clustering that extends the spectral graph theory with hypergraph Laplacian \cite{Gao-Wang-Zha-Shen-Li-Wu-Xindong-2013}, and hypergraph semi-supervised learning that introduces extra hypergraph structural cost to restrict the learning results \cite{Tian-Hwang-TaeHyun-Kuang-Rui-2009}. For large scale hypergraphs, a distributed framework \cite{Tan-Bu-Chen-Xu-Wang-He-2013} built using Apache Spark is also available.

Directed hypergraphs can be used to model things including telephony applications \cite{Goldstein-A-1982}, detecting money laundering \cite{Ranshous-Stephen-more-2017}, operations research \cite{Ausiello-Giorgio-Laura-Luigi-2017}, and transportation planning. They can also be used to model Horn-satisfiability \cite{Gallo-Longo-Pallottino-Nguyen-1993}. Directed hypergraphs can be used to model things including telephony applications, detecting money laundering, operations research, and transportation planning, and can also be used to model Horn-satisfiability.

\textbf{Hypergraphs in Deep Learning and Neural Network.} Graph networks attempt to unify various networks in deep learning, and are the generalization of graph neural networks (GNN) in deep learning theory and probabilistic graphical model (PGM). A graph network is composed of graph network blocks and has a flexible topology structure, and it can be transformed into various forms of connectionism models including feedforward neural networks (FNN) and recursive neural network (RNN) \emph{etc}. A graph network framework based on graph network blocks defines a class of functions for relational reasoning on graph structure representations. The graph network framework summarizes and extends various graph neural networks, MPNN, and NLNN methods, and supports building complex architectures from simple building blocks. Combination generalization is the primary task for artificial intelligence to achieve similar abilities to humans, and the structured representation and computation are the key to achieving this goal, and the key to achieving this goal is to represent data in a structured manner, as well as structured computation (Ref. \cite{Peter-Battaglia-Jessica-et-al-arXiv-2018}).

The technique of graph/network embedding can help computer to efficiently analyze and process the complex graph data via vector operations. Graph Neural Network, which aggregates the topological information of the neighbourhoods of each node in a graph to implement graph/network embedding. Most current work is based on static hypergraph structure, making it hard to effectively transmit information. To address this problem, Dynamic Hypergraph Neural Networks based on Key Hyperedges (DHKH) model is proposed. Considering that the graph structure data in the real world is not uniformly distributed both semantically and structurally, we define the key hyperedge as the subgraph composed of a small number of key nodes and related edges in a graph. The key hyperedge can capture the key high-order structure information, which is able to enhance global topology expression. With the supporting of hyperedge and key hyperedge, DHKH can aggregate the high-order information and key information (Ref. \cite{Kang-Li-Yao-Li-Jiang-Peng-Wu-Qi-Dong-2022}).

\vskip 0.2cm

In this article, the topological structure of finite sets and their subsets is also \emph{hypergraph}. Since hypergraphs are more difficult to draw on paper than graphs, we will use a visualization tool to study various hypergraphs including \emph{ordinary hypergraphs} (abbreviated as \emph{hypergraphs}) and \emph{non-ordinary hypergraphs}. We will find particular hyperedge sets, and reveal topological properties of hypergraphs, such as colorings, Hamiltonicity, connectivity and homomorphism, and so on. Our goal is to build up standard rules and feasible operating methods for applying hypergraphs to theoretical investigation and practices. In the final section, we construct non-ordinary hypergraphs by pan-operations for a wider range of applications, since hyperedge set matchings and hypergraph homomorphisms can be used to \emph{fully homomorphic encryption} (Ref. \cite{dai-zhang-jiang-binwu-deng-2009}).

\subsection{Preliminary}

\subsubsection{Basic terminology}

Each graph mentioned here has no multiple-edge and loop-edge. Standard concepts and notations of graph theory and hypergraphs used here can be found in \cite{Bondy-2008}, \cite{Claude-Berge-North-Holland-1989}, \cite{Gallian2022} and \cite{Jianfang-Wang-Hypergraphs-2008}. We introduce the following notation and terminology:

A $(p,q)$-graph is a graph of $p$ vertices and $q$ edges. Since the \emph{number} (also, \emph{cardinality}) of elements of a set $X$ is denoted as $|X|$, so the \emph{degree} $\textrm{deg}_G(x)$ of a vertex $x$ in a $(p,q)$-graph $G$ is equals to $\textrm{deg}_G(x)=|N_{ei}(x)|$, where $N_{ei}(x)$ is the set of neighbors of the vertex $x$ in the $(p,q)$-graph $G$. The short hand symbol $[m,n]$ denotes an integer set $\{m,m+1,\dots ,n\}$ subject to $0\leq m<n$. The symbol $Z^0$ is the set of non-negative integers, and the notation $Z$ stands for the set of integers. For integers $r,k\in Z^0$ and $s,d\in Z^0\setminus\{0\}$, we have two parameterized sets
\begin{equation}\label{eqa:two-parameterized-sets11}
{
\begin{split}S_{s,k,r,d}&=\{k+rd,k+(r+1)d,\dots ,k+(r+s)d\}\\
O_{s,k,r,d}&=\{k+[2(r+1)-1]d,k+[2(r+2)-1]d,\dots ,k+[2(r+s)-1]d\}
\end{split}}
\end{equation}
where two cardinalities $|S_{s,k,r,d}|=s+1$ and $|O_{s,k,r,d}|=s$.

The set of all subsets of a finite set $\Lambda$ is denoted as $\Lambda^2=\{X:~X\subseteq \Lambda\}$, called \emph{power set}, and each subset of $\Lambda^2$ is not empty, and the cardinality $|\Lambda^2|=2^{|\Lambda|}-1$. For a given set $\Lambda=\{a,b,c,d,e\}$, the power set $\Lambda^2$ has its own elements $\{a\}$, $\{b\}$, $\{c\}$, $\{d\}$, $\{e\}$, $\{a,b\}$, $\{a,c\}$, $\{a,d\}$, $\{a,e\}$, $\{b,c\}$, $\{b,d\}$, $\{b,e\}$, $\{c, d\}$, $\{c, e\}$, $\{d,e\}$, $\{a,b,c\}$, $\{a,b,d\}$, $\{a,b,e\}$, $\{a,c,d\}$, $\{a,c,e\}$, $\{a,d,e\}$, $\{b,c,d\}$, $\{b,c,e\}$, $\{b,d,e\}$, $\{c,d,e\}$, $\{a,b,c,d\}$, $\{a,b,c,e\}$, $\{a,c,d,e\}$, $\{a,b,d,e\}$, $\{b,c,d,e\}$ and $\Lambda$.

A \emph{$k$-set} is a set of $k$ elements, and a \emph{set-set} is a set having its own elements being sets.

\subsubsection{Graph operations}

\begin{defn} \label{defn:vertex-split-coinciding-operations}
\cite{Yao-Zhang-Sun-Mu-Sun-Wang-Wang-Ma-Su-Yang-Yang-Zhang-2018arXiv, Yao-Sun-Zhang-Mu-Wang-Jin-Xu-2018} \textbf{Vertex-splitting operation.} We vertex-split a vertex $u$ of a graph $G$ with degree $\textrm{deg}(u)\geq 2$ into two vertices $u\,'$ and $u\,''$, such that three neighbor sets $N_{ei}(u)=N_{ei}(u\,')\cup N_{ei}(u\,'')$ with $N_{ei}(u\,')\neq \emptyset$, $N_{ei}(u\,'')\neq \emptyset$ and $N_{ei}(u\,')\cap N_{ei}(u\,'')=\emptyset$, the resultant graph is denoted as $G\wedge u$, called \emph{vertex-split graph} (see an example shown in Fig.\ref{fig:v-splitting-opration} (a)$\rightarrow$(b)), as well as $|E(G\wedge u)|=|E(G)|$ and $|V(G\wedge u)|=1+|V(G)|$.

\textbf{Vertex-coinciding operation.} (Non-common neighbor vertex-coinciding operation) A vertex-coinciding operation is the \emph{inverse} of a vertex-splitting operation, and vice versa. If two vertices $u\,'$ and $u\,''$ of a graph $H$ holds $N_{ei}(u\,')\neq \emptyset$, $N_{ei}(u\,'')\neq \emptyset$ and $N_{ei}(u\,')\cap N_{ei}(u\,'')=\emptyset$ true, we vertex-coincide $u\,'$ and $u\,''$ into one vertex $u=u\,'\bullet u\,''$, such that three neighbor sets $N_{ei}(u)=N_{ei}(u\,')\cup N_{ei}(u\,'')$, the resultant graph is denoted as $H(u\,'\bullet u\,'')$, called \emph{vertex-coincided graph} (see a scheme shown in Fig.\ref{fig:v-splitting-opration} (b)$\rightarrow$(a)). Moreover, $|E(H(u\,'\bullet u\,''))|=|E(H)|$ and $|V(H(u\,'\bullet u\,''))|=|V(H)|-1$.\qqed
\end{defn}

\begin{figure}[h]
\centering
\includegraphics[width=14cm]{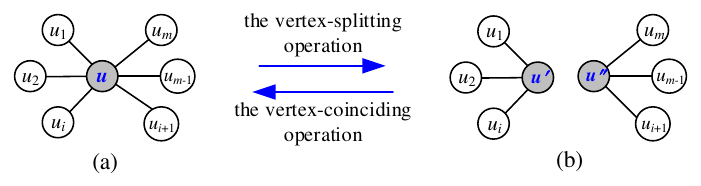}\\
\caption{\label{fig:v-splitting-opration}{\small A scheme for illustrating the vertex-splitting operation and the vertex-coinciding operation defined in Definition \ref{defn:vertex-split-coinciding-operations}.}}
\end{figure}

\begin{defn} \label{defn:split-v-set-vertex-split-graphss}
By the vertex-splitting operation introduced in Definition \ref{defn:vertex-split-coinciding-operations}, we vertex-split each vertex $w$ in a non-empty vertex subset $S$ of a graph $G$ into two vertices $w\,'$ and $w\,''$, such that the adjacent neighbor set $N_{ei}(w)=N_{ei}(w\,')\cup N_{ei}(w\,'')$ with $N_{ei}(w\,')\cap N_{ei}(w\,'')=\emptyset$, $|N_{ei}(w\,')|\geq 1$ and $|N_{ei}(w\,'')|\geq 1$, the resultant graph, called \emph{vertex-split graph}, is denoted as $G\wedge S$, and $V(G\wedge S)=V(G-S)\cup S\,'\cup S\,''$, where $S\,'=\{w\,':w\in S\}$ and $S\,''=\{w\,'':w\in S\}$, as well as $|E(G\wedge S)|=|E(G)|$. \qqed
\end{defn}

\subsubsection{Basic graph colorings}

\begin{defn}\label{defn:define-labelling-basic}
\cite{Yao-Wang-2106-15254v1, Bing-Yao-Cheng-Yao-Zhao2009} A \emph{labelling} $h$ of a graph $G$ is a mapping $h:S\rightarrow [a,b]$ for $S\subseteq V(G)\cup E(G)$ such that $h(x)\neq h(y)$ for any pair of elements $x,y$ of $S$, and write the label set $h(S)=\{h(x): x\in S\}$. The \emph{dual labelling} $h'$ of the labelling $h$ is defined as: $h'(z)=\max h(S)+\min h(S)-h(z)$ for $z\in S$. Moreover, $h(S)$ is called \emph{vertex label set} if $S=V(G)$, $h(S)$ \emph{edge label set} if $S=E(G)$, and $h(S)$ \emph{total label set} if $S=V(G)\cup E(G)$.\qqed
\end{defn}

\begin{defn} \label{defn:general-definition-set-colorings}
\cite{Yao-Ma-arXiv-2024-13354} Suppose that a graph $G$ admits a total set-coloring $\theta:V(G)\cup E(G)\rightarrow S_{et}$, where $S_{et}$ is the set of $m$ sets $e_{1},e_{2},\dots ,e_{m}$. There are the following various set-type colorings:
\begin{asparaenum}[(i) ]
\item If the vertex color set $\theta(V(G))=\emptyset$ and the edge color set $\theta(E(G))\neq \emptyset $, we call $\theta$ \emph{edge set-coloring} of the graph $G$.
\item If the edge color set $\theta(E(G))=\emptyset$ and the vertex color set $\theta(V(G))\neq \emptyset $, we call $\theta$ \emph{vertex set-coloring} of the graph $G$.
\item If the vertex color set $\theta(V(G))\neq \emptyset $ and the edge color set $\theta(E(G))\neq \emptyset $, we call $\theta$ \emph{total set-coloring} of the graph $G$.
\end{asparaenum}

Moreover, as $\theta$ is a total set-coloring of the graph $G$, we have:
\begin{asparaenum}[\textbf{Par}-1. ]
\item If there is a $W$-constraint equation $W[\theta(u),\theta(uv),\theta(v)]=0$ for each edge $uv\in E(G)$ holding true, we call $\theta$ \emph{$W$-constraint total set-coloring} of the graph $G$.
\item If the set $S_{et}=\{e_{1},e_{2},\dots ,e_{m}\}$ holds $e_{i}=\{x_i\}$ and $x_{i}\subset Z^0$ for $i\in [1,m]$, then $\theta$ is a \emph{popular $W$-constraint total coloring} as the $W$-constraint equation $W[\theta(u),\theta(uv),\theta(v)]=0$ for each edge $uv\in E(G)$ holding true (Ref. \cite{Gallian2022, Yao-Wang-2106-15254v1}).
\item If the set $S_{et}=\{e_{1},e_{2},\dots ,e_{m}\}$ is a hyperedge set $\mathcal{E}$ of a hypergraph $\mathcal{H}_{yper}=(\Lambda,\mathcal{E})$, we call $\theta$ \emph{hyperedge-set total coloring} of the graph $G$. Moreover, $\theta$ is a \emph{3I-hyperedge set total coloring} if each edge $uv\in E(G)$ holds $\theta(u)\cap \theta(v)\subseteq \theta(uv)$ with $\theta(u)\cap \theta(v)\neq \emptyset$.\qqed
\end{asparaenum}
\end{defn}

\begin{defn} \label{defn:more-string-total-coloring}
\cite{Yao-Ma-arXiv-2024-13354} Let $G$ be a bipartite $(p,q)$-graph, and its vertex set $V(G)=X\cup Y$ with $X\cap Y=\emptyset$ such that each edge $uv\in E(G)$ holds $u\in X$ and $v\in Y$. For parameters $k_s,d_s\in Z^0$, there is a group of $W$-constraint $(k_s,d_s)$-parameterized colorings
\begin{equation}\label{eqa:555555}
f_s:X\rightarrow S_{m,0,0,d}=\big \{0,d,\dots ,md\big \},~f_s:Y\cup E(G)\rightarrow S_{n,k,0,d}=\big \{k,k+d,\dots ,k+nd\big \}
\end{equation} for $s\in [1,B]$ with integer $B\geq 13$, here it is allowed $f_s(u)=f_s(w)$ for some distinct vertices $u,w\in V(G)$, such that the $W$-constraint $(k_s,d_s)$-coloring $f_s$ is one of \emph{gracefully} $(k_s,d_s)$-total coloring, \emph{odd-gracefully} $(k_s,d_s)$-total coloring, \emph{edge anti-magic} $(k_s,d_s)$-total coloring, \emph{harmonious} $(k_s,d_s)$-total coloring, \emph{odd-elegant} $(k_s,d_s)$-total coloring, \emph{edge-magic} $(k_s,d_s)$-total coloring, \emph{edge-difference} $(k_s,d_s)$-total coloring, \emph{felicitous-difference} $(k_s,d_s)$-total coloring, \emph{graceful-difference} $(k_s,d_s)$-total coloring, \emph{odd-edge edge-magic} $(k_s,d_s)$-total coloring, \emph{odd-edge edge-difference} $(k_s,d_s)$-total coloring, \emph{odd-edge felicitous-difference} $(k_s,d_s)$-total coloring, \emph{odd-edge graceful-difference} $(k_s,d_s)$-total coloring. We have:

\begin{equation}\label{eqa:three-parameterized-colorings}
\begin{tabular}{ll}
&$f_{i_1}(u)f_{i_2}(u)\cdots f_{i_B}(u)$\textrm{ is a permutation of }$f_1(u),f_2(u),\dots $, $f_B(u)$; \\[4pt]
&$f_{j_1}(uv)f_{j_2}(uv)\cdots f_{j_B}(uv)$\textrm{ is a permutation of }$f_1(uv),f_2(uv),\dots ,f_B(uv)$;\\[4pt]
&$f_{s_1}(v)f_{s_2}(v)$ $\cdots f_{s_B}(v)$\textrm{ is a permutation of }$f_1(v),f_2(v),\dots ,f_B(v)$.
\end{tabular}
\end{equation} for each edge $uv\in E(G)$.

(i) \textbf{Parameterized total string-coloring.} By (\ref{eqa:three-parameterized-colorings}), the bipartite $(p,q)$-graph $G$ admits a \emph{parameterized total string-coloring} $F$ defined as follows
\begin{equation}\label{eqa:para-total-string-coloring}
{
\begin{split}
F(u)&=f_{i_1}(u)f_{i_2}(u)\cdots f_{i_B}(u),\quad F(uv)=f_{j_1}(uv)f_{j_2}(uv)\cdots f_{j_B}(uv),\\
F(v)&=f_{s_1}(v)f_{s_2}(v)\cdots f_{s_B}(v)
\end{split}}
\end{equation} for each edge $uv\in E(G)$. Hence, there are $(B!)^3$ parameterized total string-colorings in total.

(ii) \textbf{Parameterized total vector-coloring.} By (\ref{eqa:three-parameterized-colorings}), the bipartite $(p,q)$-graph $G$ admits a \emph{parameterized total vector-coloring} $\alpha$ defined as follows
\begin{equation}\label{eqa:para-total-vector-coloring}
{
\begin{split}
\alpha(u)&=\big (f_{i_1}(u),f_{i_2}(u),\dots ,f_{i_B}(u)\big ),\quad \alpha(uv)=\big (f_{j_1}(uv),f_{j_2}(uv),\dots ,f_{j_B}(uv)\big ),\\
\alpha(v)&=\big (f_{s_1}(v),f_{s_2}(v),\dots ,f_{s_B}(v)\big )
\end{split}}
\end{equation} for each edge $uv\in E(G)$. Similarly with (i), there are $(B!)^3$ parameterized total vector-colorings, in total.

(iii) \textbf{Parameterized total set-coloring.} The bipartite $(p,q)$-graph $G$ admits a \emph{parameterized total set-coloring} $\theta$ defined as follows
\begin{equation}\label{eqa:555555}
{
\begin{split}
\theta(u)&=\big \{f_1(u),f_2(u),\dots ,f_B(u)\big \},\quad \theta(uv)=\big \{f_1(uv),f_2(uv),\dots ,f_B(uv)\big \},\\
\theta(v)&=\big \{f_1(v),f_2(v),\dots ,f_B(v)\big \}
\end{split}}
\end{equation} for each edge $uv\in E(G)$.

(iv) \textbf{Parameterized total lattice-coloring.} By (\ref{eqa:three-parameterized-colorings}), the bipartite $(p,q)$-graph $G$ admits a \emph{parameterized total lattice-coloring} $\varphi$ defined as follows
\begin{equation}\label{eqa:para-total-lattice-coloring}
\varphi(u)=\sum ^B_{k=1}a_kf_{i_k}(u),\quad \varphi(uv)=\sum ^B_{k=1}a_kf_{j_k}(uv),\quad \varphi(v)=\sum ^B_{k=1}a_kf_{s_k}(v)
\end{equation} with $\sum ^B_{k=1}a_k\geq 1$ and $a_k\in Z^0$ for each edge $uv\in E(G)$. So, there are $(B!)^3$ parameterized total lattice-colorings, in total.\qqed
\end{defn}

\begin{rem}\label{rem:333333}
In Definition \ref{defn:more-string-total-coloring}, if $i_k=j_k=s_k$ for $k\in [1,B]$ in (\ref{eqa:para-total-string-coloring}), (\ref{eqa:para-total-vector-coloring}) and (\ref{eqa:para-total-lattice-coloring}), then three colorings $F$, $\alpha$ and $\varphi$ are \emph{uniform}.\qqed
\end{rem}

\begin{defn}\label{defn:total-coloring-Topcode-matrixs}
Let $E(G)=\{e_i=x_iy_i:~i\in [1,q]\}$ be the edge set of a $(p,q)$-graph $G$, and let $\theta:V(G)\cup E(G)\rightarrow S_{pan}$ be a \emph{total pan-coloring} subject to a constraint set $R_{est}(c_1,c_2,\dots, c_m)$ with $m\geq 1$, where $S_{pan}$ is a \emph{pan-set}. Then we call the following matrix
\begin{equation}\label{eqa:topcode-matrix-total-coloring}
{
\begin{split}
T_{code}(G,f)= \left(
\begin{array}{cccccccccc}
\theta(x_1) & \theta(x_2) & \cdots & \theta(x_q)\\
\theta(e_1) & \theta(e_2) & \cdots & \theta(e_q)\\
\theta(y_1) & \theta(y_2) & \cdots & \theta(y_q)
\end{array}
\right)_{3\times q}=(\theta(X),\theta(E),\theta(Y))^T_{3\times q}
\end{split}}
\end{equation} \emph{Topcode-matrix}, where \emph{v-vector} $\theta(X)=(\theta(x_1)$, $\theta(x_2)$, $\dots$, $\theta(x_q))$, \emph{e-vector} $\theta(E)=(\theta(e_1),\theta(e_2)$, $\dots $, $\theta(e_q))$ and \emph{v-vector} $\theta(Y)=(\theta(y_1),\theta(y_2),\dots $, $\theta(y_q))$, such that each constraint $c_i$ of the constraint set $R_{est}(c_1,c_2,\dots, c_m)$ holds $c_i[\theta(x_j),\theta(e_j),\theta(y_j)]=0$ true.\qqed
\end{defn}

\begin{rem}\label{rem:333333}
About Definition \ref{defn:total-coloring-Topcode-matrixs}, we have:

(i) By Definition \ref{defn:more-string-total-coloring}, the total pan-coloring of Definition \ref{defn:total-coloring-Topcode-matrixs} exists.

(ii) The pan-set $S_{pan}$ is a set of sets being string-sets, vector-sets, matrix-sets, graph-sets, \emph{etc}.

(iii) We can require that some one constraint $c_i\in R_{est}(c_1,c_2,\dots, c_m)$ holds $c_i[\theta(x_j),\theta(e_j),\theta(y_j)]=0$ true, or several constraints, not required for each constraint of $R_{est}(c_1$, $c_2$, $\dots$, $c_m)$.\qqed
\end{rem}

\section{Hypergraphs}

\subsection{Concepts of hypergraphs}

\begin{defn}\label{defn:hypergraph-basic-definition}
\cite{Jianfang-Wang-Hypergraphs-2008} A \emph{hyperedge set} $\mathcal{E}$ is a family of distinct subsets $e_1,e_2$, $\dots $, $e_n$ of the power set $\Lambda^2$ based on a finite set $\Lambda=\{x_1,x_2,\dots ,x_n\}$, and satisfies:

(i) Each subset $e\in \mathcal{E}$, called \emph{hyperedge}, holds $e\neq \emptyset $ true;

(ii) $\Lambda=\bigcup _{e\in \mathcal{E}}e$, where each element $x_i$ of $\Lambda$ is called a \emph{vertex} (or \emph{hypervertex}).

The symbol $\mathcal{H}_{yper}^{gen}=(\Lambda,\mathcal{E})$ stands for a \emph{general hypergraph} with its own \emph{hyperedge set} $\mathcal{E}$ defined on the \emph{vertex set} $\Lambda$, and the cardinality $|\mathcal{E}|$ is \emph{size}, and the cardinality $|\Lambda|$ is \emph{order} of the general hypergraph $\mathcal{H}_{yper}=(\Lambda,\mathcal{E})$. \qqed
\end{defn}

\begin{example}\label{exa:8888888888}
Fig.\ref{fig:1-example-hypergraph} shows a hypergraph $\mathcal{H}_{yper}=(\Lambda,\mathcal{E})$ with its own vertex set $\Lambda=[1,15]$ and its own hyperedge set $\mathcal{E}=\{e_1,e_2,e_3,e_4\}$.\qqed
\end{example}

\begin{figure}[h]
\centering
\includegraphics[width=14cm]{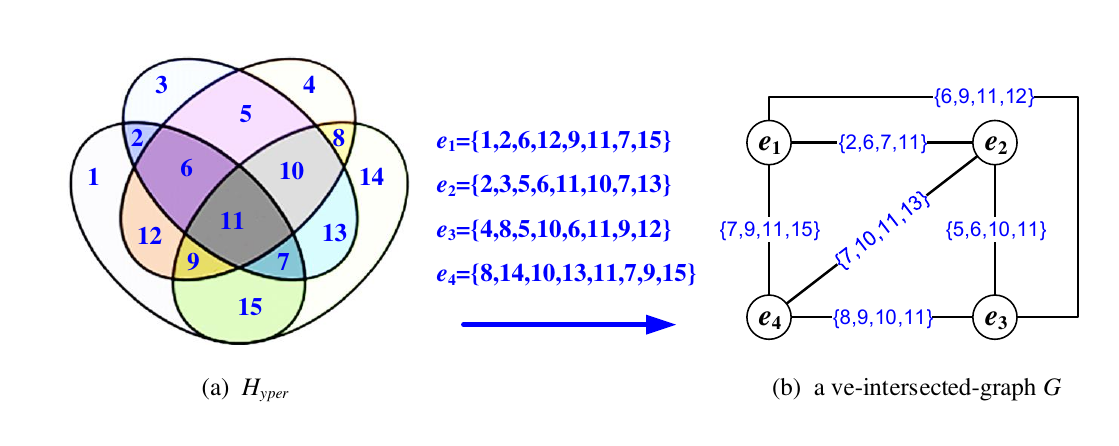}\\
\caption{\label{fig:1-example-hypergraph}{\small An example from an $8$-uniform hypergraph $H_{yper}$ to a ve-intersected graph $G$ defined in Definition \ref{defn:vertex-intersected-graph-hypergraph} admitting a set-coloring, where (a) Venn's four-set diagram using four ellipses.}}
\end{figure}

\begin{defn}\label{defn:more-conceptd-hypergraphs}
For Definition \ref{defn:hypergraph-basic-definition} and a finite set $\Lambda=\{x_1,x_2,\dots ,x_n\}$, there are the following terminology and notation for hypergraphs:
\begin{asparaenum}[\textrm{\textbf{Con}}-1.]
\item A hyperedge set $\mathcal{E}$ is \emph{proper} if any hyperedge $e\in \mathcal{E}$ holds $e\not \subseteq e\,'\in \mathcal{E}\setminus e$ true.
\item A hyperedge set $\mathcal{E}$ is called \emph{intersected} if each hyperedge $e\in \mathcal{E}$ corresponds to another hyperedge $e\,'\in \mathcal{E}\setminus e$ holding $e\cap e\,'\neq \emptyset $.
\item A hyperedge set $\mathcal{E}$ has its own \emph{complementary hyperedge set} $\overline{\mathcal{E}}$ defined by $\overline{\mathcal{E}}=\{e\,'=\Lambda \setminus e: ~e\in \mathcal{E}\}$ as the hyperedge set $\mathcal{E}$ is proper. If each hyperedge $e\in \overline{\mathcal{E}}$ corresponds to another hyperedge $e\,'=\Lambda \setminus e$, thus $\mathcal{E}=\overline{\mathcal{E}}$, we say $\mathcal{E}$ to be \emph{self-complementary}.
\item If a hyperedge set $\mathcal{E}$ can be expressed as $\mathcal{E}=S_1\cup S_2$ with $S_1\cap S_2=\emptyset$, such that each subset $e_{i,j}\in S_i$ corresponds to another subset $e_{3-i,k}\in S_{3-i}$ for $i=1,2$ holding $e_{i,j}\cap e_{3-i,k}\neq \emptyset$, and any pair of two subsets $e_{i,j},e_{i,s}\in S_i$ for $i=1,2$ holds $e_{i,j}\cap e_{i,s}=\emptyset$. Then we call the hyperedge set $\mathcal{E}$ to be \emph{bipartite hyperedge set}.

\item \cite{Jianfang-Wang-Hypergraphs-2008} An \emph{ear} $e\in \mathcal{E}$ holds:

\qquad (i) $e\cap e\,'=\emptyset$ for any hyperedge $e\,'\in \mathcal{E}\setminus e$;

\qquad (ii) there exists another hyperedge $e^*\in \mathcal{E}$, such that each vertex of $e\setminus e^*$ is not in any element of $\mathcal{E}\setminus e$.

\item An \emph{isolated vertex} $x\in \Lambda$ is in a unique hyperedge $e_i\in \mathcal{E}$, also $x\not\in e_j\in \mathcal{E}$ for $j\neq i$.
\item If each \emph{hyperedge} $e\in \mathcal{E}$ has its own cardinality $|e|=k$, then we call $\mathcal{H}_{yper}=(\Lambda,\mathcal{E})$ \emph{$k$-uniform hypergraph} (or $k$-graph used in some articles). So, ordinary graphs are 2-uniform hypergraphs.
\item \cite{Jianfang-Wang-Hypergraphs-2008} The \emph{Graham reduction} of a hyperedge set $\mathcal{E}$ is obtained by doing repeatedly

\qquad GR-1: delete a vertex $x$ if $x$ is an isolated vertex;

\qquad GR-2: delete $e_i$ if $e_i\subseteq e_j$ for $i\neq j$.
\item The \emph{hyperedge norm} $||\mathcal{E}||$ of a proper hyperedge set $\mathcal{E}$ is $||\mathcal{E}||=\sum_{e\in \mathcal{E}} |e|$ with
\begin{equation}\label{eqa:555555}
n=\min \big \{||\mathcal{E}||: \mathcal{E}\subset \Lambda^2\big \},\quad \frac{n(n-1)\cdots (n-k+1)}{k!}\leq \max \big \{||\mathcal{E}||: \mathcal{E}\subset \Lambda^2\big )\big \}
\end{equation}
\item For each vertex $x_j\in \Lambda$ with $j\in [1,n]$, the number of $x_j$ appeared in the hyperedges $e_{i,1}$, $e_{i,2}$, $\dots $, $e_{i,b_j}$ of a hyperedge set $\mathcal{E}_i$ is denoted as $b_j=\textrm{deg}_{\mathcal{E}_i}(x_j)$, called \emph{hypervertex degree}.\qqed
\end{asparaenum}
\end{defn}

\begin{rem}\label{rem:important-terminology-hypergraphs}
We, about Definition \ref{defn:hypergraph-basic-definition}, can consider other cases: each $x_i\in \Lambda$ is a graph, or a matrix, or a string, or a vector, or a set, or a hypergraph, or any thing in the world.

In practical application, a general hypergraph $\mathcal{H}_{yper}=(\Lambda,\mathcal{E})$ is a network, and each hyperedge $e_i\in \mathcal{E}$ can be seen as a community, or a local network \emph{etc}.\qqed
\end{rem}

\begin{defn} \label{defn:111111}
For a finite set $\Lambda=\{x_1,x_2,\dots ,x_n\}$, a $k$-uniform hypergraph $\mathcal{H}_{yper}=(\Lambda,\mathcal{E})$ is called \emph{self-complementary} if there is a permutation $\sigma:\Lambda\rightarrow \Lambda$, called \emph{self-complementing}, such that for every $k$-subset $e$ of $\Lambda$, each hyperedge $e \in \mathcal{E}$ if, and only if $\sigma(e)\not \in \mathcal{E}$. The hypergraph $\mathcal{H}_{yper}$ is isomorphic with the hypergraph $\mathcal{H}'_{yper}=\Big (\Lambda$, $\Lambda\choose k$$-\mathcal{E}\Big )$, where $\Lambda\choose k$ denotes the set of $k$-subsets of the power set $\Lambda^2$ and $\mathcal{E}\subset$$ \Lambda\choose k$.\qqed
\end{defn}

\subsection{Hypergraph sets}

\subsubsection{Intersected hyperedge sets}

\begin{defn}\label{defn:new-hypergraphs-sets}
\textbf{3I-hyperedge set and hypergraph set.} A hyperedge set $\mathcal{E}$ is a set of subsets of the power set $\Lambda^2$ based on a finite set $\Lambda=\{x_1,x_2,\dots ,x_n\}$ (Ref. Definition \ref{defn:hypergraph-basic-definition} and Definition \ref{defn:more-conceptd-hypergraphs}), if $\mathcal{E}$ holds each one of the following constraints:

(i) \textbf{Independence.} Any pair of distinct hyperedges $e,e\,'\in \mathcal{E}$ holds $e\not \subset e\,'$ and $e\,'\not \subset e$;

(ii) \textbf{Intersection.} Each hyperedge $e\in \mathcal{E}$ corresponds to another hyperedge $e\,'\in \mathcal{E}\setminus e$ holding $e\cap e\,'\neq \emptyset $ true; and

(iii) \textbf{Integrity.} $\Lambda=\bigcup _{e\in \mathcal{E}}e$.

Then we call $\mathcal{E}$ \emph{$3$I-hyperedge set} and the set $\mathcal{E}\big (\Lambda^2\big )=\{\mathcal{E}_i:~i\in [1,n(\Lambda)]\}$ containing all 3I-hyperedge sets of the power set $\Lambda^2$ \emph{hypergraph set}, where $n(\Lambda)=\big |\mathcal{E}\big (\Lambda^2\big )\big |$.\qqed
\end{defn}

\begin{rem}\label{rem:333333}
(i) The set $\mathcal{E}_{gen}\big (\Lambda^2\big )$ contains each hyperedge set $\mathcal{E}$ of the power set $\Lambda^2$ that does not satisfy each of Independence, Intersection and Integrity shown in Definition \ref{defn:new-hypergraphs-sets}.

(ii) We modify the ``Intersection'' in Definition \ref{defn:new-hypergraphs-sets} as ``Operation'', that is: Each hyperedge $e\in \mathcal{E}$ corresponds to another hyperedge $e\,'\in \mathcal{E}\setminus e$ holding $e[\bullet ]e\,'\neq \emptyset$ under the operation ``$[\bullet]$''. Then we call $\mathcal{E}$ \emph{IOI-hyperedge set} of $\mathcal{E}\big (\Lambda^2\big )$.\qqed
\end{rem}

\begin{defn}\label{defn:full-intersected-hyperedge-sets}
Notice that the \emph{hypervertex degree} $\textrm{deg}_{\mathcal{E}}(x_i)$ of each vertex $x_j\in \Lambda$ with $j\in [1,n]$ is the number $b_j$ of $x_j$ appeared in the hyperedges $e_{i_1}$, $e_{i_2}$, $\dots $, $e_{i_{b_j}}$ of a hyperedge set $\mathcal{E}$, denoted as $b_j=\textrm{deg}_{\mathcal{E}}(x_i)$.

Then, $\textbf{\textrm{d}}_{\mathcal{E}}=\big (\textrm{deg}_{\mathcal{E}}(x_1),\textrm{deg}_{\mathcal{E}}(x_2),\dots ,\textrm{deg}_{\mathcal{E}}(x_n)\big )$ with $\textrm{deg}_{\mathcal{E}}(x_i)\leq \textrm{deg}_{\mathcal{E}}(x_{i+1})$ for $i\in [1,n-1]$ is called \emph{hypervertex degree series} based on the hyperedge set $\mathcal{E}$. And moreover, $\mathcal{E}$ is a \emph{full hyperedge set} if each $\textrm{deg}_{\mathcal{E}}(x_i)\geq 1$ for $i\in [1,n]$.\qqed
\end{defn}

\begin{rem}\label{rem:333333}
There are $n(\Lambda)$ hypergraphs $\mathcal{H}_{yper}=(\Lambda,\mathcal{E}_i)$ for each 3I-hyperedge set $\mathcal{E}_i\in \mathcal{E}\big (\Lambda^2\big )$ with $i\in [1,n(\Lambda)]$ by Definition \ref{defn:new-hypergraphs-sets}. A \emph{simple hypergraph} $\mathcal{H}_{yper}=(\Lambda,\mathcal{E})$ based on a 3I-hyperedge set $\mathcal{E}\in \mathcal{E}\big (\Lambda^2\big )$ holds that any vertex $x\in \Lambda$ belongs to two or more subsets of the 3I-hyperedge set of the hypergraph set $\mathcal{E}\big (\Lambda^2\big )$. Also, each 3I-hyperedge set $\mathcal{E}\in \mathcal{E}\big (\Lambda^2\big )$ is a full hyperedge set.

If a hypergraph $\mathcal{H}_{yper}=(\Lambda,\mathcal{E})$ does not have $\mathcal{E}\in \mathcal{E}\big (\Lambda^2\big )$, we call it \emph{general hypergraph}, and denote it as $\mathcal{H}_{yper}^{gen}=(\Lambda,\mathcal{E})$ (Ref. Definition \ref{defn:hypergraph-basic-definition}), hereafter.\qqed
\end{rem}

\begin{problem}\label{question:444444}
About Definition \ref{defn:new-hypergraphs-sets} and Definition \ref{defn:full-intersected-hyperedge-sets}, we propose the following questions:
\begin{asparaenum}[\textbf{Hyper}-1]
\item \textbf{Characterize} the hypergraph set $\mathcal{E}\big (\Lambda^2\big )$, and \textbf{compute} the number $n(\Lambda)$, since no report is for computing the number $n(\Lambda)$.
\item \textbf{Characterize} each 3I-hyperedge set $\mathcal{E}^*$ with $|\mathcal{E}^*|=\max \{|\mathcal{E}|:\mathcal{E}\in \mathcal{E}\big (\Lambda^2\big )\}$.
\item \textbf{Is} there a proper hyperedge set $\mathcal{E}^*\in \mathcal{E}\big (\Lambda^2\big )$ such that the hyperedge norm $||\mathcal{E}^*||=m$ for each integer $m$ subject to $n<m<n(n-1)$? \textbf{Find} connections between the elements of the hypergraph set $\mathcal{E}\big (\Lambda^2\big )$.
\item \textbf{What} connections are there between two hypergraphs $\mathcal{H}_{yper}=(\Lambda,\mathcal{E})$ and $\mathcal{H}\,'_{yper}=(\Lambda,\mathcal{E}\,')$, as $\mathcal{E}\neq \mathcal{E}\,'$ based on the same vertex set $\Lambda$.
\item \textbf{Is} there a connection between two hypergraphs $\mathcal{H}_{yper}=(\Lambda,\mathcal{E})$ and $\mathcal{H}^*_{yper}=(\Lambda^*,\mathcal{E}^*)$, as $\Lambda\neq \Lambda^*$, $\Lambda\not\subset \Lambda^*$ and $\Lambda^*\not\subset \Lambda$?
\item Since the hypergraph set $\mathcal{E}\big (\Lambda^2\big )=\mathcal{E}\big (\Lambda^2_{nofull}\big )\cup \mathcal{E}\big (\Lambda^2_{full}\big )$ with $\mathcal{E}\big (\Lambda^2_{nofull}\big )\cap \mathcal{E}\big (\Lambda^2_{full}\big )=\emptyset$, where $\mathcal{E}\big (\Lambda^2_{full})$ contains every full-3I-hyperedge set (Ref. Definition \ref{defn:full-intersected-hyperedge-sets}), but $\mathcal{E}\big (\Lambda^2_{nofull})$ does not include any of the full-3I-hyperedge sets of the hypergraph set $\mathcal{E}\big (\Lambda^2\big )$. \textbf{Determine} the set $\mathcal{E}\big (\Lambda^2_{full})$, and \textbf{find} matchings $(\mathcal{E},\overline{\mathcal{E}})$ of full-3I-hyperedge sets $\mathcal{E}$ and their complementary full-3I-hyperedge sets $\overline{\mathcal{E}}$ in $\mathcal{E}\big (\Lambda^2_{full})$.
\end{asparaenum}
\end{problem}

\begin{problem}\label{question:444444}
\textbf{Find} a finite set $\Lambda_{*}$ for which a graph $G$ admits a total set-coloring $F:V(G)\cup E(G)\rightarrow \mathcal{E}^{*}\in \mathcal{E}(\Lambda^2_{*})$ with $\Lambda_{*}=\bigcup_{e\in \mathcal{E}^{*}}e$, such that the graph $G$ admits any total set-coloring $f:V(G)\cup E(G)\rightarrow \mathcal{E}\in \mathcal{E}\big (\Lambda^2\big )$ with $\Lambda=\bigcup_{s\in \mathcal{E}}s$ and $|\Lambda_{*}|\leq |\Lambda|$, as well as one or more of the following constraints:
\begin{asparaenum}[(i) ]
\item $F(uv)\supseteq F(u)\cap F(v)\neq \emptyset$ for each edge $uv\in E(G)$.
\item $f(xy)\supseteq f(x)\cap f(y)\neq \emptyset$ for each edge $xy\in E(G)$.
\end{asparaenum}
\end{problem}

We can observe the following results:

\begin{prop}\label{proposition:99999}
For a finite set $\Lambda=\{x_1,x_2,\dots ,x_m\}$, we have

(i) There are particular hyperedge sets $\mathcal{E}^*_0=\Lambda$, $\mathcal{E}^*_1=\big \{\{x_1\}$, $\{x_2\}$, $\dots $, $\{x_m\}\big \}\subset \Lambda^2$ and $\mathcal{E}_r=\big \{\{x_r\}\big \}$ with $r\in [1,m]$, such that the 3I-hyperedge set $\mathcal{E}_i\in \mathcal{E}\big (\Lambda^2\big )\setminus \big \{\mathcal{E}^*_0,\mathcal{E}^*_1,\mathcal{E}_1, \mathcal{E}_2, \dots, \mathcal{E}_m\big \}$ has its own cardinality holding $2=|\mathcal{E}_0|+1\leq |\mathcal{E}_i|\leq 2^m-3-m$.

(ii) Each hyperedge $e\in \Lambda^2$ with $2\leq |e|\leq m-1$ is in some 3I-hyperedge set $\mathcal{E}\in \mathcal{E}\big (\Lambda^2\big )$. Otherwise, we have a new 3I-hyperedge set $\mathcal{E}^*=\mathcal{E}\cup \{e\}$, such that $\mathcal{E}^*\not\in \mathcal{E}\big (\Lambda^2\big )$, a contradiction with the definition of the hypergraph set $\mathcal{E}\big (\Lambda^2\big )$. An example is $\mathcal{E}=\{e_i:|e_i|=|e|\}$ with $e\not \in \mathcal{E}$.
\end{prop}

\subsubsection{Hyperedge set matchings and fixed-point hyperedge sets}

\begin{defn} \label{defn:fixed-point-closed}
For a finite set $X=\{t_1,t_2,\dots ,t_m\}$, if there is a transforation $\varphi$ such that $\varphi(t_i)\in X$ for $i\in [1,m]$, i.e. $\varphi(X)= X$, then we say the set $X$ to be a \emph{fixed-point set} under the transforation $\varphi$, also, we say that the set $X$ is \emph{closed} to the transforation $\varphi$. \qqed
\end{defn}

\begin{prop}\label{prop:key-matching-equitable-uniform}
A \emph{hyperedge set matching} $(\mathcal{E},\overline{\mathcal{E}})$ holds $\overline{\mathcal{E}}=\{e\,'=\Lambda \setminus e: ~e\in \mathcal{E}\}$, where the 3I-hyperedge set $\mathcal{E}\in \mathcal{E}\big (\Lambda^2\big )$ and its own complementary hyperedge set $\overline{\mathcal{E}}\in \mathcal{E}\big (\Lambda^2\big )$ defined in Definition \ref{defn:more-conceptd-hypergraphs} and Definition \ref{defn:new-hypergraphs-sets}.

(i) If the 3I-hyperedge set $\mathcal{E}\in \mathcal{E}\big (\Lambda^2\big )$ is \emph{$k$-uniform}, namely, $|e_i|=|e_j|=k$ for any pair of sets $e_i$ and $e_j$ of the 3I-hyperedge set $\mathcal{E}$ (as a \emph{private-key}), then the complementary hyperedge set $\overline{\mathcal{E}}$ (as a \emph{public-key}) is $k^*$-uniform too, where $k^*=|\Lambda|-k$.

(ii) If the 3I-hyperedge set $\mathcal{E}\in \mathcal{E}\big (\Lambda^2\big )$ is \emph{equitable}, namely, $\big ||e_i|-|e_j|\big |\leq 1$ for any pair of sets $e_i$ and $e_j$ of the 3I-hyperedge set $\mathcal{E}$ (as a \emph{private-key}), then the complementary hyperedge set $\overline{\mathcal{E}}$ (as a \emph{public-key}) is equitable too.

(iii) If the elements of the 3I-hyperedge set $\mathcal{E}$ (as a \emph{private-key}) is inequality from each other, namely, $|e_i|\neq |e_j|$ as $i\neq j$, then the elements of the complementary hyperedge set $\overline{\mathcal{E}}$ (as a \emph{public-key}) is inequality from each other too, since $e\,'_i=\Lambda \setminus e_i$ and $e\,'_j=\Lambda \setminus e_j$, as well as $|e\,'_i|=|\Lambda | -|e_i|$ and $|e\,'_j|=|\Lambda |-|e_j|$.\qqed
\end{prop}

\begin{problem}\label{question:444444}
Notice:

(i) There is no a 3I-hyperedge set $\mathcal{E}\in \mathcal{E}\big (\Lambda^2\big )$ with $|\mathcal{E}|\geq 2$ such that $\Lambda\in \mathcal{E}$.

(ii) A 3I-hyperedge set $\mathcal{E}=\big \{\{1,M\},\{2,M\},\dots ,\{M-1,M\},[1,M-1]\big \}$ belongs to $\mathcal{E}\big ([1,M]^2\big )$, but its own complementary hyperedge set $\overline{\mathcal{E}}=\big \{[1,M]\setminus \{1,M\},[1,M]\setminus \{2,M\},\dots ,[1,M]\setminus \{M-1,M\},\{M\}\big \}\not \in \mathcal{E}\big ([1,M]^2\big )$, since there is no hyperedge $e\in \overline{\mathcal{E}}\setminus \{M\}$ holding $e\cap \{M\}\neq \emptyset$ true.

For a hypergraph set $\mathcal{E}\big (\Lambda^2\big )$, we propose the following questions:
\begin{asparaenum}[(1) ]
\item \textbf{Find} all matchings $(\mathcal{E},\overline{\mathcal{E}})$ of 3I-hyperedge sets $\mathcal{E}$ (as \emph{private-keys}) and their own complementary hyperedge sets $\overline{\mathcal{E}}$ (as \emph{public-keys}) in the hypergraph set $\mathcal{E}\big (\Lambda^2\big )$.
\item \textbf{Find} all matchings $(\mathcal{E},\overline{\mathcal{E}})$, such that the 3I-hyperedge set $\mathcal{E}$ and its own complementary hyperedge set $\overline{\mathcal{E}}$ are \emph{equitable} (Ref. Proposition \ref{prop:key-matching-equitable-uniform}).
\item \textbf{Find} all matchings $(\mathcal{E},\overline{\mathcal{E}})$, such that each 3I-hyperedge set $\mathcal{E}$ is $k$-uniform, and its own complementary hyperedge set $\overline{\mathcal{E}}$ is $k^*$-uniform with $k^*=|\Lambda|-k$.
\item If there is a transformation $\phi$ holding $\phi(\mathcal{E})=\mathcal{E}$, we call $\mathcal{E}$ \emph{fixed-point 3I-hyperedge set} under the transformation $\phi$, for example, the inverse transformation $\overline{\mathcal{E}}=\mathcal{E}$. \textbf{Find} transformations $\phi$ holding $\phi(\mathcal{E})=\mathcal{E}$ true.
\item \textbf{Find} all fixed-point 3I-hyperedge sets for a hypergraph set $\mathcal{E}\big ([1,M]^2\big )$.
\end{asparaenum}
\end{problem}

\begin{prop}\label{proposition:99999}
If a 3I-hyperedge set $\mathcal{E}=\{e_i: i\in [1,m]\}\in \mathcal{E}\big ([1,M]^2\big )$ holds $[1,M]\setminus e_i$ to be in $\mathcal{E}$ for each subset $e_i\in \mathcal{E}$, then two 3I-hyperedge sets $\overline{\mathcal{E}}$ and $\mathcal{E}$ are the fixed points of the hypergraph set $\mathcal{E}\big ([1,M]^2\big )$.

Since $\overline{\mathcal{E}}=\{\overline{e}_i=\Lambda\setminus e_i:~i\in [1,m],~e_i\in \mathcal{E}\}$ with $m\in \big [2,2^M-3-m\big ]$, then there are $(2^M-i)$ hyperedge $\overline{e}_i$ in total, so we have at least $\prod^m_{i=1}\big (2^M-i\big )$ matchings $(\mathcal{E},\overline{\mathcal{E}})$, in which the matchings with mutual isomorphism were not excluded.
\end{prop}

\begin{prop}\label{proposition:99999}
For a finite set $\Lambda =\{x_1,x_2,\dots ,x_m\}$, we have:

(i) For two 3I-hyperedge sets $\mathcal{E},\overline{\mathcal{E}}\in \mathcal{E}\big (\Lambda^2\big )$, then the hyperedge set $\mathcal{E}^*=\mathcal{E}\cup \overline{\mathcal{E}}\in \mathcal{E}\big (\Lambda^2\big )$ is a fixed-point 3I-hyperedge set, since $\overline{\mathcal{E}}^*=\mathcal{E}^*$.

(ii) Suppose that each 3I-hyperedge set $\mathcal{E}_i\in \mathcal{E}\big (\Lambda^2\big )$, but its own complementary hyperedge set $\overline{\mathcal{E}}_i\not \in \mathcal{E}\big (\Lambda^2\big )$ for $i=1,2$. There are hyperedge sets $\mathcal{E}_1\cup \overline{\mathcal{E}}_2=\{e_{1,i}\cup \overline{e}_{2,i}:~e_{1,i}\in \mathcal{E}_1,\overline{e}_{2,i}\in \overline{\mathcal{E}}_2,i\in [1,n]\}$ and $\overline{\mathcal{E}}_1\cup \mathcal{E}_2=\{\overline{e}_{1,i}\cup e_{2,i}:~\overline{e}_{1,i}\in \overline{\mathcal{E}}_1,e_{2,i}\in \mathcal{E}_2,i\in [1,n]\}$. We claim that the 3I-hyperedge set $\mathcal{E}^0=(\mathcal{E}_1\cup \overline{\mathcal{E}}_2)\cup (\overline{\mathcal{E}}_1\cup \mathcal{E}_2)$ is a fixed-point 3I-hyperedge set, since $\overline{\mathcal{E}}^0=\mathcal{E}^0$ by means of $\Lambda\setminus (e_{1,i}\cup \overline{e}_{2,i})=\overline{e}_{1,i}\cup e_{2,i}$ for $i\in [1,n]$.
\end{prop}

\begin{example}\label{exa:finite-module-abelian-additive-operation}
The consecutive integer set $\Lambda=[1,10]$ produces the following two 3I-hyperedge sets
\begin{equation}\label{eqa:555555}
{
\begin{split}
\mathcal{E}_{1,1}=&\big \{\{1,2\}, \{3,4\}, \{5,6\}, [1,4]\cup [7,10], \{1,2\}\cup [5,10], [8,10]\big \}\\
\mathcal{E}_{2,1}=&\big \{\{1,6,7\}, \{2,4,5,6,10\}, \{1,3,5,7,9\}, \{2,4,6,8,10\}, \{1,3,7,8,9\}, \{2,3,4,5,8,9,10\}\big \}
\end{split}}
\end{equation}such that each 3I-hyperedge set $\mathcal{E}_{i,1}\in \mathcal{E}\big ([1,10]^2\big )$ and holds the \emph{fixed-point transformation} $\overline{\mathcal{E}}_{i,1}=\mathcal{E}_{i,1}$ for $i=1,2$. According to Definition \ref{defn:set-set-additive-groups} under $(\bmod~10)$, we have

$\mathcal{E}_{1,1}=\big \{\{1,2\}, \{3,4\}, \{5,6\}, [1,4]\cup [7,10], \{1,2\}\cup [5,10], [8,10]\big \}$

$\mathcal{E}_{1,2}=\big \{\{2,3\}, \{4,5\}, \{6,7\}, [2,5]\cup [8,10]\cup \{1\}, \{2,3\}\cup [6,10]\cup \{1\}, [9,10]\cup \{1\}\big \}$

$\mathcal{E}_{1,3}=\big \{\{3,4\}, \{5,6\}, \{7,8\}, [3,6]\cup [9,10]\cup \{1,2\}, \{3,4\}\cup [7,10]\cup \{1,2\}, \{10\}\cup \{1,2\}\big \}$

$\mathcal{E}_{1,4}=\big \{\{4,5\}, \{6,7\}, \{8,9\}, [4,7]\cup \{10\}\cup \{1,2,3\}, \{4,5\}\cup [8,10]\cup \{1,2,3\}, \{1\}\cup \{2,3\}\big \}$

$\mathcal{E}_{1,5}=\big \{\{5,6\}, \{7,8\}, \{9,10\}, [1,8], \{5,6\}\cup [9,10]\cup [1,4], [2,4]\big \}$

$\mathcal{E}_{1,6}=\big \{\{6,7\}, \{8,9\}, \{10,1\}, [2,9], \{10\}\cup [1,7], [3,5]\big \}$

$\mathcal{E}_{1,7}=\big \{\{7,8\}, \{9,10\}, \{1,2\}, [3,10], [1,8], [4,6]\big \}$

$\mathcal{E}_{1,8}=\big \{\{8,9\}, \{10,1\}, \{2,3\}, [4,10]\cup \{1\}, [2,9], [5,7]\big \}$

$\mathcal{E}_{1,9}=\big \{\{9,10\}, \{1,2\}, \{3,4\}, [5,10]\cup \{1,2\}, [3,10], [6,8]\big \}$

$\mathcal{E}_{1,10}=\big \{\{10,1\}, \{2,3\}, \{4,5\}, [6,10]\cup [1,3], [4,10]\cup \{1\}, [7,9]\big \}$

$\mathcal{E}_{1,11}=\big \{\{1,2\}, \{3,4\}, \{5,6\}, [7,10]\cup [1,4], [5,10]\cup \{1,2\}, [8,10]\big \}=\mathcal{E}_{1,1}$

By the above 3I-hyperedge sets, it is not hard to verify
\begin{equation}\label{eqa:555555}
\mathcal{E}_{r,i}[+]\mathcal{E}_{r,j}[-]\mathcal{E}_{r,k}=\mathcal{E}_{r,\lambda},~\lambda=i+j-k~(\bmod~10),~r=1,2
\end{equation} for a preappointed zero $\mathcal{E}_{r,k}$. Thereby, we obtain four every-zero hypergraph groups $\big \{G(\mathcal{E}_{r,1});[+][-]\big \}$ $=\big \{G(\overline{\mathcal{E}}_{r,1});[+][-]\big \}$ with $r=1,2$, such that each every-zero hypergraph group is a \emph{fixed-point graphic group} under the finite module abelian additive operation.\qqed
\end{example}

\begin{example}\label{exa:8888888888}
\textbf{Four fixed-point graphic groups.} Every-zero graphic group defined in \cite{yao-sun-su-wang-matching-groups-zhao-2020}, Every-zero set-set additive group defined in Definition \ref{defn:set-set-additive-groups}, Every-zero set-colored graphic group defined in Definition \ref{defn:set-colored-graphic-group} and Every-zero hypergraph group defined in Definition \ref{defn:hypergraph-group-definition} are \emph{fixed-point graphic groups} under the \emph{finite module abelian additive operation}.\qqed
\end{example}

\subsubsection{Strong hyperedge sets}

\begin{defn} \label{defn:strong-proper-hyperedge-sets}
\cite{Yao-Ma-arXiv-2024-13354} Let each set-set $\mathcal{E}_r(m,m_r)=\{e_{r,1},e_{r,2},\dots ,e_{r,m_r}\}$ with $r\in [1,A_m]$ be generated from a consecutive integer set $\Lambda=[1,m]$ such that each hyperedge $e_{r,j}$ is a subset of the power set $[1,m]^2$ for $j\in [1,m_r]$ and $r\in [1,A_m]$.
\begin{asparaenum}[\textbf{\textrm{Sthyset}}-1]
\item A \emph{strong 3I-hyperedge set} $\mathcal{E}_r(m,m_r)$ satisfies: Any pair of hyperedges $e_{r,i}$ and $e_{r,j}$ with $i\neq j$ holds

\qquad (1.1) $e_{r,i}\cap e_{r,j}\neq \emptyset$; (1.2) $e_{r,i}\not \subset e_{r,j}$ and $e_{r,j}\not \subset e_{r,i}$.

\item A \emph{proper 3I-hyperedge set} $\mathcal{E}_r(m,m_r)$ with $r\in [1,A_m]$ satisfies:

\qquad (2.1) Any pair of hyperedges $e_{r,i}$ and $e_{r,j}$ holds $e_{r,i}\not \subset e_{r,j}$ and $e_{r,j}\not \subset e_{r,i}$ when $i\neq j$;

\qquad (2.2) Each hyperedge $e_{r,s}\in \mathcal{E}_r(m,m_r)$ corresponds to another hyperedge $e_{r,t}\in \mathcal{E}_r(m,m_r)$ holding $e_{r,s}\cap e_{r,t}\neq\emptyset$.

\item A \emph{perfect hypermatching} of a hypergraph $\mathcal{H}_{yper}=(\Lambda,\mathcal{E})$ is a collection of hyperedges $M_1,M_2$, $\dots$, $M_m$ $\subseteq \mathcal{E}$, such that $M_i\cap M_j=\emptyset $ for $i\neq j$ and $\bigcup ^m_{i=1}M_i=\Lambda=[1,m]$.
\item \cite{Jianfang-Wang-Hypergraphs-2008} If a hyperedge set $\mathcal{E}=\bigcup^m_{j=1} \mathcal{E}_j$ holds $\mathcal{E}_i\cap \mathcal{E}_j=\emptyset$ for $i\neq j$, and each hyperedge $e\in \mathcal{E}$ belongs to one $\mathcal{E}_j$ and $e\not\in \bigcup^m_{k=1,k\neq j} \mathcal{E}_k$, then $\{\mathcal{E}_1,\mathcal{E}_2,\dots ,\mathcal{E}_m\}$ is called \emph{decomposition} of the hyperedge set $\mathcal{E}$.
\item \cite{Jianfang-Wang-Hypergraphs-2008} A hyperedge set $\mathcal{E}$ is \emph{irreducible} if each hyperedge $e\in \mathcal{E}$ does not hold $e\subseteq e\,'$ for any hyperedge $e\,'\in \mathcal{E}\setminus e$.\qqed
\end{asparaenum}
\end{defn}

\begin{example}\label{exa:build-strong-hyperedge-sets}
Build up strong 3I-hyperedge sets $\mathcal{E}_r(m, m_r)=\{e_{r,1}, e_{r,2}, \dots , e_{r,m_r}\}$ from a consecutive integer set $[1, m]$, such that each set $e_{r,s}$ holds $|e_{r,s}|\geq r+1$ for $s\in [1, m_r]$ and $r\in [1,B_m]$, as well as each set-set $\mathcal{E}_r(m, m_r)$ holds the conditions of strong 3I-hyperedge set introduced in Definition \ref{defn:strong-proper-hyperedge-sets} true. We, next, construct particular strong 3I-hyperedge sets $\mathcal{E}_t(m, m_t)=\{e_{t,1}$, $ e_{t,2}$, $ \dots $, $e_{t,m_t}\}$ with $|e_{t,s}|=t+1$ for $s\in [1,m_t]$ and $1\leq t\leq m-1$, where
\begin{equation}\label{eqa:strong-hyperedge-set-formula00}
m_t=\sum^t_{k=1}{m-t \choose k}=|\mathcal{E}_t(m, n_t)|,~m-t> 0,~t\geq 2
\end{equation} with $m_1={m-1 \choose 1}+1=m$, and we have
\begin{equation}\label{eqa:strong-hyperedge-set-formula}
\begin{tabular}{llll}
&$\mathcal{E}_1(m, m_1)=\big \{\{1\}\cup \{a_i\}:~a_i\in [2,m]\big \}\bigcup [2,m]$;\\
&$\mathcal{E}_2(m, m_2)=\big \{\{p\}\cup \{a_{i_1},a_{i_2}\}:~a_{i_j}\in [p+1,m],j\in [1,2],~p\in [1,2]\big \}$;\\
&$\mathcal{E}_3(m, m_3)=\big \{\{p\}\cup \{a_{i_1},a_{i_2},a_{i_3}\}:~a_j\in [p+1,m],j\in [1,3],~p\in [1,3]\big \}$;\\
&$\cdots \cdots \cdots $\\
&$\mathcal{E}_t(m, m_t)=\big \{\{p\}\cup \{a_{i_1},a_{i_2},\dots ,a_{i_t}\}:~a_j\in [p+1,m],j\in [1,t],~p\in [1,t]\big \}$.\\
\end{tabular}
\end{equation}

According to Eq.(\ref{eqa:strong-hyperedge-set-formula00}) and Eq.(\ref{eqa:strong-hyperedge-set-formula}), we have the following particular strong 3I-hyperedge sets:

\textbf{Exa-1.} $m=4$. Since $4_1={3 \choose 1}+1=4$, so $\mathcal{E}_1(4, 4)=\big \{\{$1, 2$\}$, $\{$1, 3$\}$, $\{$1, 4$\}$, $\{$2, 3, 4$\}\big \}$.

Since $4_2={2 \choose 1}+{2 \choose 2}=3$, then $\mathcal{E}_2(4, 3)=\big \{\{$1, 2, 3$\}$, $\{$1, 2, 4$\}$, $\{$2, 3, 4$\}\big \}$.

\textbf{Exa-2.} $m=6$. Since $6_1={5 \choose 1}+1=6$, we get $\mathcal{E}_1(6, 6)=\big \{\{$1, 2$\}$, $\{$1, 3$\}$, $\{$1, 4$\}$, $\{$1, 5$\}$, $\{$1, 6$\}$, $\{$2, 3, 4, 5, 6$\}\big \}$.

Since $6_2={4 \choose 1}+{4 \choose 2}=10$, so $\mathcal{E}_2(6, 10)=\big \{\{$1, 2, 3$\}$, $\{$1, 2, 4$\}$, $\{$1, 2, 5$\}$, $\{$1, 2, 6$\}$, $\{$2, 3, 4$\}$, $\{$2, 3, 5$\}$, $\{$2, 3, 6$\}$, $\{$2, 4, 5$\}$, $\{$2, 4, 6$\}$, $\{$2, 5, 6$\}\big \}$.

Since $6_3={3 \choose 1}+{3 \choose 2}+{3 \choose 3}=7$, then $\mathcal{E}_3(6, 7)=\big \{\{$1, 2, 3, 4$\}$, $\{$1, 2, 3, 5$\}$, $\{$1, 2, 3, 6$\}$, $\{$2, 3, 4, 5$\}$, $\{$2, 3, 4, 6$\}$, $\{$2, 3, 5, 6$\}$, $\{$3, 4, 5, 6$\}\big \}$.

\textbf{Exa-3.} $m=8$. Since $8_1={7 \choose 1}+1=8$, we obtain $\mathcal{E}_1(8, 8)=\big \{\{$1, 2$\}$, $\{$1, 3$\}$, $\{$1, 4$\}$, $\{$1, 5$\}$, $\{$1, 6$\}$, $\{$1, 7$\}$, $\{$1, 8$\}$, $\{$2, 3, 4, 5, 6, 7, 8$\}\big \}$.

Since $8_2={6 \choose 1}+{6 \choose 2}=21$, thus $\mathcal{E}_2(8, 21)=\big \{\{$1, 2, 3$\}$, $\{$1, 2, 4$\}$, $\{$1, 2, 5$\}$, $\{$1, 2, 6$\}$, $\{$1, 2, 7$\}$, $\{$1, 2, 8$\}$, $\{$2, 3, 4$\}$, $\{$2, 3, 5$\}$, $\{$2, 3, 6$\}$, $\{$2, 3, 7$\}$, $\{$2, 3, 8$\}$, $\{$2, 4, 5$\}$, $\{$2, 4, 6$\}$, $\{$2, 4, 7$\}$, $\{$2, 4, 8$\}$, $\{$2, 5, 6$\}$, $\{$2, 5, 7$\}$, $\{$2, 5, 8$\}$, $\{$2, 6, 7$\}$, $\{$2, 6, 8$\}$, $\{$2, 7, 8$\}\big \}$.

Since $8_3={5 \choose 1}+{5 \choose 2}+{5 \choose 3}=25$, immediately, $\mathcal{E}_3(8, 25)=\big \{\{$1, 2, 3, 4$\}$, $\{$1, 2, 3, 5$\}$, $\{$1, 2, 3, 6$\}$, $\{$1, 2, 3, 7$\}$, $\{$1, 2, 3, 8$\}$, $\{$2, 3, 4, 5$\}$, $\{$2, 3, 4, 6$\}$, $\{$2, 3, 4, 7$\}$, $\{$2, 3, 4, 8$\}$, $\{$2, 3, 5, 6$\}$, $\{$2, 3, 5, 7$\}$, $\{$2, 3, 5, 8$\}$, $\{$2, 3, 6, 7$\}$, $\{$2, 3, 6, 8$\}$, $\{$2, 3, 7, 8$\}$, $\{$3, 4, 5, 6$\}$, $\{$3, 4, 5, 7$\}$, $\{$3, 4, 5, 8$\}$, $\{$3, 4, 6, 7$\}$, $\{$3, 4, 6, 8$\}$, $\{$3, 4, 7, 8$\}$, $\{$3, 5, 6, 7$\}$, $\{$3, 5, 6, 8$\}$, $\{$3, 5, 7, 8$\}$, $\{$3, 6, 7, 8$\}\big \}$.

Since $8_4={4 \choose 1}+{4 \choose 2}+{4 \choose 3}+{4 \choose 4}=15$, hence $\mathcal{E}_4(8, 15)=\big \{\{$1, 2, 3, 4, 5$\}$, $\{$1, 2, 3, 4, 6$\}$, $\{$1, 2, 3, 4, 7$\}$, $\{$1, 2, 3, 4, 8$\}$, $\{$2, 3, 4, 5, 6$\}$, $\{$2, 3, 4, 5, 7$\}$, $\{$2, 3, 4, 5, 8$\}$, $\{$2, 3, 4, 6, 7$\}$, $\{$2, 3, 4, 6, 8 $\}$, $\{$2, 3, 4, 7, 8$\}$, $\{$3, 4, 5, 6, 7$\}$, $\{$3, 4, 5, 6, 8$\}$, $\{$3, 4, 5, 7, 8$\}$, $\{$3, 4, 6, 7, 8$\}$, $\{$4, 5, 6, 7, 8$\}\big \}$.

Since $8_5={3 \choose 1}+{3 \choose 2}+{3 \choose 3}=7$, we have $\mathcal{E}_5(8, 7)=\big \{\{$1, 2, 3, 4, 5, 6$\}$, $\{$1, 2, 3, 4, 5, 7$\}$, $\{$1, 2, 3, 4, 5, 8$\}$, $\{$2, 3, 4, 5, 6, 7$\}$, $\{$2, 3, 4, 5, 6, 8$\}$, $\{$2, 3, 4, 5, 7, 8$\}$, $\{$3, 4, 5, 6, 7, 8$\}\big \}$.\qqed
\end{example}

\begin{problem}\label{qeu:444444}
\textbf{Compute} the exact value of each one of numbers $n_r$ and $A_m$ for proper hyperedge sets, or strong 3I-hyperedge sets introduced in Definition \ref{defn:strong-proper-hyperedge-sets}.
\end{problem}

\begin{thm}\label{thm:99999}
Each strong 3I-hyperedge set $\mathcal{E}_r(m, m_r)$ defined in Example \ref{exa:build-strong-hyperedge-sets} corresponds to a hypergraph $\mathcal{H}_{yper}=(\Lambda,\mathcal{E}_r(m, m_r))$ with the 3I-hyperedge set $\mathcal{E}_r(m, m_r)\in \mathcal{E}\big (\Lambda^2\big )$ to be a \emph{complete hypergraph}.
\end{thm}

\subsubsection{Set-set additive groups}

\begin{defn} \label{defn:set-set-additive-groups}
For $\Lambda=[1,M]$, let $\mathcal{B}=\{\mathcal{E}_1,\mathcal{E}_2,\dots ,\mathcal{E}_{n}\}$ be a subset of the hypergraph set $ \mathcal{E}(\Lambda^2)$, where each 3I-hyperedge set $\mathcal{E}_i=\{e_{i,1}, e_{i,2},\dots ,e_{i,p}\}$ with $i\in [1,n]$, and hyperedges $e_{i,j}=\big \{x_{i,j,1},x_{i,j,2}$, $\dots $, $x_{i,j,m(1,j)}\big \}$ with $j\in [1,p]$, as well as vertices $x_{i,j,k}\in \Lambda$ for $k\in [1,m(1,j)]$.

Notice that $|\mathcal{E}_a|=|\mathcal{E}_b|=p$ and, $|e_{a,j}|=|e_{b,j}|=m(1,j)$ for $a,b\in [1,n]$ and $a\neq b$, $j\in [1,p]$. We define the operation $\mathcal{E}_i[+]r~(\bmod~M)$ for $i\in [1,n]$ by an operation $e_{i,j}[+]r~(\bmod~M)$ for each $j\in [1,p]$ and $r\in [1,M]$, where
\begin{equation}\label{eqa:555555}
e_{i,j}[+]r~(\bmod~M):=\big \{x_{i,j,1}+r,x_{i,j,2}+r,\dots ,x_{i,j,m(1,j)}+r\big \}~(\bmod~M)
\end{equation}
The finite module abelian additive operation $e_{i,j}[+]e_{i,s}[-]e_{i,t}~(\bmod~M)=e_{i,\mu}$ is defined by
\begin{equation}\label{eqa:555555}
x_{i,j,k}+x_{i,s,k}-x_{i,t,k}~(\bmod~M)=x_{i,\mu,k},~k\in [1,m(1,j)]
\end{equation}
with $\mu=j+s-t~(\bmod~M)$. We get the finite module abelian additive operation $\mathcal{E}_j[+]\mathcal{E}_{s}[-]\mathcal{E}_{t}~(\bmod~M)=\mathcal{E}_{\mu}\in \mathcal{B}$ with $\mu=j+s-t~(\bmod~M)$ and a preappointed zero $\mathcal{E}_{t}\in \mathcal{B}$.

The hypergraph subset $\mathcal{B}$ obeys the \emph{Zero}, \emph{Inverse}, \emph{Uniqueness and Closureness}, \emph{Associative law} and \emph{Commutative law} under the finite module abelian additive operation. Then, the hypergraph subset $\mathcal{B}$ is an \emph{every-zero set-set additive group}, denoted as $\{F(\mathcal{B});[+][-]\}$.\qqed
\end{defn}

\begin{prop}\label{proposition:99999}
For a consecutive integer set $\Lambda=[1,M]$, by Definition \ref{defn:set-set-additive-groups}, each 3I-hyperedge set $\mathcal{E}\in \mathcal{E}\big ([1,M]^2\big )$ forms an every-zero set-set additive group $\{F(\mathcal{E});[+][-]\}$.
\end{prop}

\begin{thm}\label{thm:power-set-additive-group}
For a consecutive integer set $\Lambda=[1,M]$, the power set $[1,M]^2$ can be classified into subsets $U_1,U_2,\dots ,U_M$, such that each subset $U_k$ for $k\in [1,M]$ holds $|e|=k$ for each hyperedge $e\in U_k$ and forms an \emph{every-zero set-set additive group} $\{F(U_k);[+][-]\}$ under the finite module abelian additive operation, also, each subset $U_k$ with $k\in [1,M-1]$ is a \emph{fixed-point every-zero set-set additive group} (Ref. Definition \ref{defn:fixed-point-closed}), and its own complementary subset is the subset $U_{M-k}$.
\end{thm}

\begin{example}\label{exa:power-set-additive-group}
For $[1,4]=\{1,2,3,4\}$, the power set $[1,4]^2$ can be classified into every-zero set-set additive groups based on the finite module abelian additive operation by Theorem \ref{thm:power-set-additive-group}, so the power set $[1,4]^2=\bigcup ^4_{k}U_k$, where $U_1=\big \{\{1\},\{2\},\{3\},\{4\}\big \}$, $U_2=\big \{\{1,2\},\{1,3\},\{1,4\}$, $\{2,3\}$, $\{2,4\},\{3,4\}\big \}$, $U_3=\big \{\{1,2,3\},\{1,2,4\},\{1,3,4\},\{2,3,4\}\big \}$, $U_4=\{1,2,3,4\}=[1,4]$, and moreover each $U_k$ is a fixed-point every-zero set-set additive group under the transformation of the finite module abelian additive operation.

However, the power set $[1,M]^2$ contains other every-zero set-set additive groups, not like sets $U_k$ shown in Theorem \ref{thm:power-set-additive-group}, see examples shown in Example \ref{exa:finite-module-abelian-additive-operation}.\qqed
\end{example}

\subsection{Hyperedge-hamiltonian hypergraphs}

The hamiltonian problem (HP) is one of the famous problems in graph theory, and the problem of determining whether a connected graph has a Hamilton cycle is open now. The HP has led to many difficult problems of graph theory. In 1972, Karp proved that the hamiltonian problem is a NP-complete problem (Ref. \cite{Karp-R-M-1972-Computer}).

\begin{defn} \label{defn:Wang-Hyper-cycle}
\cite{Jianfang-Wang-Hypergraphs-2008} Let $\mathcal{H}_{yper}=(\Lambda,\mathcal{E})$ based on a 3I-hyperedge set $\mathcal{E}\in \mathcal{E}\big (\Lambda^2\big )$ be a $k$-uniform hypergraph based on a finite set $\Lambda=\{x_1,x_2,\dots ,x_n\}$ . A $(k-1)$-cycle $C=e_1e_2\cdots e_{t}e_1$ of the $k$-uniform hypergraph $\mathcal{H}_{yper}$ is called \emph{hyperedge-hamiltonian cycle} if every vertex of the $k$-uniform hypergraph $\mathcal{H}_{yper}$ appears exactly in each set of $e_{t}\cap e_{1}$ and $e_j\cap e_{j+1}$ with $j\in [1,t-1]$, and $|e_{s}\cap e_{t}|=k-1$ if $s\neq t$. \qqed
\end{defn}

For the $(k-1)$-cycle $C=e_1e_2\cdots e_{t}e_1$ appeared in Definition \ref{defn:Wang-Hyper-cycle}, we have $t(k-1)=n(k-1)$, so $t=n$. In other words, $C$ is a hyperedge-hamiltonian cycle. For general cases, we show the following definition of hyperedge-hamiltonian hypergraphs:

\begin{defn} \label{defn:yao-hamilton-hypergraphs}
Suppose that the vertices $x_{i_1},x_{i_2},\dots ,x_{i_n}$ of a finite set $\Lambda=\{x_1,x_2,\dots ,x_n\}$ are a permutation of vertices $x_1,x_2,\dots ,x_n$ of a hypergraph $\mathcal{H}_{yper}=(\Lambda,\mathcal{E})$ based on a 3I-hyperedge set $\mathcal{E}=\{e_1,e_2,\dots ,e_n\}\in \mathcal{E}\big (\Lambda^2\big )$, where $x_{i_j}\in e_{j-1}\cap e_j$, $x_{i_{n-1}}\in e_{n-1}\cap e_n$ and $x_{i_n}\in e_{n}\cap e_1$ for hyperedges $e_j\in \mathcal{E}$ with $j\in [1,n]$, then the hypergraph $\mathcal{H}_{yper}=(\Lambda,\mathcal{E})$ contains a \emph{proper hyperedge-hamiltonian path} $\mathcal{P}(e_1,e_n)=e_1e_2\cdots e_n$ and a \emph{proper hyperedge-hamiltonian cycle} $\mathcal{C}=e_1e_2\cdots e_ne_1$, respectively.\qqed
\end{defn}

Similarly with Definition \ref{defn:yao-hamilton-hypergraphs}, if each hyperedge of a 3I-hyperedge set $\mathcal{E}\in \mathcal{E}\big (\Lambda^2\big )$ is in a proper hyperedge-hamiltonian cycle $\mathcal{C}$ of the hypergraph $\mathcal{H}_{yper}=(\Lambda,\mathcal{E})$, such that $|\mathcal{E}|=|\mathcal{C}|$, then we call $\mathcal{H}_{yper}=(\Lambda,\mathcal{E})$ to be \emph{proper hyperedge-hamiltonian hypergraph}.

\begin{example}\label{exa:8888888888}
The graph $G$ shown in Fig.\ref{fig:4-uniform-H-graph}(a) is a ve-intersected graph (Ref. Definition \ref{defn:vertex-intersected-graph-hypergraph}) of a hypergraph $\mathcal{H}_{yper}=([1,8],\mathcal{E})$ based on a 3I-hyperedge set
{\small
\begin{equation}\label{eqa:555555}
\mathcal{E}=\big \{\{1,2,3,4\}, \{2,3,4,5\}, \{3,4,5,6\}, \{4,5,6,7\}, \{5,6,7,8\}, \{6,7,8,1\}, \{7,8,1,2\}, \{8,1,2,3\} \big \}
\end{equation}
}where the ve-intersected graph $G$ contains a $3$-uniform hyperedge-hamiltonian cycle $C_1$ and a $1$-uniform hyperedge-hamiltonian cycle $C_1$ shown in Fig.\ref{fig:4-uniform-H-graph} (b) and (c) respectively, and moreover the ve-intersected graph $G$ holds that each edge is in a hamiltonian cycle of the graph $G$, so the ve-intersected graph $G$ is \emph{edge-hamiltonian}. It is noticeable, the 3I-hyperedge set $\mathcal{E}$ is an every-zero set-set additive group under the finite module abelian additive operation $e_i[+]e_j[-]e_k=e_\lambda$ with $\lambda=i+j-k~(\bmod~8)$ for a preappointed zero $e_k$.\qqed
\end{example}

\begin{figure}[h]
\centering
\includegraphics[width=16.4cm]{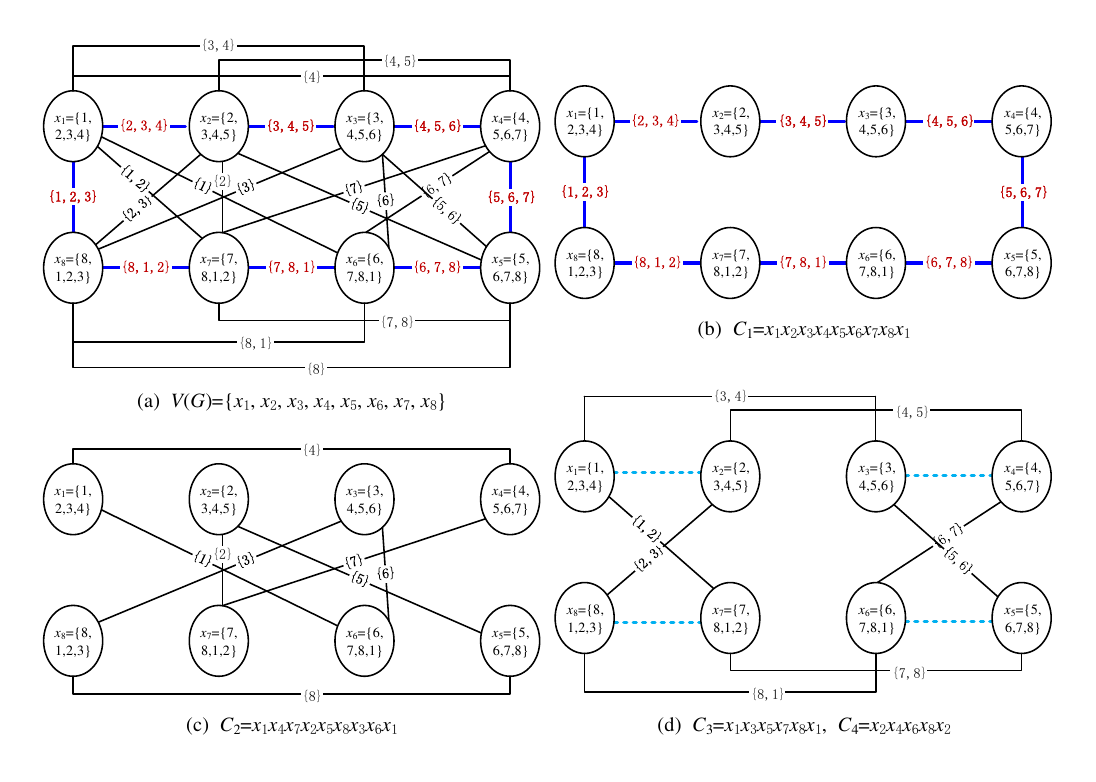}\\
\caption{\label{fig:4-uniform-H-graph}{\small A scheme for illustrating Definition \ref{defn:Wang-Hyper-cycle} and Definition \ref{defn:yao-hamilton-hypergraphs} by the ve-intersected graphs defined in Definition \ref{defn:new-intersected-hypergraphss}.}}
\end{figure}

\begin{thm}\label{thm:k-uniform-hamilton-hyperedge-cycle}
For a finite set $\Lambda=\{x_1,x_2,\dots ,x_n\}$, there are at least $B$ $k$-uniform hypergraphs $\mathcal{H}^i_{yper}=(\Lambda,\mathcal{E}_i)$ for each 3I-hyperedge set $\mathcal{E}_i\in \mathcal{E}(\Lambda^2)$ with $i\in [1, B]$, where $B=(n-1)\cdot n!$, such that each $k$-uniform hypergraphs $\mathcal{H}^i_{yper}$ contains a hyperedge-hamiltonian cycle.
\end{thm}
\begin{proof} Let $\Lambda=\{x_1,x_2,\dots ,x_n\}\rightarrow \Lambda_0=[1,n]$, we have a group $\{[1,k];[+][-]\}$ with $[1,k]=\{1,2,\dots ,k\}$ for $1<k\leq n-1$ based on the finite module abelian additive operation, and $[1+i,k+i]~(\bmod~n)$ for $i\in [1,n]$, $k+i=k+i~(\bmod~n)$ if $k+i\leq n$, $k+i=k+i-n$ if $k+i> n$. Then
\begin{equation}\label{eqa:555555}
[1+i,k+i][+][1+j,k+j][-][1+s,k+s]=[1+\lambda,k+\lambda],~\lambda=i+j-s~(\bmod~n)
\end{equation}
We have 3I-hyperedge sets $\mathcal{E}_i=\{e_{i,j}=[1+j,k+j]~(\bmod~n):j\in [1,n]\}$ with $i\in [1,n-1]$, and $\mathcal{E}_i\in \mathcal{E}(\Lambda^2_0)$, such that
\begin{equation}\label{eqa:555555}
e_{i,j}\cap e_{i,j+1}=[1+j+1,k+j],~|e_{i,j}\cap e_{i,j+1}|=k-1
\end{equation} since $k+j-(1+j+1)+1=k-1$, where $e_{i,n}=\{n\}\cup [1,k-1]$ and $e_{i,1}\cap e_{i,n}=[1,k-1]$. Notice that $j\in e_{i,j}\cap e_{i,j+1}$, we claim that each $k$-uniform hypergraph $\mathcal{H}^i_{yper}=(\Lambda,\mathcal{E}_i)$ containing a hyperedge-hamiltonian cycle $C_i=e_{i,1}e_{i,2}\cdot e_{i,n}e_{i,1}$ for each $i\in [1, n-1]$.

Each permutation $x_{i_1}, x_{i_2},\dots, x_{i_n}$ of elements $x_1,x_2,\dots ,x_n$ of the finite set $\Lambda$ corresponds to the integer set $\Lambda_0=[1,n]$, which shows that there are $k$-uniform hypergraphs $\mathcal{H}^i_{yper}=(\Lambda,\mathcal{E}_i)$ containing a hyperedge-hamiltonian cycle $C_i=e_{i,1}e_{i,2}\cdot e_{i,n}e_{i,1}$ with $i\in [1, n-1]$.

The proof of the theorem is complete.
\end{proof}

\begin{rem}\label{rem:333333}
There, about Theorem \ref{thm:k-uniform-hamilton-hyperedge-cycle}, are the following facts:

(i) We can use ve-intersected graphs of hypergraphs defined in Definition \ref{defn:vertex-intersected-graph-hypergraph} to prove Theorem \ref{thm:k-uniform-hamilton-hyperedge-cycle}, since each ve-intersected graph of a hypergraph is a visualization tool.

(ii) In \cite{Jianfang-Wang-Hypergraphs-2008}, Prof. Wang proposed: ``For a finite integer set $\Lambda=[1,n]$ and $n-2\geq k\geq 2$, the $(n-k-1)$-uniform hypergraph $\mathcal{H}_{yper}=(\Lambda,\overline{\mathcal{E}})$ contains a hyperedge-hamiltonian cycle if the $k$-uniform hypergraph $\mathcal{H}_{yper}=(\Lambda,\mathcal{E})$ containing a hyperedge-hamiltonian cycle''. However, it should have two hyperedge sets $\mathcal{E},\overline{\mathcal{E}}\in \mathcal{E}\big (\Lambda^2\big )$. \qqed
\end{rem}

\section{Topological Properties Of Hypergraphs}

The vertex intersected graphs introduced in this subsection are one of representations of the topological structure of sets and their subsets.

\subsection{A visualization tool of hypergraphs}

Hypergraphs can be studied by means of the topological characterization of subsets of the power set $\Lambda^2$ based on a finite set $\Lambda=\{x_1,x_2,\dots ,x_m\}$. The vertex intersected graphs introduced here are used as a visualization tool of hypergraphs.

\begin{defn}\label{defn:new-intersected-hypergraphss}
\textbf{The v-intersected graph.} Let $\mathcal{E}$ be a 3I-hyperedge set of the hypergraph set $\mathcal{E}\big (\Lambda^2\big )$ (Ref. Definition \ref{defn:new-hypergraphs-sets}) based on a finite set $\Lambda=\{x_1,x_2$, $\dots$, $x_m\}$. Suppose that a $(p,q)$-graph $H$ admits a set-coloring $\varphi: V(H)\rightarrow \mathcal{E}$, such that the vertex color set $\varphi(V(H))=\mathcal{E}$, and each edge $uv$ of $E(H)$ holds $\varphi(u)\cap \varphi(v)\neq \emptyset$. Conversely, each pair of hyperedges $e,e\,'\in \mathcal{E}$ with $|e\cap e\,'|\geq 1$ corresponds to an edge $xy\in E(H)$, such that $\varphi(x)=e$ and $\varphi(y)=e\,'$. Then the $(p,q)$-graph $H$ is called a \emph{v-intersected graph} of the hypergraph $\mathcal{H}_{yper}=(\Lambda,\mathcal{E})$.\qqed
\end{defn}

\begin{thm}\label{thm:new-intersected-hypergraph-topcode-matrix}
Let $E(H)=\{w_i=u_iv_i:i\in [1,q]\}$ be the edge set of the $(p,q)$-graph $H$ and let $\varphi(w_i)=\varphi(u_i)\cap \varphi(v_i)$ defined in Definition \ref{defn:new-intersected-hypergraphss}. By Definition \ref{defn:total-coloring-Topcode-matrixs}, the $(p,q)$-graph $H$ has its own Topcode-matrix as follows
\begin{equation}\label{eqa:11111111111111111}
{
\begin{split}
T_{code}(H,\varphi)= \left(
\begin{array}{cccccccccc}
\varphi(u_1) & \varphi(u_2) & \cdots & \varphi(u_q)\\
\varphi(w_1) & \varphi(w_2) & \cdots & \varphi(w_q)\\
\varphi(v_1) & \varphi(v_2) & \cdots & \varphi(v_q)
\end{array}
\right)_{3\times q}=(\varphi(U),\varphi(E),\varphi(V))^T_{3\times q}
\end{split}}
\end{equation}

Thereby, there is a graph set $G_{raph}(T_{code})$ with $T_{code}=T_{code}(H,\varphi)$, such that each graph $T_k\in G_{raph}(T_{code})$ is \emph{graph homomorphism} to $H$, i.e., $T_k\rightarrow H$, and each graph $T_k$ has its own Topcode-matrix $T_{code}(T_k,\varphi_k)=T_{code}$.
\end{thm}

\begin{rem}\label{rem:hyperedgeset-matrix}
Since $\varphi(u_i),\varphi(v_i)\in \mathcal{E}$ with $i\in [1,q]$ in Theorem \ref{thm:new-intersected-hypergraph-topcode-matrix}, the Topcode-matrix $T_{code}$ is also the \emph{topological code matrix} of the hypergraph set $\mathcal{E}\big (\Lambda^2\big )$, especially, denoted as $T^v_{code}(\Lambda,\mathcal{E})$, called \emph{hypergraph code-matrix}.\qqed
\end{rem}

\begin{problem}\label{question:444444}
Each graph $T_i\in G_{raph}(T_{code})$ in Theorem \ref{thm:new-intersected-hypergraph-topcode-matrix} is graph homomorphism to $H$, also $T_i\rightarrow H$, in other words, the adjacent matrix $A(T_i)_{n_i\times n_i}$ of the graph $T_i$ of $n_i$ vertices is \emph{adjacent matrix homomorphism} to the adjacent matrix $A(H)_{m\times m}$ of the graph $H$ of $m$ vertices, namely, $A(T_i)_{n_i\times n_i}\rightarrow A(H)_{m\times m}$ since $n_i\geq m+1$. There are many interesting investigations in the adjacent matrix homomorphism.
\end{problem}

\begin{example}\label{exa:8888888888}
The complement of a v-intersected graph $G$ is different from that of an ordinary graph, for example, we have a 3I-hyperedge set $\mathcal{E}=\big \{\{1,M\}, \{2,M\}, \dots ,\{M-1,M\}, [1,M-1]\big \}$, the v-intersected graph of the hypergraph $\mathcal{H}_{yper}=([1,M],\mathcal{E})$ with the 3I-hyperedge set $\mathcal{E}\in \mathcal{E}\big ([1,M]^2\big )$ is a complete graph of $M$ vertices, and the complementary hyperedge set of the 3I-hyperedge set $\mathcal{E}$ is $\overline{\mathcal{E}}=\big \{[2,M-1], [1,M]\setminus \{2,M\}, [1,M]\setminus \{3,M\}, \dots ,[1,M]\setminus \{M-1,M\}, \{M\}\big \}$. Notice that the complementary hyperedge set $\overline{\mathcal{E}}\not \in \mathcal{E}\big ([1,M]^2\big )$, so the v-intersected graph $H$ of the hypergraph $\mathcal{H}_{yper}=([1,M-1],\overline{\mathcal{E}}\setminus \{M\})$ based on a 3I-hyperedge set $\overline{\mathcal{E}}\setminus \{M\}\in \mathcal{E}([1,M-1]^2)$ is a complete graph of $M-1$ vertices, $|V(G)|=M=1+|V(H)|$.\qqed
\end{example}

\begin{prop}\label{proposition:99999}
Let $G_i$ be a v-intersected graph of a hypergraph $\mathcal{H}_{yper}=(\Lambda,\mathcal{E}_i)$ for $i=1,2$.

(i) If each $e\in \mathcal{E}_i$ holds $e\not \subseteq e\,'\in \mathcal{E}_{3-i}\setminus (\mathcal{E}_1\cap \mathcal{E}_2)$ with $i=1,2$, then the union graph $G_1\cup G_2$ is a v-intersected graph $H$ of a hypergraph $\mathcal{H}_{yper}=(\Lambda,\mathcal{E}_1\cup \mathcal{E}_2)$, and vice versa.

(ii) For $i=1,2$, we have $G_i\not \subseteq G_{3-i}$ if, and only if $\mathcal{E}_i\not \subseteq \mathcal{E}_{3-i}$.
\end{prop}

\begin{thm}\label{thm:infinite-v-intersected-graphs}
(a) A v-intersected graph $H$ of a hypergraph $\mathcal{H}_{yper}=(\Lambda,\mathcal{E})$ based on a 3I-hyperedge set $\mathcal{E}\in \mathcal{E}\big (\Lambda^2\big )$ corresponds to infinite hypergraphs $\mathcal{H}_{yper}=(\Lambda_k,\mathcal{E}_k)$, such that $G$ is a v-intersected graph of each one of the infinite hypergraphs $\mathcal{H}_{yper}=(\Lambda_k,\mathcal{E}_k)$ with the 3I-hyperedge set $\mathcal{E}_k\in \mathcal{E}(\Lambda_k^2)$.

(b) A hypergraph $\mathcal{H}_{yper}=(\Lambda,\mathcal{E})$ based on a 3I-hyperedge set $\mathcal{E}\in \mathcal{E}\big (\Lambda^2\big )$ has infinite v-intersected graphs defined in Definition \ref{defn:new-intersected-hypergraphss}.
\end{thm}
\begin{proof}(a) For a finite set $\Lambda=\{x_1,x_2$, $\dots$, $x_m\}$, we have a 3I-hyperedge set $\mathcal{E}=\{e_i:i\in [1,m]\}\in \mathcal{E}\big (\Lambda^2\big )$. Since $H$ is a v-intersected graph of the hypergraph $\mathcal{H}_{yper}=(\Lambda,\mathcal{E})$ by Definition \ref{defn:new-intersected-hypergraphss}, we set $\mathcal{E}_k=\big \{e_i\cup X_{k,i} :e_i\in \mathcal{E},i\in [1,m] \big \}$, where $X_{k,i}=\{x_{k,i,1},x_{k,i,2}, \dots, x_{k,i,n_{k,i}}\}$ with $n_{k,i}\geq 0$, and $X_{k,i}\cap X_{k,j}$ if $i\neq j$ and $\Lambda\cap X_{k,j}=\emptyset$ for $j\in [1,m]$, as well as $\Lambda_k=\Lambda\bigcup \big (\bigcup ^m_{i=1}X_{k,i} \big )$ and $\sum^m_{i=1}n_{k,i}\geq 1$. Clearly, $\mathcal{E}_k\in \mathcal{E}(\Lambda_k^2)$ is a 3I-hyperedge set. We claim that $H$ is a v-intersected graph of each hypergraph $\mathcal{H}_{yper}=(\Lambda_k,\mathcal{E}_k)$ for integers $k\geq 1$.

If $\Lambda_*=\bigcup ^m_{i=1}X_{k,i}$ with $X_{k,i}\neq \emptyset $ for $i\in [1,m]$ and the set-set $\mathcal{E}_*=\big \{X_{k,i}:i\in [1,m]\big \}$ form $\mathcal{E}_*\in \mathcal{E}(\Lambda_*^2)$, we get a v-intersected graph $G$ of the hypergraph $\mathcal{H}_{yper}=(\Lambda_*,\mathcal{E}_*)$. Then the union graph $L=H\cup G$ is a v-intersected graph of the hypergraph $\mathcal{H}_{yper}=(\Lambda_k,\mathcal{E}_k)$, such that the union graph $L=H\cup G$ has two vertex disjointed graphs $H$ and $G$.

(b) Notice that it is allowed that a v-intersected graph $H$ of a hypergraph $\mathcal{H}_{yper}=(\Lambda,\mathcal{E})$ has two ends of an edge are colored the same color by Definition \ref{defn:new-intersected-hypergraphss}. So the graph $H$ admits a set-coloring $\varphi: V(H)\rightarrow \mathcal{E}$, such that each edge $uv$ of $E(H)$ holds $\varphi(u)\cap \varphi(v)\neq \emptyset$. We take a copy $H\,'$ of the v-intersected graph $H$ of the hypergraph $\mathcal{H}_{yper}=(\Lambda,\mathcal{E})$. For each image vertex $u\,'\in V(H\,')$, we do:

(1) joining the image vertex $u\,'$ with the original image vertex $u\in V(H)$ by a new edge $u\,'u$;

(2) joining the image vertex $u\in V(H)$ with each vertex $v\,'_i\in N_{ei}(u\,')$ by a new edge $uv\,'_i$; and

(3) joining the image vertex $u\,'$ with each vertex $v_i\in N_{ei}(u)$ by a new edge $u\,'v_i$. \\
The resultant graph is denoted as $H[\ominus ]H\,'$. Now, we define a set-coloring $\theta$ for the graph $H[\ominus ]H\,'$ as follows:

(i) $\theta(x)=\varphi(x)$ for each vertex $x\in V(H)$;

(ii) $\theta(u\,')=\varphi(u)$ for each image vertex $u\,'\in V(H\,')$ and its original image vertex $u\in V(H)$. \\
Hence, the graph $H[\ominus ]H\,'$ is a v-intersected graph of the hypergraph $\mathcal{H}_{yper}=(\Lambda,\mathcal{E})$, because of $H$ is a v-intersected graph of the hypergraph $\mathcal{H}_{yper}=(\Lambda,\mathcal{E})$.

The theorem has been proved.
\end{proof}

\begin{defn} \label{defn:hyperedge-hamiltonian-connected}
\textbf{Hyperedge connectivity and hyperedge-hamiltonian connectivity}. If each pair of hyperedges $e,e\,'$ of a 3I-hyperedge set $\mathcal{E}$ corresponds to a hyperedge-path $\mathcal{P}(e,e\,')$ with the starting point $e$ and the end point $e\,'$, then the hypergraph $\mathcal{H}_{yper}=(\Lambda,\mathcal{E})$ based on a 3I-hyperedge set $\mathcal{E}\in \mathcal{E}\big (\Lambda^2\big )$ is a \emph{hyperedge-connected hypergraph}, correspondingly, and its v-intersected graph is \emph{hyperedge-connected}. If each pair of hyperedges $e,e\,'$ of the 3I-hyperedge set $\mathcal{E}$ is connected by a hyperedge-hamiltonian path $\mathcal{P}_{hami}(e,e\,')$, then the v-intersected graph $H$ is hamiltonian-connected, and the hypergraph $\mathcal{H}_{yper}=(\Lambda,\mathcal{E})$ is a \emph{hyperedge-hamiltonian connected hypergraph}.\qqed
\end{defn}

\begin{thm}\label{thm:666666}
For an integer set $\Lambda=[1,n]$ with $n\geq 3$, let $H$ be a v-intersected graph of a hypergraph $\mathcal{H}_{yper}=(\Lambda$, $\mathcal{E})$ based on a 3I-hyperedge set $\mathcal{E}\in \mathcal{E}\big (\Lambda^2\big )$. By Definition \ref{defn:new-intersected-hypergraphss} and Definition \ref{defn:hyperedge-hamiltonian-connected}, we have:

(i) If the v-intersected graph $H$ is \emph{hamiltonian-connected}, then the hypergraph $\mathcal{H}_{yper}=(\Lambda,\mathcal{E})$ is \emph{hyperedge-hamiltonian connected}.

(ii) If the v-intersected graph $H$ is \emph{edge-hamiltonian}, namely, each edge $uv\in E(H)$ is in a hamiltonian cycle of the v-intersected graph $H$, then the hypergraph $\mathcal{H}_{yper}=(\Lambda,\mathcal{E})$ is \emph{hyperedge-hyperedge-hamiltonian hypergraph}, also, each hyperedge $e\in \mathcal{E}$ is in a hyperedge-hamiltonian cycle of the hypergraph $\mathcal{H}_{yper}=(\Lambda,\mathcal{E})$.
\end{thm}

\begin{thm}\label{thm:666666}
For an integer set $\Lambda=[1,n]$ with $n\geq 3$, there exists a 3I-hyperedge set $\mathcal{E}\in \mathcal{E}\big (\Lambda^2\big )$ with some $e_i\in \mathcal{E}$ holding $|e_i|\geq 3$, such that a v-intersected graph of the hypergraph $\mathcal{H}_{yper}=(\Lambda,\mathcal{E})$ is a maximal planar graph.
\end{thm}
\begin{proof} By the induction. For $\Lambda=[1,3]$, the complete graph $K_3$ is a maximal planar graph of 3 vertices, let $V(K_3)=\{x,y,z\}$. We define a set-coloring $f$ for $K_3$ as follows: $f(x)=\{1,2\}$, $f(y)=\{1,3\}$, $f(z)=\{2,3\}$, so we get a 3I-hyperedge set $\mathcal{E}=\{\{1,2\}, \{1,3\}, \{2,3\}\}\in \mathcal{E}\big ([1,3]^2\big )$.

Suppose that a maximal planar graph $G$ is a v-intersected graph of the hypergraph $\mathcal{H}_{yper}=([1,n]$, $\mathcal{E})$ under a set-coloring $h: V(G)\rightarrow \mathcal{E}\in \mathcal{E}\big ([1,n]^2\big )$.

\textbf{Case-1.} If the maximal planar graph $G$ contains a $K_4$ with $V(K_4)=\{x_1,x_2,x_3,x_4\}$, such that $n\in h(x_i)$ for $i\in [1,4]$, so $G-x_4$ is still a maximal planar graph, where $x_4$ is a vertex of three degree. We define a set-coloring $g$ for the maximal planar graph $G$ as follows: $g(x)=h(x)$ for $x\in V(G)\setminus \{x_4\}$, $g(x_4)=h(x_4)\cup \{n+1\}$. Then we get a 3I-hyperedge set $\mathcal{E}^*=\big (\mathcal{E}\setminus h(x_4)\big )\cup g(x_4)\in \mathcal{E}\big ([1,n+1]^2\big )$.

\textbf{Case-2.} There is no $K_4$ in the maximal planar graph $G$. We select a triangle face $\Delta uvw$ of the maximal planar graph $G$, where $n\in h(u)=S_u$, or $n\in h(w)=S_v$, or $n\in h(w)=S_w$, but $n\not \in h(x)$ for any vertex $x\in V(G)\setminus \{u,v,w\}$, since $G$ is a maximal planar graph.

We add a new vertex $t$ in the triangle face $\Delta uvw$ and join $t$ with three vertices $u,v,w$ by edges $ut$, $vt$ and $wt$, the resultant graph is denoted as $G^*$ which still is a maximal planar graph. We make a set-coloring $h^*$ for the maximal planar graph $G^*$ in the following way: Set $h^*(x)=h(x)$ for each vertex $x\in V(G)\setminus \{u,v,w\}\subset V(G^*)$, $h^*(u)=S_u\cup \{n\}$, $h^*(v)=S_v\cup \{n\}$, $h^*(w)=S_w\cup \{n\}$, $h^*(t)=\{n,n+1\}$, we obtain a new hyperedge set $\mathcal{E}^*=\big (\mathcal{E}\setminus \{h(u),h(v),h(w)\}\big )\cup \{h^*(u),h^*(v),h^*(w),h^*(t)\}$, and $h^*:V(G^*)\rightarrow \mathcal{E}^*$. It is not hard to verify $\mathcal{E}^*\in \mathcal{E}\big ([1,n+1]^2\big )$.

The proof is complete by the induction.
\end{proof}

\begin{thm}\label{thm:666666}
For an integer set $\Lambda=[1,n]$ with $n\geq 3$, there exists a 3I-hyperedge set $\mathcal{E}\in \mathcal{E}\big (\Lambda^2\big )$ with some $e_i\in \mathcal{E}$ holding $|e_i|\geq 3$, such that a v-intersected graph of the hypergraph $\mathcal{H}_{yper}=(\Lambda,\mathcal{E})$ is a tree.
\end{thm}
\begin{proof} By the induction. The set $\mathcal{E}^*=\big \{\{1,2,3\},\{1\},\{2\},\{3\}\big \}$ for $\Lambda=[1,3]$ is a 3I-hyperedge set of $\mathcal{E}\big ([1,3]^2\big )$, and the v-intersected graph of the hypergraph $\mathcal{H}_{yper}=([1,3],\mathcal{E}^*)$ is just a tree.

Suppose that a v-intersected graph $T$ of the hypergraph $\mathcal{H}_{yper}=([1,n],\mathcal{E})$ is a tree. So, the tree $T$ admits a set-coloring $f: V(T)\rightarrow \mathcal{E}\in \mathcal{E}\big ([1,n]^2\big )$. For a vertex $u$ of one degree, the graph $T-u$ is still a tree. We define a new set-coloring $g$ for the tree $T$ as follows: $g(x)=f(x)$ for $x\in V(T)\setminus \{u,v\}$ where $uv\in E(T)$, $g(u)=\{n+1\}$, $g(v)=f(v)\cup \{n+1\}$. Then, we get a new 3I-hyperedge set $\mathcal{E}^*=\big(\mathcal{E}\setminus \{f(u),f(v)\}\big )\cup \{g(u),g(v)\}\in \mathcal{E}\big ([1,n+1]^2\big )$, such that the tree $T$ is a v-intersected graph of the hypergraph $\mathcal{H}_{yper}=([1,n+1],\mathcal{E}^*)$.

The proof is complete by the induction.
\end{proof}

\begin{problem}\label{question:444444}
(1) \textbf{Determine} the following particular 3I-hyperedge sets $\mathcal{E}\in \mathcal{E}\big ([1,n]^2\big )$:

(1.1) $\mathcal{E}$ is a \emph{tree hyperedge set} if a v-intersected graph of the hypergraph $\mathcal{H}_{yper}=([1,n],\mathcal{E})$ is a tree.

(1.2) A v-intersected graph of the hypergraph $\mathcal{H}_{yper}=([1,n],\mathcal{E})$ is a hamiltonian graph, we call $\mathcal{E}$ \emph{hamiltonian hyperedge set}.

(1.3) A v-intersected graph of the hypergraph $\mathcal{H}_{yper}=([1,n],\mathcal{E})$ is a maximal planar graph, we call $\mathcal{E}$ \emph{MPG-hyperedge set}.

(2) \textbf{List} all 3I-hyperedge sets of the hypergraph set $\mathcal{E}\big ([1,n]^2\big )$. However, it is NP-hard, since the hypergraph set $\mathcal{E}\big ([1,n]^2\big )$ corresponds to a graph set $G_{raph}(\mathcal{E})$ of v-intersected graphs, and listing all graphs of the graph set $G_{raph}(\mathcal{E})$ will meet the Subgraph Isomorphic Problem, which is a NP-complete problem.
\end{problem}

\begin{defn}\label{defn:vertex-intersected-graph-hypergraph}
\cite{Yao-Ma-arXiv-2024-13354} \textbf{The ve-intersected graph.} For a 3I-hyperedge set $\mathcal{E}\in \mathcal{E}\big (\Lambda^2\big )$ based on a finite set $\Lambda=\{x_1,x_2,\dots ,x_m\}$ (Ref. Definition \ref{defn:new-hypergraphs-sets}). Suppose that a $(p,q)$-graph $G$ admits a total set-coloring $F: V(G)\cup E(G)\rightarrow \mathcal{E}$ with $F(x)\neq F(y)$ for each edge $xy\in E(G)$ and the total color set $F( V(G)\cup E(G))=\mathcal{E}$, and $R_{est}(c_0,c_1,c_2,\dots ,c_m)$ with $m\geq 0$ is a constraint set, such that each edge $uv$ of $E(G)$ is colored with an edge color set $F(uv)$ and

(i) the first constraint $c_0$: $F(uv)\supseteq F(u)\cap F(v)\neq \emptyset$.

(ii) the $k$th constraint $c_k$: There is a function $\varphi_k$ for some $k\in [1,m]$, we have three numbers $c_{uv}\in F(uv)$, $a_u\in F(u)$ and $b_v\in F(v)$ holding the $k$th constraint $c_k:\varphi_k[a_u,c_{uv},b_v]=0$ true.

Each pair of hyperedges $e,e\,'\in \mathcal{E}$ with $|e\cap e\,'|\geq 1$ corresponds to an edge $xy\in E(G)$, such that $F(x)=e$, $F(xy)\supseteq e\cap e\,'$ and $F(y)=e\,'$. Then we call $G$ \emph{ve-intersected graph} of the hypergraph $\mathcal{H}_{yper}=(\Lambda,\mathcal{E})$ based on a 3I-hyperedge set $\mathcal{E}\in \mathcal{E}\big (\Lambda^2\big )$ subject to the constraint set $R_{est}(c_0,c_1,c_2,\dots ,c_m)$, since $F(V(G)\cup E(G))=\mathcal{E}$.\qqed
\end{defn}

\begin{defn} \label{defn:more-terminology-group}
About a ve-intersected graph $H$ of a hypergraph $\mathcal{H}_{yper}=(\Lambda,\mathcal{E})$ based on a 3I-hyperedge set $\mathcal{E}\in \mathcal{E}\big (\Lambda^2\big )$ defined in Definition \ref{defn:vertex-intersected-graph-hypergraph}, we define:
\begin{asparaenum}[\textbf{\textrm{Ter}}-1. ]
\item Each hyperedge $e\in \mathcal{E}$ has its own \emph{hyperedge-degree} $\textrm{deg}_{\mathcal{E}}(e)=\textrm{deg}_H(x)$ if $F(x)=e$ for $x\in V(H)$.
\item The \emph{hyperedge-degree sequence} $\{\textrm{deg}_{\mathcal{E}}(e_1),\textrm{deg}_{\mathcal{E}}(e_2),\dots, \textrm{deg}_{\mathcal{E}}(e_n)\}$ with $e_i\in \mathcal{E}$ satisfies Erd\"{o}s-Galia Theorem. In fact, each hyperedge-degree
\begin{equation}\label{eqa:555555}
\textrm{deg}_{\mathcal{E}}(e_i)=\Big |\big \{e_j:e_i\cap e_j\neq \emptyset, e_j\in \mathcal{E}\setminus e_i\big \}\Big |
\end{equation}
\item If each hyperedge $e\in \mathcal{E}$ has its own hyperedge-degree to be even, then $\mathcal{E}$ is called an \emph{Euler's hyperedge set}.
\item A \emph{hyperedge-path} $\mathcal{P}$ in the hypergraph $\mathcal{H}_{yper}=(\Lambda,\mathcal{E})$ is
\begin{equation}\label{eqa:555555}
\mathcal{P}(e_1,e_m)=e_1e_2\cdots e_m=e_1(e_1\cap e_2)e_2(e_2\cap e_3)\cdots (e_{m-1}\cap e_m)e_m
\end{equation} with hyperedge intersections $e_{i}\cap e_{i+1}\neq \emptyset$ for $i\in [1,m-1]$, and each hyperedge $e_i$ is not an ear for $i\in [2,m-1]$. Moreover, the hyperedge-path $\mathcal{P}$ is \emph{pure} if $e_1$ and $e_m$ are not ears of $\mathcal{E}$. If hyperedge intersections $|e_{i}\cap e_{i+1}|\geq r$ for $i\in [1,m-1]$, we call $\mathcal{P}$ $r$-\emph{uniform hyperedge-path}.
\item If the ve-intersected graph $H$ is \emph{bipartite}, then the 3I-hyperedge set $\mathcal{E}=X_{\mathcal{E}}\cup Y_{\mathcal{E}}$ with $X_{\mathcal{E}}\cap Y_{\mathcal{E}}=\emptyset$, such that any two hyperedges $e,e\,'\in X_{\mathcal{E}}$ (resp. $e,e\,'\in Y_{\mathcal{E}}$) holds $e\cap e\,'=\emptyset$.
\item A \emph{spanning hypertree} $\mathcal{T}$ of the ve-intersected graph $H$ holds that each vertex color set $F(x)$ is not an ear of $\mathcal{E}$ if $x\not\in L(\mathcal{T})$, where $L(\mathcal{T})$ is the set of all leaves of $\mathcal{T}$, and $\mathcal{T}$ contains no hyperedge-cycle.
\item If the ve-intersected graph $H$ admits a proper vertex coloring $\theta:V(H)\rightarrow [1,\chi(H)]$, then the 3I-hyperedge set $\mathcal{E}$ admits a proper hyperedge coloring $\theta:\mathcal{E}\rightarrow [1,\chi(H)]$ such that $\theta(e)$ differs from $\theta(e\,')$ if $e\,'\cap e\neq \emptyset $ for any pair of hyperedges $e,e\,'\in \mathcal{E}$, where $\chi(H)$ is the chromatic number of the graph $H$.
\item If the ve-intersected graph $H$ is connected and the 3I-hyperedge set $\mathcal{E}$ contains no ear, so the \emph{diameter} $D(H)$ of the ve-intersected graph $H$ is defined by
\begin{equation}\label{eqa:555555}
\max \{d(x,y): d(x,y)\textrm{ is the length of a shortest path between two vertices $x$ and $y$ in } H\}
\end{equation} then the \emph{hyperdiameter} $D(\mathcal{E})$ of the 3I-hyperedge set $\mathcal{E}$ is defined by $D(\mathcal{E})=D(H)$.

\item A \emph{dominating hyperedge set} $\mathcal{E}_{domi}$ is a proper subset of the 3I-hyperedge set $\mathcal{E}$ and holds: Each hyperedge $e\in \mathcal{E}\setminus \mathcal{E}_{domi}$ corresponds to some hyperedge $e^*\in \mathcal{E}_{domi}$ such that $e\cap e^*\neq \emptyset$.
\item The \emph{dual} $\mathcal{H}_{dual}$ of a hypergraph $\mathcal{H}_{yper}=(\Lambda,\mathcal{E})$ is also a hypergraph having its own vertex set $\Lambda_{dual}=\mathcal{E}=\{e_1,e_2,\dots ,e_n\}$ and its own 3I-hyperedge set $\mathcal{E}_{dual}=\{X_j\}^n_{j=1}$ with $X_j=\{e_j:x_j\in e_j,~x_j\in \Lambda\}$ and $n=|\mathcal{E}|$. Clearly, the dual of the hypergraph $\mathcal{H}_{dual}$ is just the original hypergraph $\mathcal{H}_{yper}=(\Lambda,\mathcal{E})$.\qqed
\end{asparaenum}
\end{defn}

\begin{thm}\label{thm:integer-set-vs-vertex-intersected-graph}
For any connected graph $G$, there is an integer set $\Lambda=[1,M]$, then the connected graph $G$ admits a set-coloring $F:V(G)\rightarrow \mathcal{E}\in \mathcal{E}\big (\Lambda^2\big )$ holding each edge $uv\in E(G)$ to be colored with $F(uv)=F(u)\cap F(v)\neq \emptyset$, and $F(uv)\neq F(uw)$ for distinct neighbors $v,w\in N_{ei}(u)$, such that $G$ is just a ve-intersected graph of the hypergraph $\mathcal{H}_{yper}=\big (\Lambda,\mathcal{E}\big )$ based on the 3I-hyperedge set $\mathcal{E}\in \mathcal{E}\big (\Lambda^2\big )$.
\end{thm}
\begin{proof} Using the induction. Take a leaf $x$ of a spanning tree $T$ of the connected graph $G$, where the spanning tree $T$ has the maximal leaves. So, the vertex-removing graph $G-x$ is connected. There is an integer set $\Lambda_x=[1,M_x]$, such that the vertex-removing graph $G-x$ admits a set-coloring $F_x:V(G-x)\rightarrow \mathcal{E}_x\in \mathcal{E}\big (\Lambda^2_x\big )$ and is a ve-intersected graph of the hypergraph $\mathcal{H}_{yper}(t)=\big (\Lambda_x,\mathcal{E}_x\big )$. Let $N_{ei}(x)=\{y_1,y_2,\dots ,y_m\}$ be the set of neighbors of the vertex $x$ in the connected graph $G$, and let $\mathcal{E}_x=\mathcal{E}^1_x\cup \mathcal{E}^2_x$, where $\mathcal{E}^2_x=\{F_x(y_i):y_i\in N_{ei}(x)\}$.

Let $\Lambda=[1,M]=[1,1+M_x]$. We define a set-coloring $F$ for the connected graph $G$ by setting $F(x)=\{1+M_x, 2+M_x, \dots ,m+M_x\}$, $F(y_i)=F_x(y_i)\cup \{i+M_x\}$ with $i\in [1,m]$, so each edge $xy_i\in E(G)$ is colored with $F(xy_i)=\{i+M_x\}$ with $i\in [1,m]$, and moreover we set $F(u)=F_x(u)$ for $u\in V(G)\setminus \{x,y_i:y_i\in N_{ei}(x)\}$. Clearly, $F(uv)\neq F(uw)$ for distinct vertices $v,w\in N_{ei}(u)$ in the connected graph $G$. Then we get a hyperedge set
\begin{equation}\label{eqa:555555}
\mathcal{E}=\mathcal{E}^1_x\bigcup F(x)\bigcup \left (\bigcup _{y_i\in N_{ei}(x)}F(y_i)\right )\in \mathcal{E}\big (\Lambda^2\big )
\end{equation} such that the connected graph $G$ admits a set-coloring $F:V(G)\rightarrow \mathcal{E}\in \mathcal{E}\big (\Lambda^2\big )$. Thereby, the connected graph $G$ is a ve-intersected graph of the hypergraph $\mathcal{H}_{yper}(t)=\big (\Lambda,\mathcal{E}\big )$.

We are done according to the induction.
\end{proof}

\begin{thm}\label{thm:666666666666666}
Let $E(G)=\{w_i=x_iy_i:i\in [1,q]\}$ be the edge set of the $(p,q)$-graph $G$ and let $F(w_i)\supseteq F(x_i)\cap F(y_i)\neq \emptyset$ in Definition \ref{defn:vertex-intersected-graph-hypergraph}. By Definition \ref{defn:total-coloring-Topcode-matrixs}, the $(p,q)$-graph $G$ has its own Topcode-matrix as follows
\begin{equation}\label{eqa:2222222222222}
{
\begin{split}
T_{code}(G,F)= \left(
\begin{array}{cccccccccc}
F(x_1) & F(x_2) & \cdots & F(x_q)\\
F(w_1) & F(w_2) & \cdots & F(w_q)\\
F(y_1) & F(y_2) & \cdots & F(y_q)
\end{array}
\right)_{3\times q}=(F(X),F(E),F(Y))^T_{3\times q}
\end{split}}
\end{equation} Then there is a graph set $G_{raph}(T_{code})$ with $T_{code}=T_{code}(G,F)$, such that each graph $H_k\in G_{raph}(T_{code})$ (as a \emph{public-key}) is \emph{graph homomorphism} to $G$ (as a \emph{private-key}), namely, $H_k\rightarrow G$, and each graph $H_k$ has its own Topcode-matrix $T_{code}(H_k,h_k)=T_{code}$, and the hypergraph code-matrix $T^{ve}_{code}(\Lambda,\mathcal{E})=T_{code}$ (Ref. Remark \ref{rem:hyperedgeset-matrix}).
\end{thm}

\begin{example}\label{example:444444}
Fig.\ref{fig:a-visualization-tool} is for illustrating Definition \ref{defn:new-intersected-hypergraphss} and Definition \ref{defn:vertex-intersected-graph-hypergraph}, we, in which, can see a v-intersected $(4,6)$-graph $G$ of a hypergraph $\mathcal{H}_{yper}=([1,15],\mathcal{E})$ based on a 3I-hyperedge set $\mathcal{E}=\{e_1,e_2,e_3,e_4\}$ and a ve-intersected $(4,6)$-graph $H$ of a hypergraph $\mathcal{H}_{yper}=([1,15],\mathcal{E}^*)$ based on a 3I-hyperedge set $\mathcal{E}^*=\{e_1,e_2,e_3,e_4,u_{12}, u_{13},u_{14},u_{23},u_{24},u_{34}\}$, where the $(4,6)$-graph $H$ admits a total set-coloring $\varphi$ defined as follows:

$e_1=\{1,2,6,12,9,11,7,15\}$; $e_2=\{2,3,5,6,11,10,7,13\}$;

$e_3=\{4,8,5,10,6,11,9,12\}$; $e_4=\{8,14,10,13,11,7,9,15\}$.

$u_{12}=\{2,6,7,11,12,15\}$ holding $u_{12}\supset e_1\cap e_2=\{2,6,7,11\}$ true;

$u_{13}=\{2,3,6,7,9,11,12,13\}$ holding $u_{13}\supset e_1\cap e_3=\{6,9,11,12\}$ true;

$u_{14}=\{4,5,6,7,9,11,12,15\}$ holding $u_{14}\supset e_1\cap e_4=\{7,9,11,15\}$ true;

$u_{23}=\{4,5,6,7,9,11,15\}$ holding $u_{23}\supset e_2\cap e_3=\{5,6,10,11\}$ true;

$u_{24}=\{1,3,6,7,9,10,11,13\}$ holding $u_{24}\supset e_2\cap e_4=\{7,10,11,13\}$ true;

$u_{34}=\{1,2,3,4,8,9,10,11,13\}$ holding $u_{34}\supset e_3\cap e_4=\{8,9,10,11\}$ true.

By Definition \ref{defn:total-coloring-Topcode-matrixs}, the $(4,6)$-graph $H$ has its own Topcode-matrix as follows
\begin{equation}\label{eqa:Topcode-matrix-split-operation}
{
\begin{split}
T_{code}(H,\varphi)=& \left(
\begin{array}{cccccccccc}
e_1 & e_1 & e_1 & e_2 & e_2 & e_3 \\
u_{12} & u_{13} & u_{14} & u_{23} & u_{24} & u_{34} \\
e_2 & e_3 & e_4 & e_3 & e_4 & e_4
\end{array}
\right)\\
=&\left(
\begin{array}{cccccccccc}
e_1 \\
u_{12}\\
e_2
\end{array}
\right)\bigcup\left(
\begin{array}{cccccccccc}
e_1 \\
u_{13}\\
e_3
\end{array}
\right)\bigcup\left(
\begin{array}{cccccccccc}
e_1 \\
u_{14}\\
e_4
\end{array}
\right)\bigcup\left(
\begin{array}{cccccccccc}
e_2 \\
u_{23}\\
e_3
\end{array}
\right)\bigcup\left(
\begin{array}{cccccccccc}
e_2 \\
u_{24}\\
e_4
\end{array}
\right)\bigcup\left(
\begin{array}{cccccccccc}
e_3 \\
u_{34}\\
e_4
\end{array}
\right)\\
=&T^{ve}_{code}(\mathcal{E}^*)
\end{split}}
\end{equation}

In Fig.\ref{fig:a-visualization-tool}, there are \emph{set-colored graph homomorphisms} $H_i\rightarrow H$ and Topcode-matrices $T_{code}(H,\varphi)=$ $T_{code}(H_i$, $\varphi_i)$ for $i\in [1,3]$, we say that the set-colored graph $H$ is \emph{core ve-intersected graph} of the hypergraph $\mathcal{H}_{yper}=([1,15],\mathcal{E}^*)$, since each set-colored graph $H_i$ is a ve-intersected graph of the hypergraph $\mathcal{H}_{yper}=([1,15],\mathcal{E}^*)$ for $i\in [1,3]$. Eq.(\ref{eqa:Topcode-matrix-split-operation}) shows us the \emph{Topcode-matrix splitting operation}.
\end{example}

\begin{figure}[h]
\centering
\includegraphics[width=14.2cm]{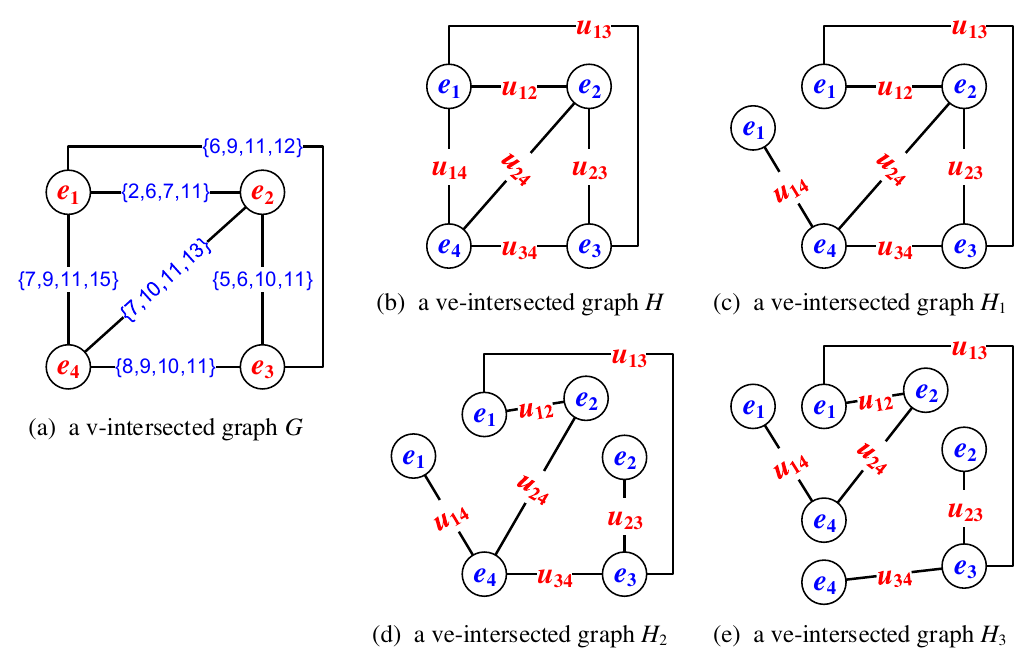}\\
\caption{\label{fig:a-visualization-tool}{\small The set-colored graphs for Definition \ref{defn:new-intersected-hypergraphss} and Definition \ref{defn:vertex-intersected-graph-hypergraph}.}}
\end{figure}

\begin{defn}\label{defn:hyperedge-structure-hypergraphs}
If a v-intersected graph $G$ of a hypergraph $\mathcal{H}_{yper}=(\Lambda,\mathcal{E})$ based on a 3I-hyperedge set $\mathcal{E}\in \mathcal{E}\big (\Lambda^2\big )$ defined in Definition \ref{defn:new-intersected-hypergraphss}, or a ve-intersected graph $G$ of a hypergraph $\mathcal{H}_{yper}=(\Lambda,\mathcal{E})$ based on a 3I-hyperedge set $\mathcal{E}\in \mathcal{E}\big (\Lambda^2\big )$ defined in Definition \ref{defn:vertex-intersected-graph-hypergraph} is one of tree, planar graph, graph having hamiltonian cycle, Euler graph, bipartite graph, and multiple-partition graph, then we say the hypergraph $\mathcal{H}_{yper}=(\Lambda,\mathcal{E})$ based on a 3I-hyperedge set $\mathcal{E}\in \mathcal{E}\big (\Lambda^2\big )$ to be \emph{tree-hypergraph}, \emph{hyperedge-planar hypergraph}, \emph{hyperedge-hamiltonian hypergraph}, \emph{hyperedge-Euler hypergraph}, \emph{hyperedge-bipartite hypergraph}, and \emph{hyperedge-multiple-partition hypergraph}.
\end{defn}

\begin{rem}\label{rem:333333}
\textbf{Distinguishability.} Hyperedge-hamiltonian hypergraphs defined in Definition \ref{defn:hyperedge-structure-hypergraphs} differ completely from proper hyperedge-hamiltonian hypergraphs defined in Definition \ref{defn:yao-hamilton-hypergraphs}, since a proper hyperedge-hamiltonian hypergraph $\mathcal{H}_{yper}=(\Lambda,\mathcal{E})$ based on a 3I-hyperedge set $\mathcal{E}\in \mathcal{E}\big (\Lambda^2\big )$ is related with some ordinary hamiltonian graph as the 3I-hyperedge set $\mathcal{E}$ is a 2-uniform 3I-hyperedge set, similarly, a proper connected hypergraph $\mathcal{H}_{yper}=(\Lambda,\mathcal{E})$ based on a 3I-hyperedge set $\mathcal{E}\in \mathcal{E}\big (\Lambda^2\big )$ is related with some ordinary connected graph.\qqed
\end{rem}

Thereby, we claim that
\begin{prop}\label{proposition:99999}
There is no necessary and sufficient condition for determining whether a hypergraph is a hyperedge-hamiltonian hypergraph, since there is no necessary and sufficient condition for determining whether a graph contains a hamiltonian cycle.
\end{prop}

\begin{problem}\label{question:444444}
\textbf{Find} the positive integer $M^*$ for each connected graph such that Theorem \ref{thm:integer-set-vs-vertex-intersected-graph} based on the integer set $\Lambda=[1,M^*]$ holds true.
\end{problem}

Theorem \ref{thm:integer-set-vs-vertex-intersected-graph} enables us to obtain the following theorems:

\begin{thm}\label{thm:infinite-hypergraph-v-a-graph}
Any connected graph $G$ corresponds to infinite hypergraphs, such that the connected graph $G$ is a ve-intersected graph of each one of the infinite hypergraphs.
\end{thm}

\begin{thm} \label{them:hyperedge-Hamilton-cycles}
Suppose that $G$ is a ve-intersected graph of a hypergraph $\mathcal{H}_{yper}=(\Lambda,\mathcal{E})$ based on a 3I-hyperedge set $\mathcal{E}\in \mathcal{E}\big (\Lambda^2\big )$, so the ve-intersected graph $G$ admits a proper total set-coloring $F: V(G)\cup E(G)\rightarrow \mathcal{E}$ holding $F(x)\neq F(y)$ for each edge $xy\in E(G)$ defined in Definition \ref{defn:vertex-intersected-graph-hypergraph}. Then we have:
\begin{asparaenum}[(1) ]
\item If the ve-intersected graph $G$ contains a hamiltonian cycle, then the hypergraph $\mathcal{H}_{yper}=(\Lambda,\mathcal{E})$ contains a \emph{hyperedge-hamiltonian cycle}.
\item If the ve-intersected graph $G$ is a tree, then the hypergraph $\mathcal{H}_{yper}=(\Lambda,\mathcal{E})$ is acyclic by the Graham reduction defined in \cite{Jianfang-Wang-Hypergraphs-2008}.
\item If the ve-intersected graph $G$ holds $|F(E(G))|=|E(G)|$ and $F(uv)\cap F(xy)=\emptyset$ for any pair of edges $uv$ and $xy$ of $E(G)$, and the ve-intersected graph $G$ is not a tree, then the hypergraph $\mathcal{H}_{yper}=(\Lambda,\mathcal{E})$ contains a \emph{hyperedge cycle}.
\end{asparaenum}
\end{thm}

\begin{thm}\label{thm:theorem3535}
For any connected $(p,q)$-graph $G$, there is a set-coloring $F:V(G)\rightarrow \mathcal{E}$, where the 3I-hyperedge set $\mathcal{E}\in \mathcal{E}\big (\Lambda^2\big )$ (Ref. Definition \ref{defn:new-hypergraphs-sets}), such that $F(x)\neq F(y)$ for distinct vertices $x,y\in V(G)$ and $F(V(G))=\mathcal{E}$, as well as induced edge colors $F(uv)=F(u)\cap F(v)$ for each edge $uv\in E(G)$, and $F(xy)\neq F(wz)$ for distinct edges $xy,wz\in E(G)$, and moreover $|\Lambda|$ is the smallest one for any set-coloring $F^*:V(G)\rightarrow \mathcal{E}^*$ based on a 3I-hyperedge set $\mathcal{E}^*\in \mathcal{E}(\Lambda^2_{*})$, as well as $F(x)\not\subset F(y)$ for distinct vertices $x,y\in V(G)$ and $F(xy)\not\subset F(wz)$ for distinct edges $xy,wz\in E(G)$, as well as $|\Lambda|\leq |\Lambda_{*}|$.
\end{thm}

\begin{figure}[h]
\centering
\includegraphics[width=14cm]{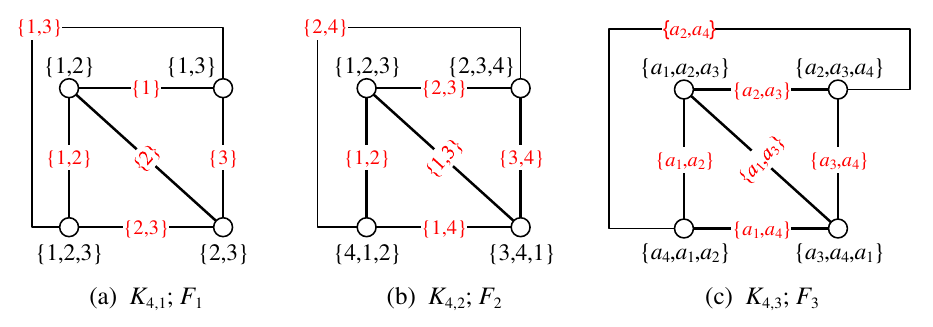}\\
\caption{\label{fig:optimal-set-matrix} {\small Three different set-colorings of the complete graph $K_4$.}}
\end{figure}

For understanding Theorem \ref{thm:theorem3535}, see three examples $K_{4,1}$, $K_{4,2}$ and $K_{4,3}$ shown in Fig.\ref{fig:optimal-set-matrix}. There are transformations $\theta_{i,j}$ holding $F_j=\theta_{i,j}(F_i)$, and $T_{code}(K_{4,j},F_j)=\theta_{i,j}(T_{code}(K_{4,i},F_i))$ for $1\leq i,j\leq 3$ and $i\neq j$.

\begin{problem}\label{qeu:444444}
Let $\mathcal{E}\big (\Lambda^2\big )=\{\mathcal{E}_i:~i\in [1,n(\Lambda)]\}$ be the hypergraph set defined in Definition \ref{defn:new-hypergraphs-sets}, so each 3I-hyperedge set $\mathcal{E}_i$ holds $\Lambda=\bigcup _{e\in \mathcal{E}_i}e$ and forms a hypergraph $\mathcal{H}^i_{yper}=(\Lambda, \mathcal{E}_i)$, which induces the ve-intersected graph $G_i$. \textbf{Find} possible connections between two hypergraphs $\mathcal{H}^i_{yper}=(\Lambda, \mathcal{E}_i)$ and $\mathcal{H}^j_{yper}=(\Lambda, \mathcal{E}_j)$ (resp. $G_i$ and $G_j$) if $i\neq j$.
\end{problem}

\begin{defn}\label{defn:set-set-distinguishing-colorings}
\textbf{Distinguishing hyperedge-set colorings.} By the proper total set-coloring $F$ defined in Definition \ref{defn:vertex-intersected-graph-hypergraph}, we have the following set-sets
\begin{equation}\label{eqa:set-set-groups}
{
\begin{split}
C_v(x,F)=&\{F(y):y\in N_{ei}(x)\}\textrm{ (\emph{local v-color set-set})};\\
C_v[x,F]=&\{F(x)\}\cup C_v(x,F)\textrm{ (\emph{closed local v-color set-set})};\\
C_e(x,F)=&\{F(xy):y\in N_{ei}(x)\}\textrm{ (\emph{local e-color set-set})};\\
C_e[x,F]=&\{F(x)\}\cup C_e(x,F)\textrm{ (\emph{closed local e-color set-set})};\\
C_{ve}(x,F)=&C_v(x,F)\cup C_e(x,F)\textrm{ (\emph{local ve-color set})};\\
C_{ve}[x,F]=&\{F(x)\}\cup C_v(x,F)\cup C_e(x,F)\textrm{ (\emph{closed local ve-color set-set})};\\
C_{ve}\{x,F\}=&\{C_e(x,F), C_e[x,F], C_v[x,F],C_{ve}[x,F]\}\textrm{ (\emph{closed local (4)-color set-set})}.
\end{split}}
\end{equation} The set-sets defined in Eq.(\ref{eqa:set-set-groups}) enables us to obtain the adjacent-type distinguishing hyperedge-set colorings of the $(p,q)$-graph $H$ as follows:
\begin{asparaenum}[\textbf{Set-set}-1. ]
\item A \emph{v-adjacent distinguishing hyperedge-set coloring} holds $C_v(u,F)\neq C_v(v,F)$ for each edge $uv\in E(H)$.
\item A \emph{closed v-adjacent distinguishing hyperedge-set coloring} holds $C_v[u,F]\neq C_v[v,F]$ for each edge $uv\in E(H)$.
\item An \emph{e-adjacent distinguishing hyperedge-set coloring} holds $C_e(u,F)\neq C_e(v,F)$ for each edge $uv\in E(H)$.
\item A \emph{closed e-adjacent distinguishing hyperedge-set coloring} holds $C_e[u,F]\neq C_e[v,F]$ for each edge $uv\in E(H)$.
\item A \textrm{ve-adjacent distinguishing hyperedge-set coloring} holds $C_{ve}(u,F)\neq C_{ve}(v,F)$ for each edge $uv\in E(H)$.
\item A \textrm{closed ve-adjacent distinguishing hyperedge-set coloring} holds $C_{ve}[u,F]\neq C_{ve}[v$, $F]$ for each edge $uv\in E(H)$.
\item A \textrm{closed $(4)$-adjacent distinguishing hyperedge-set coloring} holds $C_{ve}\{u,F\}\neq C_{ve}\{v$, $F\}$ for each edge $uv\in E(H)$.\qqed
\end{asparaenum}
\end{defn}

\begin{defn}\label{defn:parameterized-hyperedge-set-colorings}
\textbf{Parameterized total hyperedge-set coloring.} A bipartite $(p,q)$-graph $G$ admits a \emph{parameterized total hyperedge-set coloring} $\varphi: V(G)\cup E(G)\rightarrow \mathcal{E}\in \mathcal{E}\big (\Lambda_{col}^2\big )$, where the coloring set $\Lambda_{col}=\{f_1,f_2,\dots ,f_B\}$ is defined in Definition \ref{defn:more-string-total-coloring}. The set-coloring $\varphi$ satisfies the following constraints:

(1) $|\varphi(u)|=|e_u|\geq \textrm{deg}_G(u)$ for each vertex $u\in V(G)$ and $e_u\in \mathcal{E}$;

(2) $|\varphi(uv)|=|e_{uv}|\geq 1$ for each edge $uv\in E(G)$ and $e_{uv}\in \mathcal{E}$;

(3) each $f_s\in \Lambda_{col}$ is in $\varphi(w)$ for some $w\in V(G)\cup E(G)$;

(4) each pair of adjacent edges $xy$ and $xz$ with $y,z\in N_{ei}(x)$ holds $\varphi(xy)\neq \varphi(xz)$;

(5) each $a_{uv}\in \varphi(uv)$ for each edge $uv\in E(G)$ corresponds to $a_{u}\in \varphi(u)$ and $a_{v}\in \varphi(v)$, such that there is a constraint $W[f_s(a_{u}), f_s(a_{uv}), f_s(a_{v})]=0$ for some $f_s\in \Lambda_{col}$;

(6) each $b_{u}\in \varphi(u)$ (resp. $b_{v}\in \varphi(v)$) for each edge $uv\in E(G)$ corresponds to some $b_{uv}\in \varphi(uv)$ and $b_{v}\in \varphi(v)$ (resp. $b_{u}\in \varphi(u)$), such that there is a constraint $W[f_i(b_{u}), f_i(b_{uv}), f_i(b_{v})]=0$ for some $f_i\in \Lambda_{col}$;

(7) the total color set $\varphi(V(G)\cup E(G))=\Lambda_{col}$.\qqed
\end{defn}

\begin{defn} \label{defn:set-hyperedge-sets}
\cite{Yao-Ma-arXiv-2024-13354} Let $H^{yper}_{color}(G)$ be the set of hyperedge sets for all set-colorings of a connected graph $G$, such that each 3I-hyperedge set $\mathcal{E}\in H^{yper}_{color}(G)$ defines a set-coloring $F:V(G)\rightarrow \mathcal{E}$, where $\bigcup_{e\in \mathcal{E}}e=\Lambda$, and the connected graph $G$ is a ve-intersected graph of the hypergraph $\mathcal{H}_{yper}=(\Lambda,\mathcal{E})$ based on a 3I-hyperedge set $\mathcal{E}\in \mathcal{E}\big (\Lambda^2\big )$ defined in Definition \ref{defn:vertex-intersected-graph-hypergraph}. For another set-coloring $F\,':V(G)\rightarrow \mathcal{E}\,'\in H^{yper}_{color}(G)$ with $\bigcup_{e\in \mathcal{E}\,'}e=\Lambda\,'$, we defined the third set-coloring $F^*:V(G)\rightarrow \mathcal{E}^*$, such that $F^*(x)=F(x)\cup F\,'(x)$ for each vertex $x\in V(G)$ and $F^*(uv)=F(uv)\cup F\,'(uv)$ for each edge $uv\in E(G)$, and $\mathcal{E}^*=\mathcal{E}\cup \mathcal{E}\,'$ with $e\in \mathcal{E}\setminus (\mathcal{E}\cap \mathcal{E}')$ holding $e\not \subseteq e\,'\in \mathcal{E}'\setminus (\mathcal{E}\cap \mathcal{E}')$, as well as

(1) $\Lambda^*=\Lambda=\Lambda\,'$;

(2) $\Lambda^*=\Lambda \cup \Lambda\,'$ if $\Lambda \neq \Lambda\,'$. \\
Then we call the set-coloring $F^*$ \emph{union set-coloring} of two set-colorings $F$ and $F\,'$, the hyperedge set $\mathcal{E}^*$ \emph{hyperedge union set} of two 3I-hyperedge sets $\mathcal{E}$ and $\mathcal{E}\,'$. \qqed
\end{defn}

\begin{thm}\label{thm:666666}
\cite{Yao-Ma-arXiv-2024-13354} (i) Any connected graph $H$ admits a total set-coloring $f:V(H)\cup E(H)\rightarrow \mathcal{E}\in \mathcal{E}\big (\Lambda^2\big )$, such that $f(u)\cap f(v)\neq \emptyset$ for each edge $uv\in E(G)$.

(ii) A complete graph $K_n$ of $n$ vertices admits a total set-coloring $F:V(K_n)\cup E(K_n)\rightarrow \mathcal{E}\in \mathcal{E}([1,n]^2)$, such that $F(uv)\supseteq F(u)\cap F(v)\neq \emptyset$ for each edge $uv\in E(K_n)$, and the 3I-hyperedge set $\mathcal{E}$ is \emph{proper}.
\end{thm}

\begin{problem}\label{qeu:444444}
The set $H^{yper}_{color}(G)$ defined in Definition \ref{defn:set-hyperedge-sets} can be classified into two parts
\begin{equation}\label{eqa:555555}
H^{yper}_{color}(G)=I^{yper}_{color}(G)\bigcup N^{yper}_{color}(G)
\end{equation} such that each hyperedge set of $I^{yper}_{color}(G)$ (as a \emph{public-key set}) is not the union set of any two sets of the set $H^{yper}_{color}(G)$ (as a \emph{private-key set}), but each hyperedge set of $N^{yper}_{color}(G)$ is the hyperedge union set of some two hyperedge sets of $H^{yper}_{color}(G)$. \textbf{Determine} the hyperedge set $I^{yper}_{color}(G)$ for a connected graph $G$.
\end{problem}

Motivated from Definition \ref{defn:set-hyperedge-sets}, we have the following the union operation of hyperedge sets:

\begin{prop}\label{prop:hyperedge-sets-unions}
\textbf{The union operation of hyperedge sets.} Let two finite sets $\Lambda_a$ (as a \emph{public-ky set}) and $\Lambda_b$ (as a \emph{private-ky set}) hold $\Lambda_a\cap \Lambda_b\neq \emptyset$, and there are two 3I-hyperedge sets $\mathcal{E}_a\in \mathcal{E}(\Lambda_a^2)$ and $\mathcal{E}_b\in \mathcal{E}(\Lambda_b^2)$. Then, there is a new 3I-hyperedge set $\mathcal{E}^*=\mathcal{E}_a\cup \mathcal{E}_b$ (as an \emph{authentication set}) holding the new finite set $\Lambda=\bigcup_{e\in \mathcal{E}^*}e$ true, where $\Lambda=\Lambda_a\cup \Lambda_b$.
\end{prop}

\begin{rem}\label{rem:333333}
By Proposition \ref{prop:hyperedge-sets-unions}, there is some 3I-hyperedge set $\mathcal{E}_k\in \mathcal{E}\big (\Lambda^2\big )$ holding the 3I-hyperedge set $\mathcal{E}_k\setminus \mathcal{E}^*_k=\mathcal{E}_k\,''\in \mathcal{E}(\Lambda_k^2)$ with $k=a,b$, such that the 3I-hyperedge set $\mathcal{E}_k= \mathcal{E}^*_k\cup \mathcal{E}_k\,''$, however, not necessarily the set $ \mathcal{E}^*_k\in \mathcal{E}(\Lambda_k^2)$, even the set $\mathcal{E}^*_k$ is not a hyperedge set based on two finite sets $\Lambda_a$ and $\Lambda_b$.

(i) If the finite set $\Lambda_a$ is a proper subset of the finite set $\Lambda_b$, namely $\Lambda_a\subset \Lambda_b$, then the result of Proposition \ref{prop:hyperedge-sets-unions} is still valid.

(ii) If $\Lambda_a\cap \Lambda_b=\emptyset$ in Proposition \ref{prop:hyperedge-sets-unions}, then the \emph{union hyperedge set} $\mathcal{E}= \mathcal{E}\,'\cup \mathcal{E}\,''$ with two 3I-hyperedge sets $\mathcal{E}\,'\in \mathcal{E}(\Lambda_a^2)$ and $\mathcal{E}\,''\in \mathcal{E}(\Lambda_b^2)$ forms a \emph{hyperedge dis-connected hypergraph}.\qqed
\end{rem}

\subsection{Connectivity of hypergraphs}

Suppose that $G$ is a hyperedge-connected ve-intersected graph of a hyperedge-connected hypergraph $\mathcal{H}_{yper}=(\Lambda,\mathcal{E})$ based on a 3I-hyperedge set $\mathcal{E}\in \mathcal{E}\big (\Lambda^2\big )$ (Ref. Definition \ref{defn:more-terminology-group}), so $G$ admits a total set-coloring $F:V(G)\cup E(G)\rightarrow \mathcal{E}$. We define a total set-coloring $F^*$ for the vertex-split graph $G\wedge S$ for a proper subset $S\subset V(G)$ defined in Definition \ref{defn:split-v-set-vertex-split-graphss} as follows:

(A-1) $F^*(x)=F(x)$ for $x\in V(G)\cup E(G)$ and $x\not\in S\,'\cup S\,''$, as well as $F^*(w\,')=F(w)$ and $F^*(w\,'')=F(w)$ for $w\in S$ and $w\,',w\,''\in S\,'\cup S\,''$; and

(A-2) $F^*(uv)=F(uv)$ for $uv\in E(G-S)$, $F^*(uw\,')=F(uw)$ and $F^*(uw\,'')=F(uw)$ for edges $uw\,'$, $uw\,''\in E(G\wedge S)$.

Suppose that the vertex-split graph $G\wedge S$ has subgraphs $G_1,G_2,\dots ,G_m$, such that

(i) $V(G\wedge S)=\bigcup ^m_{i=1}V(G_i)$ with $V(G_i)\cap V(G_j)=\emptyset$ if $i\neq j$, and

(ii) $E(G\wedge S)=\bigcup ^m_{i=1}E(G_i)$ with $E(G_i)\cap E(G_j)=\emptyset$ if $i\neq j$,\\
then the vertex-split graph $G\wedge S$ is not hyperedge-connected, also, the vertex-split graph $G\wedge S$ is disconnected in general, and we call $S$ a \emph{vertex-split set-cut-set} of the graph $G$. For each subgraph $G_i$, the set-coloring $F$ induces a total set-coloring $F_i:V(G_i)\cup E(G_i)\rightarrow \mathcal{E}_i$ with $\mathcal{E}_i\subset \mathcal{E}$ with $i\in [1,m]$, and call each subgraph $G_i$ \emph{partial hypergraph}.

Conversely, we do the vertex-coinciding operation defined in Definition \ref{defn:vertex-split-coinciding-operations} to the subgraphs $G_1,G_2,\dots ,G_m$ of the vertex-split graph $G\wedge S$ by vertex-coinciding two vertices $w\,'$ and $w\,''$ into one vertex $w=w\,'\bullet w\,''$, and get the original hyperedge-connected ve-intersected graph $G$, we write this case as $G=[\bullet]^m_{k=1} G_k$.

In the view of homomorphism, we get that a group of graphs $G_1,G_2,\dots ,G_m$ above is graph homomorphism to $G$, denoted as $(G_1,G_2,\dots ,G_m)\rightarrow G$, namely, $G\wedge S\rightarrow G$.

\begin{defn} \label{defn:hypergraph-connectivity}
\cite{Yao-Ma-arXiv-2024-13354} \textbf{Hypergraph connectivity.} Suppose that a vertex-split graph $G\wedge S^*$ of the hyperedge-connected ve-intersected graph $G$ of a hypergraph $\mathcal{H}_{yper}=(\Lambda,\mathcal{E})$ based on a 3I-hyperedge set $\mathcal{E}\in \mathcal{E}\big (\Lambda^2\big )$ holds $|S^*|\leq |S|$ for any vertex-split graph $G\wedge S$, where $G\wedge S$ is not hyperedge-connected, then the number $|S^*|$ is called \emph{hyperedge split-connected number}, written as $|S^*|=n_{vsplit}(G)=n_{vsplit}(\mathcal{H}_{yper})$. Since $G-S^*$ is a disconnected graph having components $G_1-S^*,G_2-S^*,\dots ,G_m-S^*$, that is, the hyperedge-connected ve-intersected graph $G$ is \emph{vertex $|S^*|$-connectivity}, and the hypergraph $\mathcal{H}_{yper}=(\Lambda,\mathcal{E})$ is \emph{hyperedge $|\mathcal{E}^*|$-connectivity}, where $\mathcal{E}^*=\{e:F(w)=e\in \mathcal{E}, w\in S^*\}$ makes the hyperedge set $\mathcal{E}\setminus \mathcal{E}^*$ to be \emph{hyperedge disconnected}, we call the hyperedge set $\mathcal{E}^*$ \emph{hyperedge set-cut-set} of the hypergraph $\mathcal{H}_{yper}=(\Lambda,\mathcal{E})$.\qqed
\end{defn}

\begin{rem}\label{rem:333333}
In the view of decomposition, the hyperedge-connected ve-intersected graph $G$ can be decomposed into vertex disjoint partial hypergraphs $G_1,G_2,\dots ,G_m$.

For $w\in S$ and $w\,',w\,''\in S\,'\cup S\,''$ in the case (A-1) above, we redefine $F^*(w\,')=F(uw)$ and $F^*(w\,'')=F(w)\setminus F^*(w\,')$, since $F(uw)\subseteq F(w)\cap F(u)$. Thereby, we get more families of subgraphs like the above graphs $G_1,G_2,\dots ,G_m$, in other words, a hyperedge-connected hypergraph can be decomposed into many groups of hyperedge disjoint partial hypergraphs. \qqed
\end{rem}

As known, the vertex-splitting connectivity of a connected graph is equivalent to its own vertex connectivity proven in \cite{Wang-Su-Yao-2021-computer-science}, so we have the following result:

\begin{thm}\label{thm:666666}
The hyperedge splitting connectivity of a hyperedge-connected hypergraph is equal to its own hyperedge connectivity.
\end{thm}

\subsection{Colorings of hypergraphs}

The authors in \cite{Weifan-Wang-Kemin-Zhang-2000} presented a survey about the colorings of hypergraphs in the thirty years before 2000.

\begin{defn} \label{defn:join-type-edge-set-coloring}
\cite{Yao-Ma-arXiv-2024-13354} Let $\mathcal{E}$ be a set of some subsets of the power set $\Lambda^2$ based on a finite set $\Lambda$ such that each hyperedge $e\in \mathcal{E}$ satisfies $e\neq \emptyset$ and corresponds to some hyperedge $e\,'\in \mathcal{E} \setminus e$ holding $e\cap e\,'\neq \emptyset$, as well as $\Lambda=\bigcup _{e\in \mathcal{E}}e$. Suppose that a connected graph $H$ admits an \emph{edge set-coloring} $F\,': E(H)\rightarrow \mathcal{E}$ holding $F\,'(uv)\neq F\,'(uw)$ for any two adjacent edges $uv,uw\in E(H)$, and the vertex color set $F\,'(w)$ for each vertex $w\in V(G)$ is induced by one of the following cases:
\begin{asparaenum}[\textbf{\textrm{Edgeinduce}}-1.]
\item $F\,'(w)=\{F\,'(wz):z\in N_{ei}(w)\}\subseteq \Lambda^2$.
\item $F\,'(w)=\bigcup _{z\in N_{ei}(w)}F\,'(wz)\subseteq \Lambda$.
\end{asparaenum}
We call $H$ \emph{edge-set-colored graph}, and have:

(i) The \emph{edge-induced graph} $L_H$ of the edge-set-colored graph $H$ has its own vertex set $V(L_H)= E(H)$, and admits a vertex set-coloring $F\,': V(L_H)\rightarrow \mathcal{E}$, such that two vertices $w_{uv}(:=uv)$ and $w_{xy}(:=xy)$ of $V(L_H)$ are the ends of an edge of $L_H$, if, and only if, $F\,'(w_{uv})\cap F\,'(w_{xy})\neq \emptyset$ (i.e. $F\,'(uv)\neq F\,'(xy)$ in $H$).

(ii) The edge-induced graph $L_H$ is a \emph{ve-intersected graph} of the hypergraph $\mathcal{H}_{yper}=(\Lambda,\mathcal{E})$ based on a 3I-hyperedge set $\mathcal{E}\in \mathcal{E}\big (\Lambda^2\big )$.\qqed
\end{defn}

\begin{thm}\label{thm:hyperedge-vs-vertex-coloring}
A proper hyperedge coloring of a hypergraph $\mathcal{H}_{yper}=(\Lambda,\mathcal{E})$ based on a 3I-hyperedge set $\mathcal{E}\in \mathcal{E}\big (\Lambda^2\big )$ is equivalent to a proper vertex-coloring of a ve-intersected graph $G$ of the hypergraph $\mathcal{H}_{yper}=(\Lambda,\mathcal{E})$, and vice versa. Thereby, we have $\chi(G)=\chi\,'(\Lambda,\mathcal{E})$, where $\chi(G)$ is the \emph{chromatic number} of the graph $G$, and $\chi\,'(\Lambda,\mathcal{E})$ is the \emph{hyperedge-chromatic index} of the hypergraph $\mathcal{H}_{yper}=(\Lambda,\mathcal{E})$.
\end{thm}

\begin{rem}\label{rem:333333}
In Definition \ref{defn:join-type-edge-set-coloring}, an edge-set-colored graph $H$ may be a subgraph of a ve-intersected graph of a hypergraph $\mathcal{H}_{yper}=(\Lambda,\mathcal{E}^*)$ with $\mathcal{E}^*\neq \mathcal{E}$ because of \textbf{Edgeinduce}-1 and \textbf{Edgeinduce}-2 defined in Definition \ref{defn:join-type-edge-set-coloring}. However, the \emph{edge-induced graph} $L_H$ is not the \emph{line graph} of the edge-set-colored graph $H$. Theorem \ref{thm:hyperedge-vs-vertex-coloring} tells us: The proper hyperedge coloring problem of a hypergraph is a NP-type problem, since there is a well-known conjecture on the proper vertex-colorings of graphs, that is, Bruce Reed in 1998 conjectured: The \emph{chromatic number} $\chi (G)\leq \left \lceil \frac{\Delta(G)+1+K(G)}{2}\right \rceil $, where $\Delta(G)$ is the \emph{maximum degree} of the graph $G$ and $K(G)$ is the \emph{maximum clique number} of the graph $G$. This conjecture is still unsolved today.\qqed
\end{rem}

\begin{defn} \label{defn:hyperedge-hyper-total-coloring}
\cite{Yao-Ma-arXiv-2024-13354} A \emph{hyper-total coloring} $\theta$ of a hypergraph $\mathcal{H}_{yper}=(\Lambda,\mathcal{E})$ based on a 3I-hyperedge set $\mathcal{E}\in \mathcal{E}\big (\Lambda^2\big )$ (Ref. Definition \ref{defn:new-hypergraphs-sets}) is defined by

(i) $\theta:\mathcal{E}\rightarrow [1,b]$, and $\theta(e_i)\neq \theta(e_j)$ if $e_i\cap e_j\neq \emptyset$;

(ii) each vertex color $\theta(x_{i,j})\in [1,b]$, and $\theta(x_{i,j})\neq \theta(x_{i,k})$ for some two distinct vertices $x_{i,j},x_{i,k}\in e_i$ if $|e_i|\geq 2$.

And $\chi\,''(\Lambda,\mathcal{E})$ is the \emph{smallest number} of $b$ for which $\mathcal{H}_{yper}$ admits a hyper-total coloring.\qqed
\end{defn}

\begin{rem}\label{rem:333333}
For a \emph{hyperedge coloring} $\varphi: \mathcal{E}\rightarrow [1,M]$, such that the elements of the hyperedge neighbor set $N_{ei}(e_i)=\big \{e_j: e_i\cap e_j\neq \emptyset,e_j\in \mathcal{E}\setminus e_i \big \}$ of a hyperedge $e_i$ are colored with different colors from each other, and the largest number $\Delta(\mathcal{E}_{\cap})=\max\{|N_{ei}(e_i)|:e_i\in \mathcal{E}\}$ holds the following inequalities
\begin{equation}\label{eqa:555555}
\Delta(\mathcal{E}_{\cap})\leq M\leq \Delta(\mathcal{E}_{\cap})+1
\end{equation} true by the famous Vizing's theorem on the edge colorings of ordinary graphs.\qqed
\end{rem}

For a 3I-hyperedge set $\mathcal{E}\in \mathcal{E}\big (\Lambda^2\big )$, we color the vertices of the vertex set $\Lambda$ with $k$ colors such that each hyperedge $e\in \mathcal{E}$ with $|e|\geq 2$ contains two vertices colored with different colors. Clearly, different 3I-hyperedge sets correspond to different vertex-colorings of the vertex set $\Lambda$.

\begin{defn} \label{defn:111111}
A proper vertex $k$-coloring of a hypergraph $\mathcal{H}_{yper}=(\Lambda,\mathcal{E})$ based on a 3I-hyperedge set $\mathcal{E}\in \mathcal{E}\big (\Lambda^2\big )$ is a mapping $F:\Lambda\rightarrow [0,k-1]$, such that $F$ does not create monochromatic hyperedges (i.e. every hyperedge $e$ contains at least two vertices, namely, $|e|\geq 2$, which receive different colors). The number
\begin{equation}\label{eqa:555555}
\chi(\Lambda,\mathcal{E})=\min \{k:\textrm{ each $k$-coloring of $\Lambda$ based on a 3I-hyperedge set $\mathcal{E}$}\}
\end{equation} is called \emph{hypervertex chromatic number} of the hypergraph $\mathcal{H}_{yper}=(\Lambda,\mathcal{E})$.\qqed
\end{defn}

The authors in \cite{Bjorklund-Husfeldt-Koivisto-2009} pointed: The problem of deciding a given $k$-uniform hypergraph to be $k$-colorable is NP-complete.

\subsection{Homomorphism and isomorphism of hypergraphs}

\subsubsection{Hypergraph homomorphism}

\begin{defn} \label{defn:hypergraph-operation-homomorphisms}
\textbf{Hypergraph homomorphism.} For two 3I-hyperedge sets $\mathcal{E}_i$ and $\mathcal{E}_j$ of the hypergraph set $\mathcal{E}\big (\Lambda^2\big )$ based on a finite set $\Lambda$, we have:

(i) If there exists a coloring $\varphi:\mathcal{E}_i\rightarrow \mathcal{E}_j$ such that each hyperedge $e_{i,s}\in \mathcal{E}_i$ corresponds to its own image $\varphi(e_{i,s})\in \mathcal{E}_j$, and moreover two hyperedges $e_{i,s}$ and $e_{i,t}$ hold a property $P$ in $\mathcal{E}_i$ if, and only if two hyperedges $\varphi(e_{i,s})$ and $\varphi(e_{i,t})$ of the 3I-hyperedge set $\mathcal{E}_j$ hold this property $P$ too, we say that $\mathcal{E}_i$ is \emph{$P$-property hypergraph homomorphism} to $\mathcal{E}_j$, denoted as $\mathcal{E}_i\rightarrow _P\mathcal{E}_j$.

(ii) For an operation ``$[\bullet]$'' on the hypergraph set $\mathcal{E}\big (\Lambda^2\big )$, we have a \emph{hypergraph $[\bullet]$-operation homomorphism} $(\mathcal{E}_i, \mathcal{E}_j)\rightarrow \mathcal{E}_{i[\bullet] j}$ if $\mathcal{E}_i[\bullet] \mathcal{E}_j=\mathcal{E}_{i[\bullet] j}\in \mathcal{E}\big (\Lambda^2\big )$. \qqed
\end{defn}

\begin{problem}\label{qeu:444444}
As the operation ``$[\bullet]=\bigcup$'' in Definition \ref{defn:hypergraph-operation-homomorphisms}, we have a hypergraph $\bigcup$-operation homomorphism $(\mathcal{E}_i, \mathcal{E}_j)\rightarrow \mathcal{E}_{i\cup j}$ since $\mathcal{E}_{i\cup j}=\mathcal{E}_i\cup \mathcal{E}_j$. However, it is useful to find more hypergraph $[\bullet]$-operation homomorphisms.
\end{problem}

\begin{defn} \label{defn:111111}
By Definition \ref{defn:hypergraph-basic-definition} and Definition \ref{defn:hypergraph-operation-homomorphisms}, we define the following colorings:

(i) A graph $G$ admits a \emph{hypergraph-set total coloring} $\beta:V(G)\cup E(G)\rightarrow \mathcal{E}\big (\Lambda^2\big )$ if $\beta(uv)=\beta(u)\cup \beta(v)$ for each edge $uv\in E(G)$ holds the hypergraph $\bigcup$-operation homomorphism
\begin{equation}\label{eqa:555555}
(\beta(u), \beta(v))=(\mathcal{E}_i, \mathcal{E}_j)\rightarrow \mathcal{E}_{i\cup j}=\beta(uv)
\end{equation}

(ii) A graph $G$ admits a \emph{hypergraph-set vertex coloring} $\theta:V(G)\rightarrow \mathcal{E}\big (\Lambda^2\big )$ if each edge $uv\in E(G)$ holds a $P$-property hypergraph homomorphism $\theta(u)=\mathcal{E}_i\rightarrow _P\mathcal{E}_j=\theta(v)$ true.\qqed
\end{defn}

\begin{defn}\label{defn:definition-graph-homomorphism}
\cite{Bondy-2008} A \emph{graph homomorphism} $G\rightarrow H$ from a graph $G$ into another graph $H$ is a coloring $f: V(G) \rightarrow V(H)$ such that each edge $f(u)f(v)\in E(H)$ if, and only if each edge $uv\in E(G)$.\qqed
\end{defn}

\begin{defn} \label{defn:W-constraint-coloring-graph-homomorphism}
A \emph{$W$-constraint colored graph homomorphism} $G\rightarrow _{color}H$ is defined as: A graph $G$ admits a $W$-constraint coloring $F$ and another graph $H$ admits another $W$-constraint coloring $F^*$, and there is a graph homomorphism $\varphi:V(G)\rightarrow V(H)$, such that the $W$-constraint $W[F(u),F(uv)$, $F(v)]=0$ holds true if, and only if the $W$-constraint $W[F^*(\varphi(u))$, $F^*(\varphi(u)\varphi(v))$, $F^*(\varphi(v))]=0$ holds true.\qqed
\end{defn}

Motivated from Definition \ref{defn:definition-graph-homomorphism} and Definition \ref{defn:W-constraint-coloring-graph-homomorphism}, we have:

\begin{defn} \label{defn:111111}
Let $G[\mathcal{E}]$ be a set-colored graph admitting a set-coloring $F:V(G)\rightarrow \mathcal{E}$ on a graph $G$ and a hypergraph $\mathcal{H}_{yper}=(\Lambda,\mathcal{E})$ based on a 3I-hyperedge set $\mathcal{E}\in \mathcal{E}\big (\Lambda^2\big )$, and let $H[\mathcal{E}^*]$ be a set-colored graph admitting a set-coloring $F^*:V(H)\rightarrow \mathcal{E}^*$ on another graph $H$ and another hypergraph $\mathcal{H}_{yper}=(\Lambda_*,\mathcal{E}^*)$ with a 3I-hyperedge set $\mathcal{E}^*\in \mathcal{E}\big (\Lambda^2_*\big )$.

Since $F(uv)=F(u)[\bullet]F(v)$ if, and only if $F^*(\varphi(u)\varphi(v))=F^*(\varphi(u))[\bullet]F^*(\varphi(v))$ under the coloring $\varphi: \Lambda \rightarrow \Lambda_*$ and an operation ``$[\bullet]$'', we have defined a \emph{set-colored graph homomorphism} $G[\mathcal{E}]\rightarrow H[\mathcal{E}^*]$, and a \emph{hypergraph homomorphism} $(\Lambda,\mathcal{E})\rightarrow (\Lambda_*,\mathcal{E}^*)$.\qqed
\end{defn}

\begin{example}\label{exa:8888888888}
Observe Fig.\ref{fig:v-split-homomorphism} and Fig.\ref{fig:v-split-homomorphism-1}, we have four set-colored graph homomorphisms $S_k[\mathcal{E}_k]\rightarrow S_{k-1}[\mathcal{E}_{k-1}]$ for $k\in [1,4]$, where the set-colored graph $S_0=S$ shown in Fig.\ref{fig:v-split-homomorphism-2}. Conversely, each set-colored graph $S_i$ is a result of doing the vertex-splitting operation to $S_{i-1}$ for $i\in [1,4]$. More or less, we have shown the \emph{hyperedge-splitting operation} and the \emph{hyperedge-coinciding operation} of hypergraphs shown in Fig.\ref{fig:v-split-homomorphism}, Fig.\ref{fig:v-split-homomorphism-1} and Fig.\ref{fig:v-split-homomorphism-2}.\qqed
\end{example}

\begin{figure}[h]
\centering
\includegraphics[width=16.4cm]{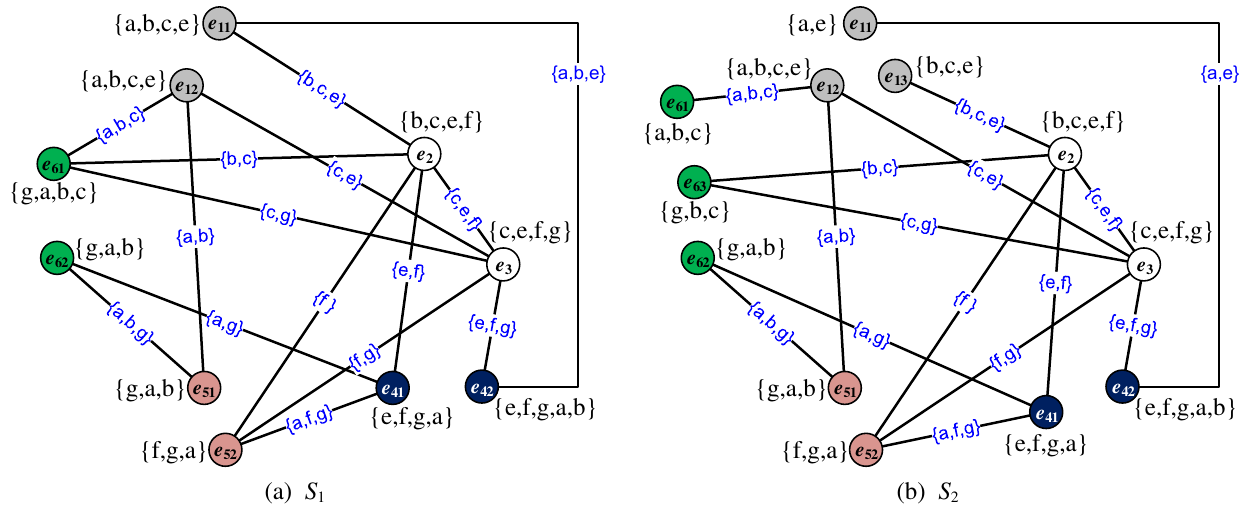}\\
\caption{\label{fig:v-split-homomorphism}{\small The first scheme for illustrating the hyperedge-splitting operation and the hyperedge-coinciding operation.}}
\end{figure}

\begin{defn} \label{defn:hyperedge-coinciding-splitting-homomorphism}
\textbf{Hyperedge set homomorphism.} Let $\Lambda$ be a finite set. If there are two hyperedges $e_1,e_2\in \mathcal{E}^*\in \mathcal{E}\big (\Lambda^2\big )$ holding $|e_1\cup e_2|=|e_1|+|e_2|$, then we get a \emph{hyperedge set homomorphism} $\mathcal{E}^*\rightarrow \mathcal{E}$, where $e_1\cup e_2\in \mathcal{E}\in \mathcal{E}\big (\Lambda^2\big )$, and $\mathcal{E}^*\setminus \{e_1,e_2\}=\mathcal{E}\setminus \{e_1\cup e_2\}$, this operation process is called \emph{hyperedge-coinciding operation} on 3I-hyperedge sets. Conversely, we split the hyperedge $e_1\cup e_2\in \mathcal{E}$ into two hyperedges $e_1$ and $e_2$ to obtain the 3I-hyperedge set $\mathcal{E}^*$, this operation process is called \emph{hyperedge-splitting operation} on 3I-hyperedge sets.\qqed
\end{defn}

\begin{rem}\label{rem:333333}
If $e_{\cap}=e_1\cap e_2\neq \emptyset $ in Definition \ref{defn:hyperedge-coinciding-splitting-homomorphism}, we hyperedge-split the hyperedge $e_1\cup e_2\in \mathcal{E}$ into two hyperedges $e_1=(e_1\cup e_2)\setminus (e_2\cup e_{\cap})$ and $e_2=(e_1\cup e_2)\setminus (e_1\cup e_{\cap})$. In other words, doing the hyperedge-splitting operation to a hyperedge set $\mathcal{E}$ can obtain two or more hyperedge sets $\mathcal{E}^*$ holding $\mathcal{E}^*\rightarrow \mathcal{E}$ true. For understanding the above hyperedge-splitting operation, refer to Fig.\ref{fig:v-split-homomorphism}, Fig.\ref{fig:v-split-homomorphism-1} and Fig.\ref{fig:v-split-homomorphism-2}, and, by the hyperedge-splitting operation and the hyperedge-coinciding operation, we have the following \emph{hyperedge set homomorphisms}:
$$
S_4\rightarrow _{hyperedge}S_3\rightarrow _{hyperedge}S_2\rightarrow _{hyperedge}S_1\rightarrow _{hyperedge}S
$$ also, each $S_k$ for $k\in [1,4]$ holds $S_k\rightarrow _{hyperedge}S$.\qqed
\end{rem}

\begin{figure}[h]
\centering
\includegraphics[width=16.4cm]{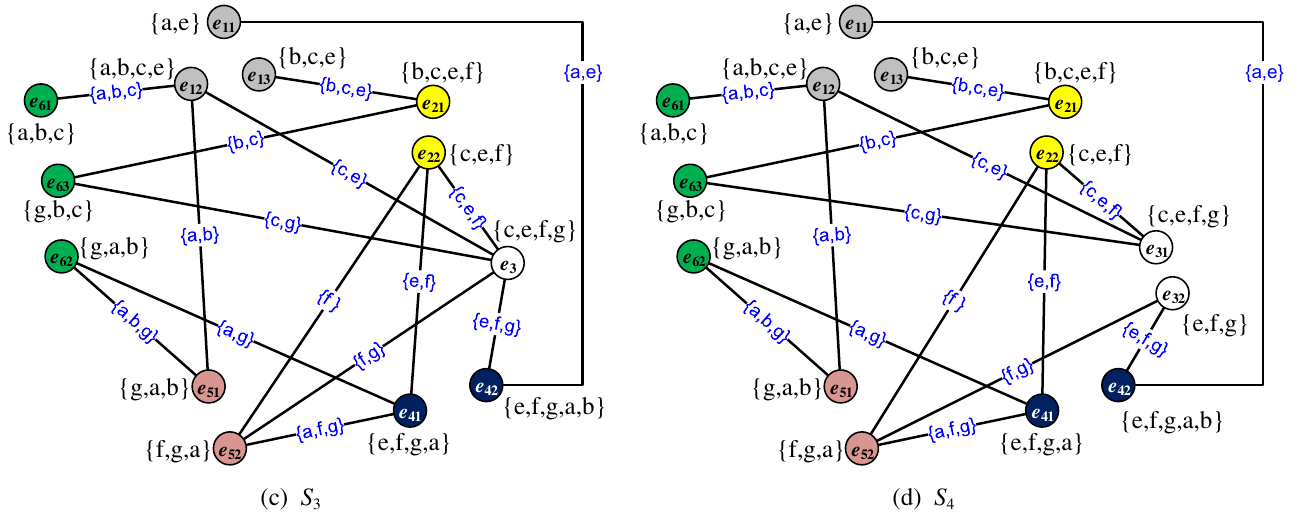}\\
\caption{\label{fig:v-split-homomorphism-1}{\small The second scheme for illustrating the hyperedge-splitting operation and the hyperedge-coinciding operation.}}
\end{figure}

\begin{thm}\label{thm:666666}
\textbf{Hypergraph homomorphism.} Let $\mathcal{E}\big (\Lambda^2\big )$ be the hypergraph set defined on a finite set $\Lambda$. If each 3I-hyperedge set $\mathcal{E}\in \mathcal{E}\big (\Lambda^2\big )$ with some hyperedge $e\in \mathcal{E}$ holding $|e|\geq 2$ is \emph{hyperedge set homomorphism} to another 3I-hyperedge set $\mathcal{E}^*\in \mathcal{E}\big (\Lambda^2\big )$, namely, $\mathcal{E}\rightarrow \mathcal{E}^*$ for $|\mathcal{E}|\geq 1+|\mathcal{E}^*|$, then we have a \emph{hypergraph homomorphism} $(\Lambda,\mathcal{E})\rightarrow (\Lambda,\mathcal{E}^*)$.
\end{thm}

\begin{figure}[h]
\centering
\includegraphics[width=10cm]{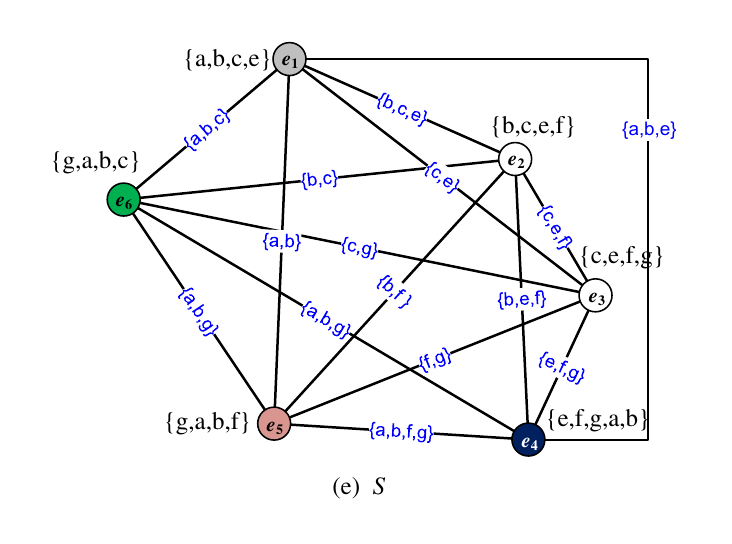}\\
\caption{\label{fig:v-split-homomorphism-2}{\small The third scheme for illustrating the hyperedge-splitting operation and the hyperedge-coinciding operation.}}
\end{figure}

\subsubsection{Isomorphisms between hypergraphs}

Recall the following famous conjecture :
\begin{conj}\label{conj:0000000000}
(Kelly-Ulam's Reconstruction Conjecture (1942), \cite{Bondy-2008}) Let $G$ and $H$ be graphs with $n$ vertices. If there is a bijection $f: V(G)\rightarrow V(H)$ such that two vertex-removing graphs $G-u \cong H-f(u)$ for each vertex $u\in V(G)$, then $G \cong H$.
\end{conj}

\begin{thm}\label{thm:2-vertex-split-graphs-isomorphic}
\cite{Yao-Wang-2106-15254v1} Suppose that two connected graphs $G$ and $H$ admit a coloring $h:V(G)\rightarrow V(H)$. In general, a vertex-split graph $G\wedge u$ with $\textrm{deg}_G(u)\geq 2$ is not unique, so we have a vertex-split graph set $S_G(u)=\{G\wedge u\}$, naturally, we have another vertex-split graph set $S_H(h(u))=\{H\wedge h(u)\}$. If each vertex-split graph $L\in S_G(u)$ corresponds to another vertex-split graph $T\in S_H(h(u))$ holding $L\cong T$ true, and vice versa, we write this fact as
\begin{equation}\label{eqa:555555}
G\wedge u\cong H\wedge h(u)
\end{equation} then we claim that $G$ is \emph{isomorphic} to $H$, namely, $G\cong H$.
\end{thm}

\begin{rem}\label{rem:333333}
The vertex-splitting graph $G\wedge u$ with degree $\textrm{deg}_G(u)\geq 2$ forms a vertex-split graph set $S_G(u)=\{G\wedge u\}$ in Theorem \ref{thm:2-vertex-split-graphs-isomorphic}. However, determining this graph set $S_G(u)=\{G\wedge u\}$ will meet the Integer Partition Problem (IPP).

For the computational complexity of IPP, the authors, in \cite{Bing-et-al-arXiv-asymmetric-4520331}, have partitioned a positive integer $m\geq 2$ into a group of $a_i$ parts $m_{i,1},m_{i,2},\dots ,m_{i,a_i}$ holding $m=m_{i,1}+m_{i,2}+\cdots +m_{i,a_i}$ with each $m_{i,j}>0$ and $a_i\geq 2$. Correspondingly, the vertex $u$ of the graph $G$ is vertex-split into vertices $u_{i,1},u_{i,2},\dots ,u_{i,a_i}$, such that the adjacent neighbor set $N_{ei}(u)=\bigcup^{a_i}_{j=1}N_{ei}(u_{i,j})$ in the vertex-splitting graph $G\wedge u$, where two adjacent neighbor sets $N_{ei}(u_{i,j})\cap N_{ei}(u_{i,k})=\emptyset$ for $j\neq k$. Suppose there are $P_{art}(m)$ groups of such $a_i$ parts. Computing the number $P_{art}(m)$ can be transformed into finding the number $A(m,a_i)$ of solutions of \emph{Diophantine equation} $m=\sum ^k_{i=1}ix_i$. There is a recursive formula
\begin{equation}\label{eqa:integer-partition-problem}
A(m,a_i)=A(m,a_i-1)+A(m-a_i,a_i)
\end{equation}
with $0 \leq a_i\leq m$. It is not easy to compute the exact value of $A(m,a_i)$ defined in Eq.(\ref{eqa:integer-partition-problem}), for example, the authors in \cite{Shuhong-Wu-Accurate-2007} and \cite{WU-Qi-qi-2001} computed exactly
\begin{equation}\label{eqa:555555}
{
\begin{split}
A(m,6)=&\biggr\lfloor \frac{1}{1036800}(12m^5 +270m^4+1520m^3-1350m^2-19190m-9081)+\\
&\frac{(-1)^m(m^2+9m+7)}{768}+\frac{1}{81}\left[(m+5)\cos \frac{2m\pi}{3}\right ]\biggr\rfloor
\end{split}}
\end{equation}

On the other hands, for any odd integer $m\geq 7$, it was conjectured $m=p_1+p_2+p_3$ with three primes $p_1,p_2,p_3$ from the famous Goldbach's conjecture: \emph{Every even integer, greater than 2, can be expressed as the sum of two primes}. In other words, computing $A(m,3)$ is not a slight work, also, it is difficult to express an odd integer $m=\sum^{3n}_{k=1} p\,'_k$ with each $p\,'_k$ is a prime integer.\qqed
\end{rem}

\begin{thm}\label{thm:hypergraphs-isomorphic}
Suppose that $G_i$ is a v-intersected graph of the hypergraph $\mathcal{H}^i_{yper}=(\Lambda,\mathcal{E}_i)$ with $i=1,2$, and $|\mathcal{E}_1|=|\mathcal{E}_2|$, as well as $G_i-e$ for any edge $e\in E(G_i)$ is not a v-intersected graph of the hypergraph $\mathcal{H}^i_{yper}=(\Lambda,\mathcal{E}_i)$ for $i=1,2$. If $G_1\cong G_2$, then $(\Lambda,\mathcal{E}_1)=\mathcal{H}^1_{yper}\cong \mathcal{H}^2_{yper}=(\Lambda,\mathcal{E}_2)$.
\end{thm}

Motivated from Theorem \ref{thm:2-vertex-split-graphs-isomorphic} and Theorem \ref{thm:hypergraphs-isomorphic}, we have the following result:

\begin{thm}\label{thm:hypergraphs-vetex-split-isomorphic}
\cite{Yao-Ma-arXiv-2024-13354} \textbf{Hypergraph isomorphism.} Suppose that two connected graphs $G$ and $H$ admit a coloring $f:V(G)\rightarrow V(H)$, where $G$ is a ve-intersected graph of a hypergraph $\mathcal{H}_{yper}=(\Lambda,\mathcal{E})$ based on a 3I-hyperedge set $\mathcal{E}\in \mathcal{E}\big (\Lambda^2\big )$, and $H$ is a ve-intersected graph of another hypergraph $\mathcal{H}^*_{yper}=(\Lambda^*,\mathcal{E}^*)$ based on a 3I-hyperedge set $\mathcal{E}^*\in \mathcal{E}\big ((\Lambda^*)^2\big )$. Vertex-splitting a vertex $u$ of the ve-intersected graph $G$ with $\textrm{deg}_G(u)\geq 2$ produces a vertex-split graph set $I_G(u)=\{G\wedge u\}$ (resp. a hypergraph set $\{\mathcal{H}_{yper}\wedge e\}$ with $e=F(u)$, since $F:V(G)\rightarrow \mathcal{E}$). Similarly, another vertex-split graph set $I_H(f(u))=\{H\wedge f(u)\}$ is obtained by vertex-splitting a vertex $f(u)$ of the ve-intersected graph $H$ with $\textrm{deg}_H(f(u))\geq 2$, so we have a hypergraph set $\{\mathcal{H}^*_{yper}\wedge e\,'\}$ with $e\,'=F\,'(f(u))$, since $F\,':V(H)\rightarrow \mathcal{E}^*$. If each vertex-split graph $L\in I_G(u)$ (resp. $\mathcal{L}\in \{\mathcal{H}_{yper}\wedge e\}$) corresponds to another vertex-split graph $T\in I_H(f(u))$ (resp. $\mathcal{T}\in \{\mathcal{H}^*_{yper}\wedge e\,'\}$) such that $L\cong T$ (resp. $\mathcal{L}\cong \mathcal{T}$), and vice versa, we write this fact as
\begin{equation}\label{eqa:555555}
G\wedge u\cong H\wedge f(u)~ (\textrm{resp.} ~\mathcal{H}_{yper}\wedge e\cong \mathcal{H}^*_{yper}\wedge e\,')
\end{equation} then we claim that $G\cong H$ (resp. $\mathcal{H}_{yper}\cong \mathcal{H}^*_{yper}$).
\end{thm}

\begin{prop}\label{proposition:99999}
\textbf{Hypergraph isomorphism.} For tow $k$-uniform hypergraphs $\mathcal{H}_{yper}=(\Lambda,\mathcal{E})$ and $\mathcal{H}'_{yper}=(\Lambda',\mathcal{E}')$, if there is a bijection $\theta:\Lambda\rightarrow \Lambda'$ such that the bijection $\theta$ induces a bijection $\varphi:\mathcal{E}\rightarrow \mathcal{E}'$, then these tow $k$-uniform hypergraphs are isomorphic from each other.\qqed
\end{prop}

We present the following isomorphism conjectures:

\begin{conj}\label{conj:0000000000}
\textbf{Graph isomorphism conjectures.}

(i) Assume that there are edge subsets $E_G\subset E(G)$ and $E_H\subset E(H)$ with $|E_G|=|E_H|$ such that two \emph{edge-removed graphs} $G-E_G\cong H-E_H$ for two connected $(p,q)$-graphs $G$ and $H$ admitting the isomorphic subgraph similarity. Then $G\cong H$ by Kelly-Ulam's Reconstruction Conjecture.

(ii) If each spanning tree $T_a$ of a connected $(p,q)$-graph $G_a$ corresponds to a spanning tree $T_b$ of another connected graph $G_b$ such that $T_a\cong T_b$, and vice versa, then $G_a\cong G_b$.

(iii) Let both $G$ and $H$ be graphs with $n$ vertices. If each proper subset $S_G\in V(G)\cup E(G)$ corresponds to another proper subset $S_H\in V(H)\cup E(H)$ such that two \emph{subset-removed graphs} $G-S_G \cong H-S_H$, then $G \cong H$.
\end{conj}

\begin{conj}\label{conj:0000000000}
\cite{Yao-Ma-arXiv-2024-13354} \textbf{Hypergraph isomorphism.} Let $\mathcal{H}_{yper}=(\Lambda,\mathcal{E})$ based on a 3I-hyperedge set $\mathcal{E}\in \mathcal{E}\big (\Lambda^2\big )$ and $\mathcal{H}^*_{yper}=(\Lambda^*,\mathcal{E}^*)$ based on a 3I-hyperedge set $\mathcal{E}^*\in \mathcal{E}\big ((\Lambda^*)^2\big )$ both be hypergraphs with two cardinalities $|\mathcal{E}|=|\mathcal{E}^*|$. If there is a bijection $\theta: \mathcal{E}\rightarrow \mathcal{E}^*$ such that two hyperedges $e\in \mathcal{E}$ and $\theta(e)\in \mathcal{E}^*$ hold two isomorphic hypergraphs
\begin{equation}\label{eqa:555555}
\big (\Lambda\setminus\{e_{\cap}\},\mathcal{E}\setminus e\big )\cong \big (\Lambda^*\setminus\{\theta(e_{\cap})\},\mathcal{E}^*\setminus \theta(e)\big )
\end{equation} where $e_{\cap}\subset e\in \mathcal{E}$ and $e_{\cap}\cap e\,'=\emptyset$ for any $e\,'\in \mathcal{E}\setminus e$, then $\mathcal{H}_{yper} \cong \mathcal{H}^*_{yper}$.
\end{conj}

\begin{thm}\label{thm:graph-set-graph-isomorphism}
Each totally colored and connected graph $H$ corresponds to a totally colored graph set $G_{raph}(H)$, such that each totally colored graph $T\in G_{raph}(H)$ is totally colored graph homomorphism to $H$, namely, $T\rightarrow _{color}H$.
\end{thm}

\begin{thm}\label{thm:666666}
By Theorem \ref{thm:graph-set-graph-isomorphism}, two totally colored and connected graphs $G$ and $H$ correspond to two totally colored graph sets $G_{raph}(G)$ and $G_{raph}(H)$ respectively, such that each totally colored graph $L\in G_{raph}(G)$ holds $L\rightarrow _{color}G$, and each totally colored graph $T\in G_{raph}(H)$ holds $T\rightarrow _{color}H$. Suppose that each totally colored graph $L\,'\in G_{raph}(G)$ corresponds to a totally colored graph $H\,'\in G_{raph}(H)$ holding $L\,'\cong H\,'$ true, and vice versa, then we claim that $G\cong H$.
\end{thm}

\begin{defn} \label{defn:set-colored-graphic-group}
\cite{Yao-Ma-arXiv-2024-13354} If a set-colored graph set $F_{\mathcal{E}}(G)=\{G_1,G_2,\dots ,G_n\}$ holds:

(i) Each graph $G_i$ is isomorphic to $G_1$, i.e. $G_i\cong G_1$;

(ii) each set-colored graph $G_i$ admits a total hyperedge-set coloring $F_i:V(G_i)\cup E(G_i)\rightarrow \mathcal{E}_i$, where $\mathcal{E}_i$ is a 3I-hyperedge set of the hypegraph set $\mathcal{E}(\Lambda^2_i)$ defined on a consecutive integer set $\Lambda_i$ for $i\in [1,n]$;

(iii) there is a positive integer $M=\max |\Lambda_1|$, such that the color $F_i(w_s)=\{b_{i,s,1},b_{i,s,2},\dots $, $b_{i,s,c(i,s)}\}$ of each element $w_s$ of $V(G_i)\cup E(G_i)$ is defined as $b_{i,s,r}=b_{1,s,r}+i-1~(\bmod ~M)$ with $r\in [1,c(i,s)]$ and $s\in [1,n]$; and

(iv) the finite module abelian additive operation $G_i [+_k] G_j:=G_i [+] G_j[-]G_k$ is defined by
\begin{equation}\label{eqa:set-colored-graphic-groups}
b_{i,s,r}+b_{j,s,r}-b_{k,s,r}=b_{\lambda,s,r}
\end{equation} with $\lambda=i+j-k~(\bmod ~M)$ for some $b_{i,s,r}\in F_i(w_s)$, $b_{j,s,r}\in F_j(w_s)$ and $b_{\lambda,s,r}\in F_{\lambda}(w_s)$, as well as $b_{k,s,r}\in F_k(w_s)$, where $G_k$ is a preappointed \emph{zero}.

Thereby, we call the set-colored graph set $F_{\mathcal{E}}(G)$ \emph{every-zero set-colored graphic group} (see an example shown in Fig.\ref{fig:set-coloring-group}), rewrite it as $\{F_{\mathcal{E}}(G)$; $[+][-]\}$.\qqed
\end{defn}

\begin{figure}[h]
\centering
\includegraphics[width=16.4cm]{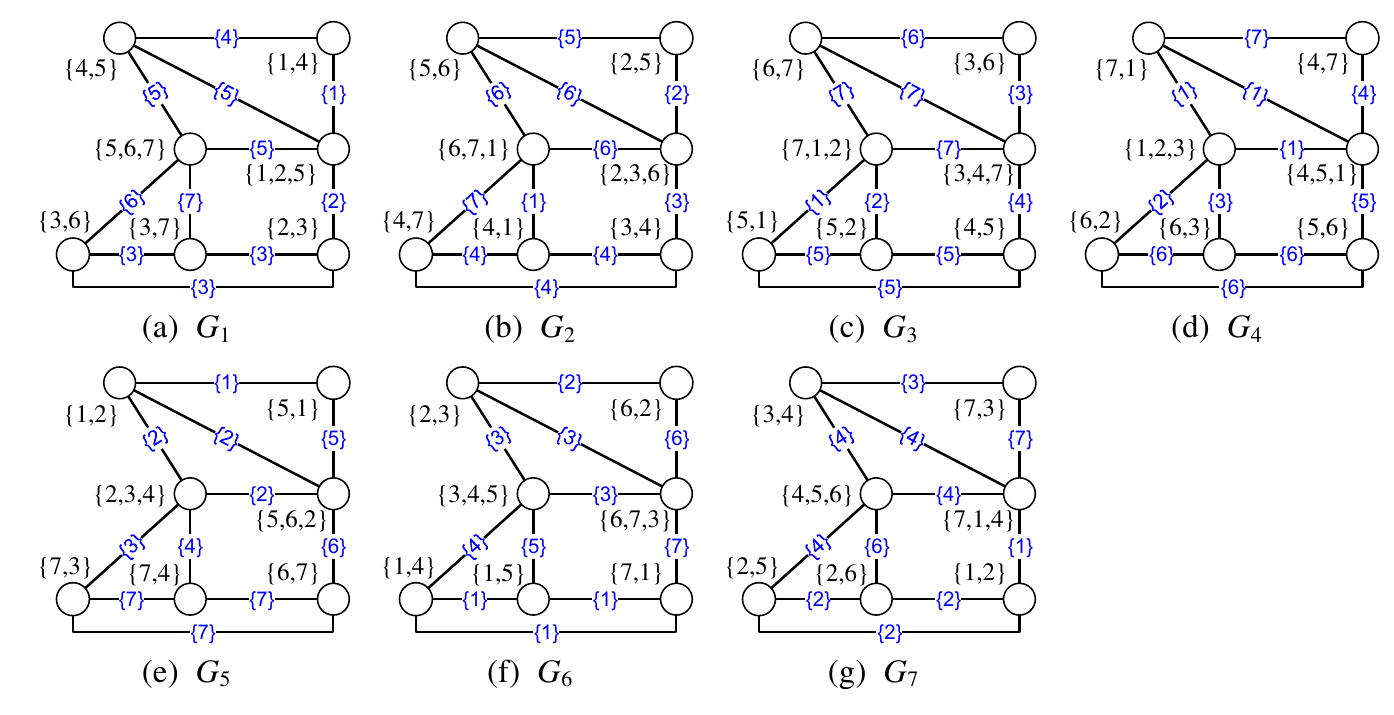}\\
\caption{\label{fig:set-coloring-group}{\small An every-zero set-colored graphic group $\{F_{\mathcal{E}}(G); [+][-]\}$ defined in Definition \ref{defn:set-colored-graphic-group}.}}
\end{figure}

\begin{defn}\label{defn:hypergraph-group-definition}
\textbf{Every-zero hypergraph group.} Suppose that the vertex set $\Lambda_{[1,N]}=[1,N]$. For getting a structural representation of the hypergraph set $\mathcal{E}\big (\Lambda_{[1,N]}^2\big )$ of all hyperedge sets defined on $\Lambda_{[1,N]}=[1,N]$, we take a particular 3I-hyperedge set $\mathcal{E}_1=\{e_{1,1},e_{1,2},\dots $, $e_{1,a_1}\}\in \mathcal{E}\big (\Lambda_{[1,N]}^2\big )$, where $a_1\geq 2$ and hyperedges $e_{1,s}=\{x_{1,s,1},x_{1,s,2},\dots ,x_{1,s,m_s}\}$ holding $x_{1,s,t}\in [1,N]$ with $t\in [1,m_s]$ and $s\in [1,a_1]$.

We use this especial hyperedge set $\mathcal{E}_1$ to make 3I-hyperedge sets $\mathcal{E}_i=\{e_{i,1},e_{i,2},\dots ,e_{i,a_i}\}$ with $i\in [1,N]$ and hyperedges $e_{i,s}=\{x_{i,s,1},x_{i,s,2},\dots ,x_{i,s,m_s}\}$ holding $x_{i,s,t}=x_{1,s,t}+(i-1)~(\bmod~N)\in [1,N]$ for $t\in [1,m_s]$ and $s\in [1,a_1]$. The set of 3I-hyperedge sets $\mathcal{E}_i$ generated by $\mathcal{E}_1$ is denoted as $G(\mathcal{E}_1)$.

We get an \emph{every-zero hypergraph group} $\big \{G(\mathcal{E}_1);[+][-]\big \}$ since any pair of 3I-hyperedge sets $\mathcal{E}_i$ and $\mathcal{E}_j$ of the set $G(\mathcal{E}_1)$ holds the finite module abelian additive operation $\mathcal{E}_i[+_k]\mathcal{E}_j$ true, where the finite module abelian additive operation
\begin{equation}\label{eqa:every-zero-hypergraph-group11}
\mathcal{E}_i[+_k]\mathcal{E}_j:=\mathcal{E}_i[+]\mathcal{E}_j[-]\mathcal{E}_k=\mathcal{E}_\lambda
\end{equation} for any preappointed \emph{zero} $\mathcal{E}_k\in G(\mathcal{E}_1)$ is defined as follows:
\begin{equation}\label{eqa:hypergraph-group-operation}
{
\begin{split}
\big (x_{i,s,t}+x_{j,s,t}\big )-x_{k,s,t}&=x_{1,s,t}+(i-1)+x_{1,s,t}+(j-1)-\big [x_{1,s,t}+(k-1)\big ]~(\bmod~N)\\
&=x_{1,s,t}+(i+j-k-1)~(\bmod~N)\\
&=x_{\lambda,s,t}\in [1,N]
\end{split}}
\end{equation} with $\lambda=i+j-k-1~(\bmod~N)$, where $x_{i,s,t}\in e_{i,s}\in \mathcal{E}_i$ and $x_{j,s,t}\in e_{j,s}\in \mathcal{E}_j$, as well as $x_{k,s,t}\in e_{k,s}\in \mathcal{E}_k$ with $t\in [1,m_s]$ and $s\in [1,a_1]$.\qqed
\end{defn}

We propose that following two classes of hypergraph-type Top-graph databases:

\begin{defn} \label{defn:111111}
$^*$ Suppose that a total coding graph set $\textbf{\textrm{G}}=\{G_1, G_2, \dots, G_M\}$ forms an every-zero graph group $\{F(\textbf{\textrm{G}}); [+][-]\}$ based on the finite module abelian additive operation, and the set $S=\{e_1, e_2, \dots, e_n\}$ with $e_i\in \textbf{\textrm{G}}^2$ is a 3I-hyperedge set of the power set $\textbf{\textrm{G}}^2$ defined in Definition \ref{defn:new-hypergraphs-sets}.

(i) If a connected graph $H$ admits a proper vertex set-coding $\beta: V(H)\rightarrow S$, such that each edge $xy\in E(H)$ holds $\beta(x)\cap \beta(y)\neq \emptyset$, and moreover each pair of hyperedges $e_i,e_j\in S$ with $|e_i\cap e_j|\geq 1$ corresponds to an edge $uv\in E(H)$ holding $\beta(u)=e_i$ and $\beta(v)=e_j$, then we call the connected graph $H$ to be a \textbf{\emph{vertex-intersected hypergraph-type Top-graph database}}.

(ii) If a connected graph $L$ admits a proper total set-coding $\beta: V(L)\cup E(L)\rightarrow S$, and each edge $uv\in E(L)$ holds $\beta(uv)=e_{ij}$, for some total coding graph $G_c\in e_{ij}$, then there are total coding graphs $G_a\in e_{i}=\beta(u)$ and $G_b\in e_{j}=\beta(v)$, such that
\begin{equation}\label{eqa:555555}
G_a[+]G_b[-]G_k=G_c
\end{equation} with $c=a+b-k (\bmod ~ M)$ for an arbitrarily preappointed zero $G_k\in \textbf{\textrm{G}}$, then the connected graph $L$ is called \textbf{\emph{total hypergraph-type Top-graph database}}.\qqed
\end{defn}

\section{Non-ordinary Hypergraphs}

We, in this section, will build up three kinds of non-ordinary hypergraphs: Topology-hypergraphs, Code-hypergraphs and Operation-hypergraphs.

\subsection{Topology-hypergraphs}

\subsubsection{$K_{tree}$-hypergraphs}

\begin{defn} \label{defn:complete-graph-spanning-trees}
Let $S_{pan}(K_n)$ be the set of spanning trees of a complete graph $K_n$ of $n$ vertices, where $K_n$ admits a vertex coloring $f:V(K_n)\rightarrow [1,n]$ holding the vertex color set $f(V(K_n))=[1,n]$ true. We have:

(1) The cardinality $|S_{pan}(K_n)|=n^{n-2}$ by the famous Cayley's formula $\tau (K_{n})=n^{n-2}$.

(2) Each spanning tree $T_i\in S_{pan}(K_n)$ admits a vertex coloring $f_i:V(T_i)\rightarrow [1,n]$ holding the vertex color set $f_i(V(T_i))=[1,n]$.

(3) $S_{pan}(K_n)=\bigcup ^{A_n}_{k=1}N^k_{pan}(K_n)$, where $A_n$ is the number of non-isomorphic spanning tree classes in the $n^{n-2}$ spanning trees, such that

\quad (3.1) any pair of spanning trees $T_{k,i},T_{k,j}\in N^k_{pan}(K_n)$ holds $T_{k,i}\not\cong T_{k,j}$ if $i\neq j$;

\quad (3.2) any spanning tree $T_{k,i}\in N^k_{pan}(K_n)$ is isomorphic to some spanning tree $T_{l,t}\in N^l_{pan}(K_n)$ for each $l\in [1,A_n]\setminus \{k\}$.

(4) $S_{pan}(K_n)=\bigcup ^{B_n}_{k=1}I^k_{pan}(K_n)$, where $B_n$ is the number of isomorphic spanning tree classes in the $n^{n-2}$ spanning trees, such that

\quad (4.1) any pair of spanning trees $T_{i,j},T_{i,t}\in I^i_{pan}(K_n)$ holds $T_{i,j}\cong T_{i,t}$;

\quad (4.2) any spanning tree $T_{k,i}\in I^k_{pan}(K_n)$ is not isomorphic to any spanning tree $T_{l,s}\in I^p_{pan}(K_n)$ for each $p\in [1,B_n]\setminus \{k\}$.\qqed
\end{defn}

\begin{prop}\label{prop:99999}
Each spanning tree set $I^k_{pan}(K_n)$ defined in the case (4) of Definition \ref{defn:complete-graph-spanning-trees} is the union of several graphic groups, that is, $I^k_{pan}(K_n)=\bigcup ^{m_k}_{i=1}\{F(G_i);[+][-]\}$, where each spanning tree set $F(G_i)=\{T_{i,1},T_{i,2},\dots ,T_{i,n}\}$ with $T_{i,j}\cong T_{i,y}$ and each spanning tree $T_{i,s}$ admits a vertex coloring $f_{i,s}$ defined in the case (2) of Definition \ref{defn:complete-graph-spanning-trees}, such that $f_{i,j}(x)=f_{i,1}(x)+j-1~(\bmod ~n)$ for $x\in V(T_{i,1})=V(T_{i,j})$ with $j\in [1,n]$.
\end{prop}

\begin{example}\label{exa:8888888888}
By the notation of Definition \ref{defn:complete-graph-spanning-trees}, the complete graph $K_{26}$ has
$$26^{24}=9,106,685,769,537,220,000,000,000,000,000,000$$ colored spanning trees in the spanning tree set $S_{pan}(K_{26})$. Since there are $t_{26}=279,793,450$ non-isomorphic spanning trees of 26 vertices in $K_{26}$, so there are
\begin{equation}\label{eqa:555555}
26^{24}\div t_{26}=32,547,887,627,595,300,000,000,000
\end{equation} trees being isomorphic to the non-isomorphic spanning trees of 26 vertices.

For the case (3) of Definition \ref{defn:complete-graph-spanning-trees}, we have $A_{26}~(\geq 26^{24}\div t_{26})$ classes $N^k_{pan}(K_{26})$ of non-isomorphic spanning trees, and each cardinality $|N^k_{pan}(K_{26})|\leq t_{26}$ for $k\in [1,A_{26}]$.

Clearly, $B_{26}=t_{26}$ in the case (4) of Definition \ref{defn:complete-graph-spanning-trees}.\qqed
\end{example}

\begin{defn} \label{defn:vertex-coinciding-two-spanning-trees}
We define a graph operation ``$[\bullet^{coin}_{colo}]$'' between two spanning trees $T_j,T_k\in S_{pan}(K_n)$ defined in Definition \ref{defn:complete-graph-spanning-trees} by vertex-coinciding a vertex $u_i$ of $T_j$ with a vertex $x_i$ of $T_k$ into one vertex $u_i\bullet x_i$ if $f_j(u_i)=f_k(x_i)=i$ for $i\in [1,n]$, and the resultant graph after removing multiple-edge is denoted as $T_j[\bullet^{coin}_{colo}]T_k$.\qqed
\end{defn}

\begin{thm}\label{thm:666666}
By Definition \ref{defn:complete-graph-spanning-trees} and Definition \ref{defn:vertex-coinciding-two-spanning-trees}, any spanning tree $T_i\in S_{pan}(K_n)$ corresponds to two spanning trees $T_j,T_k\in S_{pan}(K_n)$, such that the spanning tree $T_i$ is a subgraph of the graph $T_j[\bullet^{coin}_{colo}]T_k$, here, $T_i\not \cong T_j$, $T_i\not \cong T_k$ and $T_j\not \cong T_k$.
\end{thm}
\begin{proof}Obviously, the graph $T_j[\bullet^{coin}_{colo}]T_k$ is connected and has at least a cycle $C=x_1x_2\cdots x_mx_1$ with $m\geq 3$, because of $T_j\not \cong T_k$. Without loss of generality, $x_mx_1\in E(T_j)$ and $x_1x_2\in E(T_k)$, we remove an edge $x_2x_3$ from the graph $T_j[\bullet^{coin}_{colo}]T_k$, the resultant graph is denoted as $G_1=\big (T_j[\bullet^{coin}_{colo}]T_k\big )-x_2x_3$, clearly, $G_1$ is connected and $x_1x_2,x_mx_1\in E(G_1)$.

If $G_1$ has a cycle $C_1$, we remove an edge $u_1v_1\in E(C_1)\setminus \{x_1x_2,x_mx_1\}$ from $G_1$, and get a connected graph $G_2=G_1-u_1v_1$, go on in this way, we obtain a spanning tree $T_i\in S_{pan}(K_n)$ holding $x_1x_2,x_mx_1\in E(T_i)$ true. Notice that the vertex color set
$$f(V(K_n))=[1,n]=f(V(T_i))=f(V(T_j))=f(V(T_j))
$$ so $T_i\not \cong T_j$, $T_i\not \cong T_k$ and $T_j\not \cong T_k$. The proof of the theorem is complete.
\end{proof}

\begin{problem}\label{qeu:444444}
\textbf{Determine} the number $A_n$ defined in Definition \ref{defn:complete-graph-spanning-trees}.
Since the number $\tau (K_{m,n})$ of all spanning trees of a bipartite complete graph $K_{m,n}$ is $\tau (K_{m,n})=m^{n-1}n^{m-1}$, \textbf{do} researching works like that for the complete graph $K_{n}$.
\end{problem}

\begin{defn} \label{defn:spanning-tree-hyperedge-set}
A hyperedge set $\mathcal{E}_i=\{e_{i,1},e_{i,2},\dots ,e_{i,i_a}\}$ is a set of subsets of the spanning tree set $S_{pan}(K_n)$ defined in Definition \ref{defn:complete-graph-spanning-trees} holds $S_{pan}(K_n)=\bigcup_{e_{i,s}\in \mathcal{E}_i}e_{i,s}$, and has one of the following properties:
\begin{asparaenum}[\textbf{\textrm{Prop}}-1.]
\item \label{Prop:11} Each hyperedge $e_{i,s}\in \mathcal{E}_i$ corresponds to a hyperedge $e_{i,t}\in \mathcal{E}_i\setminus e_{i,s}$ holding $e_{i,s}\cap e_{i,t}\neq \emptyset$.
\item \label{Prop:22} A spanning tree $T_{i,k}\in e_{i,k}$ is a proper subgraph of the graph $T_{i,s}[\bullet^{coin}_{colo}]T_{i,t}$ for some two spanning trees $T_{i,s}\in e_{i,s}$ and $T_{i,t}\in e_{i,t}$.
\item \label{Prop:33} Two spanning trees $T_{i,s}\in e_{i,s}$ and $T_{i,t}\in e_{i,t}$ hold the \emph{adding-edge-removing operation}
\begin{equation}\label{eqa:adding-edge-removing-operation}
T_{i,t}=T_{i,s}+x_iy_i-u_iv_i\quad (T_{i,t}-x_iy_i\cong T_{i,s}-u_iv_i)
\end{equation} for $x_iy_i\not\in E(T_{i,s})$ and $u_iv_i\in E(T_{i,s})$, we write this \emph{adding-edge-removing operation} as $T_{i,t}=\pm_e[T_{i,t}]$. Thereby, the spanning tree set $S_{pan}(K_n)$ is a \emph{fixed-point spanning-tree set} under the \emph{adding-edge-removing operation} shown in Eg.(\ref{eqa:adding-edge-removing-operation}).\qqed
\end{asparaenum}
\end{defn}

\begin{thm}\label{thm:adding-edge-removing-sets}
\cite{Yao-Su-Ma-Wang-Yang-arXiv-2202-03993v1} Let $T_{\pm e}(\leq n)$ be the set of trees of $p$ vertices with $p\leq n$. Then each tree $H\in T_{\pm e}(\leq n)$ is a star $K_{1,p-1}$, or corresponds to another tree $T\in T_{\pm e}(\leq n)$ holding the \emph{adding-edge-removing operation} $H-uv\cong T-xy$ for $xy\in E(T)$ and $uv \in E(H)$, so $T_{\pm e}(\leq n)$ is a \emph{fixed-point tree set} under the \emph{adding-edge-removing operation} shown in Eg.(\ref{eqa:adding-edge-removing-operation}).
\end{thm}

\begin{thm}\label{thm:adding-edge-removing-w-type-sets}
\cite{Yao-Su-Ma-Wang-Yang-arXiv-2202-03993v1} Let $G_{tree}(\leq n)$ be the set of trees of $p$ vertices with $p\leq n$. Then each tree $H\in G_{tree}(\leq n)$ admits a $W$-constraint coloring and corresponds to another tree $T\in G_{tree}(\leq n)$ admitting a $W$-constraint coloring holding $H-uv\cong T-xy$ (resp. $H+xy\cong T+uv$) under the \emph{adding-edge-removing operation} for some edges $xy\in E(T)$ and $uv \in E(H)$.
\end{thm}

\begin{defn} \label{defn:three-vertex-coincided-intersected-graphs}
Suppose that a graph $H$ admits a total coloring $F:V(H)\cup E(H)\rightarrow \mathcal{E}_i$, where the hyperedge set $\mathcal{E}_i=\{e_{i,1},e_{i,2},\dots ,e_{i,i_a}\}$ is a set of subsets of the spanning tree set $S_{pan}(K_n)$ defined in Definition \ref{defn:complete-graph-spanning-trees}.

(i) If each edge $uv\in E(H)$ holds $F(uv)\supseteq F(u)\cap F(v)\neq \emptyset$, also, holds \textbf{\textrm{Prop}}-\ref{Prop:11} in Definition \ref{defn:spanning-tree-hyperedge-set}, then $H$ is called \emph{$K_{tree}$-spanning ve-intersected graph} of the $K_{tree}$-hypergraph $\mathcal{H}_{yper}=(S_{pan}(K_n),\mathcal{E}_i)$.

(ii) If each edge $uv\in E(H)$ holds $T_{i,k}\subset T_{i,s}[\bullet^{coin}_{colo}]T_{i,t}$ for some spanning trees $T_{i,k}\in F(uv)$, $T_{i,s}\in F(u)$ and $T_{i,t}\in F(v)$ (also, \textbf{\textrm{Prop}}-\ref{Prop:22} in Definition \ref{defn:spanning-tree-hyperedge-set}), then $H$ is called \emph{$K_{tree}$-spanning vertex-coincided graph} of the $K_{tree}$-hypergraph $\mathcal{H}_{yper}=(S_{pan}(K_n),\mathcal{E}_i)$.

(iii) If each edge $uv\in E(H)$ holds $T_{i,k}=\pm_e[T_{i,s}]$ and $T_{i,t}=\pm_e[T_{i,k}]$ (also, \textbf{\textrm{Prop}}-\ref{Prop:33} in Definition \ref{defn:spanning-tree-hyperedge-set}) for some spanning trees $T_{i,k}\in F(uv)$, $T_{i,s}\in F(u)$ and $T_{i,t}\in F(v)$, then $H$ is called \emph{$K_{tree}$-spanning added-edge-removed graph} of the $K_{tree}$-hypergraph $\mathcal{H}_{yper}=(S_{pan}(K_n),\mathcal{E}_i)$.\qqed
\end{defn}

\begin{rem}\label{rem:333333}
We generalize Definition \ref{defn:complete-graph-spanning-trees}, Definition \ref{defn:vertex-coinciding-two-spanning-trees}, Definition \ref{defn:spanning-tree-hyperedge-set} and Definition \ref{defn:three-vertex-coincided-intersected-graphs} to ordinary graphs. Let $S_{pan}(G)$ be the set of all spanning trees of a connected graph $G$. There are the $G_{tree}$-spanning ve-intersected graph, $G_{tree}$-spanning vertex-coincided graph and the $G_{tree}$-spanning added-edge-removed graph of the $G_{tree}$-hypergraph $\mathcal{H}_{yper}=(S_{pan}(G),\mathcal{E})$.

However, vertex-splitting a graph $G$ into vertex disjoint spanning trees is a NP-complete problem, since ``Counting trees in a graph is \#P-complete'' \cite{Jerrum-Mark-1994-information}.\qqed
\end{rem}

\textbf{$K_{tree}$-spanning lattice.} We select randomly spanning trees $T^c_1,T^c_2,\dots ,T^c_m$ from the spanning tree set $S_{pan}(K_n)$ to form a \emph{spanning tree base} $\textbf{\textrm{T}}^c=(T^c_1,T^c_2,\dots $, $T^c_m)$ with $T^c_i\not \subset T^c_j$ and $T^c_i\not \cong T^c_j$ if $i\neq j$ and have a permutation $J_1,J_2,\dots ,J_A$ of vertex disjoint spanning trees $a_1T^c_1,a_2T^c_2,\dots ,a_mT^c_m$, where $A=\sum ^m_{k=1}a_k\geq 1$. By Definition \ref{defn:vertex-split-coinciding-operations}, we do the vertex-coinciding operation ``$[\bullet ]$'' to two spanning trees $J_1$ and $J_2$ by vertex-coinciding a vertex $u$ of the spanning tree $J_1$ with a vertex $v$ of the spanning tree $J_2$ into one vertex $u\bullet v$ if these two vertices are colored the same color, and then get a connected graph $H_1=J_1[\bullet]J_2$, next we get another connected graph $H_2=H_1[\bullet]J_3$ by the same action for obtaining the connected graph $H_1=J_1[\bullet]J_2$; go on in this way, we get connected graphs $H_k=H_{k-1}[\bullet]J_{k+1}$ for $k\in [1,A]$ with $H_0=J_1$. Write $H_A=[\bullet]^m_{k=1}a_kT^c_k$, and the following set
\begin{equation}\label{eqa:K-tree-spanning-lattice}
\textbf{\textrm{L}}(Z^0[\bullet]\textbf{\textrm{T}}^c)=\left \{[\bullet]^m_{k=1}a_kT^c_k:~a_k\in Z^0, \sum^m_{k=1} a_k\geq 1, T^c_k\in \textbf{\textrm{T}}^c \right \}
\end{equation} is called \emph{$K_{tree}$-spanning lattice} based on the spanning tree base $\textbf{\textrm{T}}^c \subseteq S_{pan}(K_n)$ which is the set of spanning trees of a complete graph $K_n$.

Thereby, each connected graph $G\in \textbf{\textrm{L}}(Z^0[\bullet]\textbf{\textrm{T}}^c)$ shown in Eq.(\ref{eqa:K-tree-spanning-lattice}) can be vertex-split into the vertex disjoint spanning trees $a_1T^c_1$, $a_2T^c_2$, $\dots $, $a_mT^c_m$ with $\sum ^m_{k=1}a_k\geq 1$, and forms a $G_{tree}$-hypergraph $\mathcal{H}_{yper}=(S_{pan}(G),\mathcal{E})$ for a 3I-hyperedge set $\mathcal{E}\in \mathcal{E}(S^2_{pan}(G))$ (Ref. Definition \ref{defn:more-conceptd-hypergraphs} and Definition \ref{defn:new-hypergraphs-sets}).

In the view of homomorphism, the group of spanning trees $a_1T^c_1$, $a_2T^c_2$, $\dots $, $a_mT^c_m$ is graph homomorphism to $G$, denoted as
\begin{equation}\label{eqa:555555}
\big (a_1T^c_1,a_2T^c_2,\dots ,a_mT^c_m\big )\rightarrow _{[\bullet]}G
\end{equation}

\vskip 0.2cm

\textbf{$W$-operation tree-lattices.} Let $\textbf{\textrm{H}}=\{G_1,G_2,\dots ,G_m\}$ be a set of connected graphs of $q$ edges, where $G_1=K_{1,q}$, and each connected graph $G_i$ for $i\in [2,m]$ is not a tree and has $q=|E(G_i)|$ edges, and moreover $S_{plit}(G_i)$ for $i\in [2,m]$ is a set of trees of $q$ edges obtained by vertex-splitting $G_i$ into trees of $q$ edges. We have a tree set $T_{ree}(q)=\bigcup^m_{i=1}S_{plit}(G_i)$, where $S_{plit}(G_1)=\{G_1=K_{1,q}\}$.

We call the set $\textbf{\textrm{S}}_{plit}=\big (S_{plit}(G_1),S_{plit}(G_2),\dots ,S_{plit}(G_m)\big )$ \emph{$q$-tree-set base}, where $S_{plit}(G_k)=\{T_{k,1}, T_{k,2},\dots ,T_{k,s_k}\}$ for $k\in [1,m]$, $N=\sum^m_{k=1} s_k$.
\begin{quote}
Let $L_{i_1},L_{i_2},\dots, L_{i_N}$ be a permutation of vertex disjoint trees $a_1T_{1,1}, a_1T_{1,2},\dots ,a_1T_{1,s_1}$ (abbreviated as $a_1S_{plit}(G_1)$), $a_2T_{2,1}, a_2T_{2,2},\dots ,a_2T_{2,s_2}$ (abbreviated as $a_1S_{plit}(G_2)$), $\dots$, $a_mT_{m,1}, a_mT_{m,2},\dots ,a_mT_{m,s_m}$ (abbreviated as $a_mS_{plit}(G_m)$).
\end{quote}
We do the vertex-coinciding operation to the trees $L_{i_1},L_{i_2},\dots, L_{i_N}$ in the following way: We vertex-coincide a vertex of the tree $L_{i_1}$ with a vertex of the tree $L_{i_2}$ together to produce a tree $H_1=L_{i_1}[\bullet]L_{i_2}$, and then we vertex-coincide a vertex of the tree $H_1$ with a vertex of the tree $L_{i_3}$ together to make a tree $H_2=H_1[\bullet]L_{i_3}$; go on in this way, we get trees $H_j=H_{j-1}[\bullet]L_{i_{j+1}}$ for $j\in [2,N-1]$. For simplicity, we write
\begin{equation}\label{eqa:vertex-coinciding-trees}
H_{N-1}=H_{N-2}[\bullet]L_{i_{N}}=L_{i_1}[\bullet]L_{i_2}[\bullet]\dots [\bullet]L_{i_N}=[\bullet]^m_{k=1}a_kS_{plit}(G_k)
\end{equation}
and get a \emph{vertex-coinciding tree-lattice} as follows
\begin{equation}\label{eqa:vertex-coinciding-tree-lattice}
\textbf{\textrm{L}}(Z^0[\bullet]\textbf{\textrm{S}}_{plit})=\left \{[\bullet]^m_{k=1}a_kS_{plit}(G_k):~a_k\in Z^0,~\sum ^m_{k=1}a_k\geq 1, S_{plit}(G_k)\in \textbf{\textrm{S}}_{plit}\right \}
\end{equation}

We join a vertex of the tree $L_{i_1}$ with a vertex of the tree $L_{i_2}$ by an edge to produce a tree $I_1=L_{i_1}[\ominus]L_{i_2}$, and then we join a vertex of the tree $I_1$ with a vertex of the tree $L_{i_3}$ by an edge to make a tree $I_2=I_1[\ominus]L_{i_3}$; go on in this way, we get trees $I_j=I_{j-1}[\ominus]L_{i_{j+1}}$ for $j\in [2,N-1]$. Similarly with the above Eq.(\ref{eqa:vertex-coinciding-trees}), we write
\begin{equation}\label{eqa:vertex-join-trees}
I_{N-1}=I_{N-2}[\ominus]L_{i_{N}}=L_{i_1}[\ominus]L_{i_2}[\ominus]\dots [\ominus]L_{i_N}=[\ominus]^m_{k=1}a_kS_{plit}(G_k)
\end{equation} and obtain an \emph{edge-joining tree-lattice}
\begin{equation}\label{eqa:edge-joining-tree-lattice}
\textbf{\textrm{L}}(Z^0[\ominus]\textbf{\textrm{S}}_{plit})=\left \{[\ominus]^m_{k=1}a_kS_{plit}(G_k):~a_k\in Z^0,~\sum ^m_{k=1}a_k\geq 1, S_{plit}(G_k)\in \textbf{\textrm{S}}_{plit}\right \}
\end{equation}

Notice that each tree $T_{k,i}\in S_{plit}(G_k)=\{T_{k,1}, T_{k,2},\dots ,T_{k,s_k}\}$ for $i\in [1,s_k]$ and $k\in [1,m]$ holds $T_{k,i}\rightarrow G_k$, we do the graph homomorphism to the trees in $[\bullet]^m_{k=1}a_kS_{plit}(G_k)$ and $[\ominus]^m_{k=1}a_kS_{plit}(G_k)$, and get the following two graph lattice:
\begin{equation}\label{eqa:vertex-coinciding-graphs-lattice}
\textbf{\textrm{L}}(Z^0[\bullet]\textbf{\textrm{H}})=\left \{[\bullet]^m_{k=1}A_kG_k:~A_k\in Z^0,~\sum ^m_{k=1}A_k\geq 1, G_k\in \textbf{\textrm{H}}\right \}
\end{equation}
and
\begin{equation}\label{eqa:vertex-coinciding-graphs-lattice}
\textbf{\textrm{L}}(Z^0[\ominus]\textbf{\textrm{H}})=\left \{[\ominus]^m_{k=1}B_kG_k:~B_k\in Z^0,~\sum ^m_{k=1}B_k\geq 1, G_k\in \textbf{\textrm{H}}\right \}
\end{equation}

\subsubsection{$F_{orest}$-hypergraphs}

\begin{defn} \label{defn:produce-forests-spanning-trees}
Removing the edges of one edge set $E_i=\{e_{i,1},e_{i,2}, \dots,e_{i,a_i}\}$ from a spanning tree $T_i$ of a complete graph $K_n$, such that the resultant graph $F_{i}=T_i-E_i$ is just a forest having component trees $T_{i,1},T_{i,2}, \dots,T_{i,b_i}$ with $|V(T_{i,j})|\geq 2$ for $j\in [1,b_i]$. Notice that $V(K_n)=V(T_i)=V(F_i)=\bigcup^{b_i}_{j=1}V(T_{i,j})$. Let $F_{orest}(T_i)=\{F_{j}:j\in[1,n_{orest}(T_i)]\}$ be the set of distinct forests produced from a spanning tree $T_i$ of $S_{pan}(K_n)$, where $n_{orest}(T_i)$ is the number of distinct forests produced by the spanning tree $T_i$. So, the forest set $F_{orest}(K_n)=\{F_{orest}(T_i):T_i\in S_{pan}(K_n)\}$ of distinct forests is produced by the spanning trees of $S_{pan}(K_n)$.\qqed
\end{defn}

\begin{defn} \label{defn:two-operations-forests}
\textbf{Adding-edgeset-removing operation.} Since the spanning tree $T_i=F_{i}+E_i$ obtained by adding the edges of an edge set $E_i\subset E(K_n)$ in Definition \ref{defn:produce-forests-spanning-trees}, then we can add the edges of other edge set $E_j\subset E(K_n)$ to the forest $F_{i}$, such that the resultant graph $F_{i}+E_j$ is just a spanning tree $T_j\in S_{pan}(K_n)$, that is, $T_j=F_{i}+E_j$. By the \emph{adding-edge-removing operation}, we get graphs $T_i-E_i=F_{i}=T_j-E_j$, and
\begin{equation}\label{eqa:add-eset-move-2-spanning-trees}
T_j=T_i-E_i+E_j,~E_i\subset E(T_i),~E_j\cap E(T_i)=\emptyset
\end{equation} denoted as $T_j=\pm_E[T_i]$.

\textbf{Forest vertex-coinciding operation.} We define a graph operation ``$[\bullet^{coin}_{forest}]$'' between two forests $F_j,F_k\in F_{orest}(K_n)$ defined in Definition \ref{defn:produce-forests-spanning-trees} by vertex-coinciding a vertex $x_i$ of $F_j$ with a vertex $y_i$ of $F_k$ into one vertex $x_i\bullet y_i$ if $f_j(x_i)=f_k(y_i)=i$ for $i\in [1,n]$ (it is guaranteed by Definition \ref{defn:complete-graph-spanning-trees}), and the resultant graph after removing multiple-edge is denoted as $F_j[\bullet^{coin}_{forest}]F_k$.\qqed
\end{defn}

\begin{thm}\label{thm:666666}
A forest of a connected graph $G$ can be produced by many spanning trees of the spanning tree set $S_{pan}(G)$ by the removing edge operation defined in Definition \ref{defn:two-operations-forests}.
\end{thm}

\begin{thm}\label{thm:666666}
Two forests $F_j,F_k\in F_{orest}(K_n)$ produce a forest $F_i\subset F_j[\bullet^{coin}_{forest}]F_k$ under the graph operation $[\bullet^{coin}_{forest}]$ on the forests of $F_{orest}(K_n)$ defined in Definition \ref{defn:produce-forests-spanning-trees} and Definition \ref{defn:two-operations-forests}.
\end{thm}

\begin{defn} \label{defn:coincided-intersected-graphs-by-hyperedge-sets}
A hyperedge set $\mathcal{E}_i=\{e_{i,1},e_{i,2},\dots ,e_{i,i_a}\}$ is a set of subsets of the forest set $F_{orest}(K_n)$ defined in Definition \ref{defn:produce-forests-spanning-trees} holds $F_{orest}(K_n)=\bigcup_{e_{i,s}\in \mathcal{E}_i}e_{i,s}$, and has one of the following properties:
\begin{asparaenum}[\textbf{\textrm{Forest}}-1.]
\item \label{Spantree:11} Each hyperedge $e_{i,s}\in \mathcal{E}_i$ corresponds to another hyperedge $e_{i,t}\in \mathcal{E}_i\setminus e_{i,s}$ holding $e_{i,s}\cap e_{i,t}\neq \emptyset$.
\item \label{Spantree:22} A forest $F_{i,k}\in e_{i,k}$ is a proper subgraph of the graph $F_{i,s}[\bullet^{coin}_{forest}]F_{i,t}$ defined in Definition \ref{defn:two-operations-forests} for some forests $F_{i,s}\in e_{i,s}$ and $F_{i,t}\in e_{i,t}$.
\item \label{Spantree:33} Two spanning trees $T_{i,s}\in e_{i,s}$ and $T_{i,t}\in e_{i,t}$ hold the \emph{adding-edgeset-removing operation} $T_{i,t}=T_{i,s}+E^*_i-E_i$ for two edge sets $E^*_i\cap E(T_{i,s})=\emptyset$, $E^*_i\subset E(T_{i,t})$ and $E_i\subset E(T_{i,s})$, also $T_{i,t}=\pm_E[T_{i,t}]$ according to Definition \ref{defn:two-operations-forests}.\qqed
\end{asparaenum}
\end{defn}

Similarly with the graphs defined in Definition \ref{defn:three-vertex-coincided-intersected-graphs}, then Definition \ref{defn:coincided-intersected-graphs-by-hyperedge-sets} enables us to have the following graphs:

\begin{defn} \label{defn:111111}
A graph $H$ admits a total coloring $g:V(H)\cup E(H)\rightarrow \mathcal{E}_i$ for a hyperedge set $\mathcal{E}_i$ of subsets of $F_{orest}(K_n)$.

(i) If each edge $uv\in E(H)$ holds $g(uv)\supseteq g(u)\cap g(v)\neq \emptyset$, also, holds \textbf{\textrm{Spantree}}-\ref{Spantree:11} in Definition \ref{defn:coincided-intersected-graphs-by-hyperedge-sets}, then $H$ is called \emph{$K$-forest ve-intersected graph} of the $K_{tree}$-hypergraph $\mathcal{H}_{yper}=(F_{orest}(K_n),\mathcal{E}_i)$.

(ii) If each edge $uv\in E(H)$ holds $F_{i,s}[\bullet^{coin}_{forest}]F_{i,t}$ for some forests $F_{i,k}\in g(uv)$, $F_{i,s}\in g(u)$ and $F_{i,t}\in g(v)$ (also, \textbf{\textrm{Spantree}}-\ref{Spantree:22} in Definition \ref{defn:coincided-intersected-graphs-by-hyperedge-sets}), then $H$ is called \emph{$K$-forest vertex-coincided graph} of the $K_{tree}$-hypergraph $\mathcal{H}_{yper}=(F_{orest}(K_n),\mathcal{E}_i)$.

(iii) If each edge $uv\in E(H)$ holds $T_{i,k}=\pm_e[T_{i,s}]$ and $T_{i,t}=\pm_e[T_{i,k}]$ (also, \textbf{\textrm{Spantree}}-\ref{Spantree:33} in Definition \ref{defn:coincided-intersected-graphs-by-hyperedge-sets}) for some spanning trees $F_{i,k}\in g(uv)$, $F_{i,s}\in g(u)$ and $F_{i,t}\in g(v)$, then $H$ is called \emph{$K$-forest added-edgeset-removed graph} of the $K_{tree}$-hypergraph $\mathcal{H}_{yper}=(F_{orest}(K_n),\mathcal{E}_i)$.\qqed
\end{defn}

\begin{rem}\label{rem:333333}
A \emph{forest} $S(k)=\{T_1,T_2,\dots, T_k\}$ with $k\in [1,n]$ has its vertex number $n=\sum^{k} _{i=1}|V(T_i)|$. Takacs, in \cite{L-Takacs-574-781-1990}, showed the number $F_{orest}(n)$ of distinct forests $S(k)$ with $k\in [1,n]$ to be
\begin{equation}\label{eqa:Forest-number-Takacs}
F_{orest}(n)=\frac{n!}{n+1}\sum^{\lfloor\frac{n}{2}\rfloor}_{k=0}(-1)^k\frac{(2k+1)(n+1)^{n-2k}}{2^k\cdot k!\cdot (n-2k)!}
\end{equation}
\begin{equation}\label{eqa:555555}
F_{orest}(n)=H_n(n+1)-nH_{n-1}(n+1)
\end{equation} where $H_n(x)$ is the $n$-th Hermite polynomial such that
\begin{equation}\label{eqa:555555}
H_n(x)=n!\sum^{\lfloor\frac{n}{2}\rfloor}_{k=0}\frac{(-1)^k x^{n-2k}}{2^k\cdot k!\cdot (n-2k)!}=\frac{1}{\sqrt{2\pi}}\int ^{+\infty}_{-\infty}e^{-\frac{u^2}{2}}(x-iu)^n\textrm{d}u
\end{equation}

As $k=0$ in the formula (\ref{eqa:Forest-number-Takacs}), we get the famous Caley's formula $F_{orest}(n)=\tau(K_{n+1})=(n+1)^{n-1}$. As $k=1$ in the formula (\ref{eqa:Forest-number-Takacs}), we have \begin{equation}\label{eqa:555555}
F_{orest}(n)=(n+1)^{n-1}+\frac{n!}{n+1}\frac{(-1)3(n+1)^{n-2}}{2(n-2)!}=(n+1)^{n-1}-\frac{3n(n-1)(n+1)^{n-3}}{2}
\end{equation}

Partitioning an positive integer $n$ as a sum $n=\sum^{k} _{i=1}|V(T_i)|$ with $|V(T_i)|\geq 2$ for $i\in [1,k]$ is the Integer Partition Problem (IPP), however, there is no polynomial algorithm for IPP. \qqed
\end{rem}

\subsubsection{$E_{hami}$-hypergraphs}

\begin{defn}\label{defn:edge-hamiltonian-connected}
Suppose that a connected graph $G$ contains a particular subgraph $L$ with $|E(L)|\geq 1$. We say $G$ to be \emph{edge-$L$-graph}, if each edge $uv\in E(G)$ is in some subgraph $L$ of the graph $G$. And if each edge $uv\in E(G)$ is in a hamiltonian cycle of the connected graph $G$, we say $G$ to be \emph{edge-hamiltonian}, and moreover we say $G$ to be \emph{edge-hamiltonian-path} if each edge $uv\in E(G)$ is in a hamiltonian path of the connected graph $G$. A \emph{hamiltonian-connected graph} holds that any pair of vertices $w$ and $z$ is connected by a hamiltonian-path $P(w,z)$. A \emph{spanning-tree-connected graph} holds that any two vertices $w$ and $z$ are the leaves of each spanning-tree $T$ of the graph $G$.\qqed
\end{defn}

However, it is not easy to determine whether any connected graph is \emph{edge-spanning-tree} defined in Definition \ref{defn:edge-hamiltonian-connected}.

\begin{problem}\label{qeu:1111}
\textbf{Characterize} edge-hamiltonian graphs and edge-hamiltonian-path graphs.
\end{problem}

Let $G$ be an edge-hamiltonian $(p,q)$-graph (Ref. Definition \ref{defn:edge-hamiltonian-connected}), and $f$ be a vertex coloring from $V(G)$ to $[1,p]$, such that $f(V(G))=[1,p]$. We have a hamiltonian cycle set $E_{hami}(G)$, which contains all hamiltonian cycles of the edge-hamiltonian $(p,q)$-graph $G$. Each 3I-hyperedge set $\mathcal{E}\in \mathcal{E}(E^2_{hami}(G))$ with $\bigcup_{e\in \mathcal{E}}e=E_{hami}(G)$ forms a hypergraph $\mathcal{H}_{yper}=(E_{hami}(G),\mathcal{E})$. We define the \emph{$\Gamma$-operation} by the following operations:
\begin{asparaenum}[\textbf{Ehami}-1]
\item Each hyperedge $e_i\in \mathcal{E}$ corresponds to another hyperedge $e\,'_i\in \mathcal{E}\setminus e_i$, such that $e_i\cap e\,'_i=\{H_{i,1}$, $H_{i,2}$, $\dots $, $H_{i,a_i}\}$ with $a_i\geq 1$, where each $H_{i,j}$ for $j\in [1,a_i]$ is a hamiltonian cycle of the edge-hamiltonian $(p,q)$-graph $G$.
\item A hamiltonian cycle $H_{i}\in e_i$ corresponds to another hamiltonian cycle $H_{j}\in e_j$ holding $H_{k}\subset H_{i}\cup H_{j}$ for some hamiltonian cycle $H_{k}\in e_k\in \mathcal{E}$.
\item A hamiltonian cycle $H_{i}\in e_i$ corresponds to another hamiltonian cycle $H_{j}\in e_j$ holding the \emph{adding-edgeset-removing operation}
$$H_{i}-\{x_1y_1,x_2y_2\}\cong H_{j}-\{u_1v_1,u_2v_2\}
$$ for edges $x_1y_1,x_2y_2\in E(H_{i})$ and $u_1v_1,u_2v_2\in E(H_{j})$.
\end{asparaenum}

Thereby, the $\Gamma$-operation hypergraph $\mathcal{H}_{yper}=(E_{hami}(G),\mathcal{E})$ has its own $\Gamma$-operation graph $H$ admitting a total set-coloring $F:V(H)\cup E(H)\rightarrow \mathcal{E}\in \mathcal{E}(E^2_{hami}(G))$, such that $F(V(H)\cup E(H))=\mathcal{E}$, and the edge color $F(\alpha\beta)$ for each edge $\alpha\beta\in E(H)$ and two vertex colors $F(\alpha)$ and $F(\beta)$ hold one of operations \textbf{Ehami}-$t$ with $t\in [1,3]$ true.

\subsubsection{$G$-hypergraphs}

\begin{defn} \label{defn:G-hypergraphs-abstract-operation}
Let $G_{pub}=\{G_1,G_2,\dots ,G_m\}$ (as a \emph{public-key set}), $H_{pri}=\{H_1,H_2,\dots ,H_m\}$ (as a \emph{private-key set}) and $A_{uth}=\{A_1,A_2$, $\dots $, $ A_m\}$ (as an \emph{authentication set}) be graph sets. There is a $[\ast]$-operation holding $G_i[\ast]H_i=A_i$ true for $i\in [1,m]$. We have a $G$-hypergraph $H_{yper}=(\Psi,\mathcal{E})$ based on a hyperedge set $\mathcal{E}\in \mathcal{E}_{gen}(\Psi^2)$ (Ref. Definition \ref{defn:new-hypergraphs-sets}), where $\Psi=G_{pub}\cup H_{pri}$ is a graph set, such that each $G$-hyperedge set $e\in \mathcal{E}$ corresponds to another $G$-hyperedge set $e\,'\in \mathcal{E}\setminus e$ holding $G_i[\ast]H_i=A_i\in A_{uth}$ for some graph $G_i\in e$ and some graph $H_i\in e\,'$.\qqed
\end{defn}

We show the following examples for understanding Definition \ref{defn:G-hypergraphs-abstract-operation}:
\begin{asparaenum}[\textbf{\textrm{Example}}-1. ]
\item The operation $G_i[\ast]H_i=A_i$ for $i\in [1,m]$ in Definition \ref{defn:G-hypergraphs-abstract-operation} is $G_i\cup H_i=K_{n_i}$ with three vertex sets $V(G_i)=V(H_i)=V(K_{n_i})$ and three edge sets $E(G_i)\cap E(H_i)=\emptyset$ and $E(G_i)\cup E(H_i)=E(K_{n_i})$, where $K_{n_i}$ is a complete graph of $n_i$ vertices, and ``$[\ast]=\cup$'' is the union operation on ordinary graphs.
\item Suppose that a connected $(p,q)$-graph $G$ admits a proper vertex coloring $f:V(G)\rightarrow [1,p]$ such that the vertex color set $f(V(G))=[1,p]$. We vertex-split the connected $(p,q)$-graph $G$ into connected unique-cycle graphs $G_i$ with $i\in [1,n_{vs}(G)]$, where $n_{vs}(G)$ is the number of connected graphs, such that $G_i\not\cong G_j$ and $|E(G_i)|=|E(G)|$, and each graph $G_i$ admits a proper vertex coloring $f_i:V(G_i)\rightarrow [1,p]$, such that $f_i$ is induced by the proper vertex coloring $f$, as well as the vertex color set $f(V(G_i))=[1,p]$. Let $\Gamma=S_{cycle}(G)$ contain these connected unique-cycle graphs $G_i$ with $i\in [1,n_{vs}(G)]$. We get a \emph{$G$-hypergraph} $H_{yper}=(\Gamma,\mathcal{E})$ based on a 3I-hyperedge set $\mathcal{E}\in \mathcal{E}(\Gamma^2)$, such that a graph $H$ admits a vertex coloring $F:V(H)\rightarrow \mathcal{E}$ holding $F(u)[\bullet ]F(v)$ for each edge $uv\in E(H)$ based on a graph operation ``$[\bullet ]$''. For example, each connected unique-cycle graph $T_u\in F(u)$ corresponds to another connected unique-cycle graph $T_v\in F(v)$ such that the operation ``$T_u[\bullet ]T_v$'' is the \emph{adding-edge-removing operation} defined as follows
\begin{equation}\label{eqa:555555}
T_v=T_u+xy-st,~st\in E(T_u)\textrm{ and }xy\in E(T_v)
\end{equation}
\item There is another set $S_{tree}(G)$ being the set of trees of $q$ edges obtained by vertex-splitting a connected $(p,q)$-graph $G$ into the trees of $q$ edges, such that each tree $T_i\in S_{tree}(G)$ holds $|E(T_i)|=|E(G)|$. Let $|S_{tree}(G)|=n_{tree}(G)$, and let $\Omega=S_{tree}(G)\cup C_{omp}(G)$, where the complementary graph set $C_{omp}(G)=\{L_i:i\in [1,n_{tree}(G)]\}$ holding $T_i\cup L_i=K_{q+1}$. We obtain a \emph{$G$-hypergraph} $H_{yper}=(\Omega,\mathcal{E})$ based on an IOI-hyperedge set $\mathcal{E}\in \mathcal{E}(\Omega^2)$, such that each $G$-hyperedge set $e\in \mathcal{E}$ corresponds to another $G$-hyperedge set $e\,'\in \mathcal{E}\setminus e$ holding $T_i\cup L_i=K_{q+1}$ for some tree $T_i\in e$ and some complementary graph $L_i\in e\,'$.
\item The graphs $G_i,H_i$ and $A_i$ in the graph sets $G_{pub}$, $H_{pri}$ and $A_{uth}$ defined in Definition \ref{defn:G-hypergraphs-abstract-operation} are hamiltonian graphs. The operation $G_i[\ast]H_i=A_i$ for $i\in [1,m]$ is defined as follows: $C_{G}=x_1x_2\cdots x_ax_1$ is a hamiltonian cycle of $G_i$, and $C_{H}=y_1y_2\cdots y_by_1$ is a hamiltonian cycle of $H_i$. We join $x_1$ with $y_1$ by an edge $x_1y_1$, and join $x_a$ with $y_b$ by an edge $x_ay_b$, the resultant graph is just the graph $A_i=C_{G}[\ominus]C_{H}$, and the graph $A_i$ has a hamiltonian cycle $C_{G[\ominus]H}=x_1x_2\dots x_ay_by_{b-1}\cdots y_2y_1x_1$, such that each $G$-hyperedge $e\in \mathcal{E}$ corresponds to another $G$-hyperedge $e\,'\in \mathcal{E}\setminus e$ holding $C_{G}[\ominus]C_{H}=C_{G[\ominus]H}$ for some tree $G_i\in e$ and some complementary graph $H_i\in e\,'$.
\item \label{exa:example-555} Each graph in the graph sets $G_{pub}$, $H_{pri}$ and $A_{uth}$ defined in Definition \ref{defn:G-hypergraphs-abstract-operation} is hamiltonian-connected defined in Definition \ref{defn:edge-hamiltonian-connected}, i.e., there is a hamiltonian path $P(x,y)$ for any pair of vertices $x$ and $y$ in the graphs $G_i,H_i$ and $A_i$ repectively. The operation $G_i[\ast]H_i=A_i$ for $i\in [1,m]$ is defined as follows: $P_{G}(x_1,x_a)=x_1x_2\cdots x_a$~(denoted as $x_1\rightarrow_{hami} x_a$) is a hamiltonian path of $G_i$, and $P_{H}(y_1,y_b)=y_1y_2\cdots y_b$~(denoted as $y_1\rightarrow_{hami} y_b$) is a hamiltonian path of $H_i$. Let $N_{ei}(x_1)=\{u_1,u_2,\dots ,u_r\}$ and $N_{ei}(y_1)=\{v_1,v_2,\dots ,v_s\}$.

\qquad (i) We join the vertex $y_b$ with each vertex $u_i\in N_{ei}(x_1)$ by edges $y_bu_i$, and join the vertex $x_a$ with each vertex $v_j\in N_{ei}(y_1)$ by edges $x_av_j$.

\qquad (ii) We join the vertex $x_a$ with the vertex $y_b$ by an edge $x_ay_b$, and join the vertex $x_1$ with the vertex $y_1$ by an edge $x_1y_1$.

\qquad The resultant graph is just $A_i=G_i[\ominus_{path}]H_i$, and the graph $A_i$ has a hamiltonian cycle $P_{G}(x_1,x_a)+x_1y_1+x_ay_b+P_{H}(y_1,y_b)$. We show the graph $A_i$ to be hamiltonian-connected.

\qquad \textbf{Case-1.} For two vertices $x_{i_r},x_{i_t}\in V(G_i)\subset V(A_i)$, there is a hamiltonian path
\begin{equation}\label{eqa:555555}
P(x_{i_r},x_{i_t})=x_{i_r}x_{i_r+1}\cdots x_1u^*\cdots x_a\cdots x_{i_t}\textrm{ (denoted as }x_{i_r}\rightarrow_{hami} x_{i_t})
\end{equation} of the graph $G_i$, where $u^*\in N_{ei}(x_1)$, we rewrite this hamiltonian path by $x_{i_r}\rightarrow_{hami} x_{i_t}$. We can construct a hamiltonian path $P_{A_i}(x_{i_r},x_{i_t})$ of the graph $A_i$ as follows
\begin{equation}\label{eqa:555555}
P_{A_i}(x_{i_r},x_{i_t}):~x_{i_r}\rightarrow x_1\rightarrow y_1\rightarrow_{hami} y_b\rightarrow u^*\rightarrow x_a\rightarrow x_{i_t}
\end{equation} See Fig.\ref{fig:hami-connected}(a).

\qquad \textbf{Case-2.} For any vertex $x_{i_r}\in V(G_i)\subset V(A_i)$ and an arbitrary vertex $y_{j_r}\in V(H_i)\subset V(A_i)$, the graph $A_i$ has a hamiltonian path $x_{i_r}\rightarrow_{hami} x_1\rightarrow y_1\rightarrow_{hami} y_{j_r}$ (see Fig.\ref{fig:hami-connected}(a)).
\end{asparaenum}

The graph $A_i$ can be constructed more complex, see Fig.\ref{fig:hami-connected}(b), where two graphs $T_i$ and $L_i$ are hamiltonian-connected.

\begin{figure}[h]
\centering
\includegraphics[width=16.4cm]{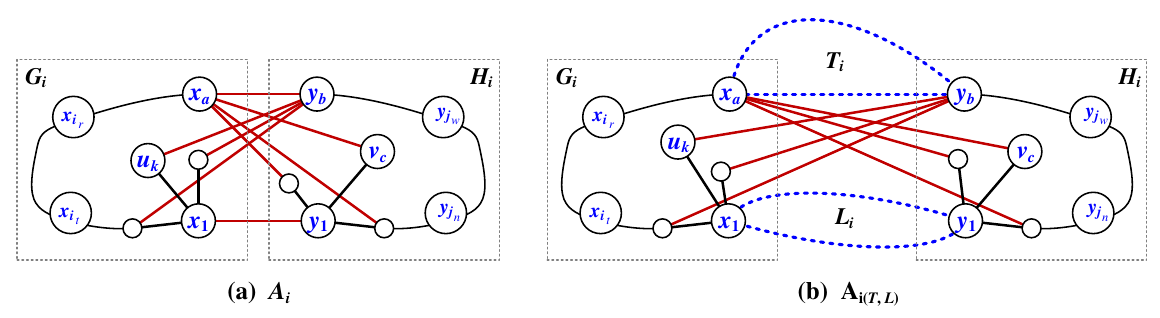}\\
\caption{\label{fig:hami-connected}{\small A scheme for illustrating the $W$-constraint hyperedge sets.}}
\end{figure}

\begin{thm}\label{thm:hami-connected-string-than-edge-hami}
For hamiltonian-connected graphs and edge-hamiltonian graphs, we have the following connections:

(i) If a graph $G$ is hamiltonian-connected, then the graph $G$ is edge-hamiltonian.

(ii) If a graph $H$ is not edge-hamiltonian, then the graph $H$ is not hamiltonian-connected.
\end{thm}
\begin{proof} (i) Since any edge $uv$ corresponds to a hamiltonian path $P(u,v)$, then $P(u,v)+uv$ is a hamiltonian cycle.

(ii) Since the graph $H$ is not edge-hamiltonian, then there is an edge $xy\in E(H)$, such that there is no hamiltonian cycle containing the edge $xy$, in other words, there is no hamiltonian path $P(x,y)$ in the graph $H$.
\end{proof}

\begin{prop}\label{proposition:99999}
In the above Example \ref{exa:example-555}, the methods of constructing a hamiltonian-connected graph $A_i$ can be used to show: \emph{No necessary and sufficient condition for determining whether a graph is a hamiltonian-connected graph}.
\end{prop}

\begin{figure}[h]
\centering
\includegraphics[width=13cm]{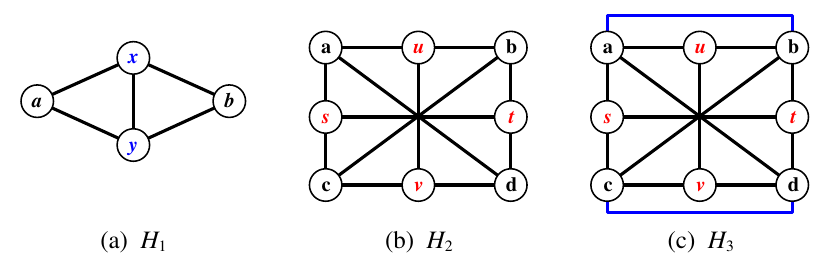}\\
\caption{\label{fig:3-hami-properties}{\small Three hamiltonian properties.}}
\end{figure}

\begin{rem}\label{rem:333333}
Let $P_1$ be the hamiltonian-connected property, $P_2$ be the spanning-tree-connected property, and $P_3$ be the edge-hamiltonian property. For hamiltonian property set $P_{hami}=\big \{P^1_{hami}$, $P^2_{hami}$, $P^3_{hami}\big \}$, we have a graph $G$ is a $P^i_{hami}$-graph with $i\in [1,3]$, or the graph $G$ is a total-$P_{hami}$-graph. By Definition \ref{defn:edge-hamiltonian-connected} and Theorem \ref{thm:hami-connected-string-than-edge-hami}, we can see that $P^1_{hami}$ is stronger than $P^j_{hami}$ with $j=2,3$. If $P^j_{hami}$ with $j=2,3$ is not true, then $P^1_{hami}$ is not true too.

In Fig.\ref{fig:3-hami-properties}, the graph $H_1$ is not a total-$P_{hami}$-graph; the graph $H_2$ is not a $P^1_{hami}$-graph, but is a $P^i_{hami}$-graph with $i=2,3$; and the graph $H_3$ is a total-$P_{hami}$-graph.

\textbf{Characterize} total-$P_{hami}$-graphs. \qqed
\end{rem}

\begin{defn} \label{defn:p-property-hypergraphs}
\textbf{The $P$-property hypergraph.} Let $\mathcal{P}=\{P_1,P_2,\dots ,P_m\}$ be a graphic property set, and $\textrm{\textbf{G}}=\{G_1,G_2,\dots ,G_n\}$ be a graph set. For a hyperedge set $\mathcal{E}$ with $\bigcup_{e\in \mathcal{E}}e=\mathcal{P}$, a graph $H$ has its own vertex set $V(H)=G$, and admits a vertex set-coloring $\varphi:V(H)\rightarrow \mathcal{E}$, such that each graph $G_i\in G=V(H)$ has each one of the property subset $\varphi(G_i)\in \mathcal{E}$, and $\varphi(u)\cap \varphi(v)=e_u\cap e_v\neq \emptyset $ for each edge $uv\in E(H)$ and $e_u, e_v\in \mathcal{E}$. Conversely, if a hyperedge $e\in \mathcal{E}$ corresponds to another hyperedge $e\,'\in \mathcal{E}\setminus e$ holding $e\cap e\,'\neq \emptyset $, then there is an edge $xy\in E(H)$ holding $\varphi(x)=e$ and $\varphi(y)=e\,'$, so $H$ is a v-intersected graph of the \emph{$P$-property hypergraph} $H_{yper}=(\mathcal{P},\mathcal{E})$. \qqed
\end{defn}

\subsection{Code-hypergraphs}

Graph coloring is one of the important techniques in topology code theory (TCT), and is also the main method of asymmetric topology encryption.

\subsubsection{$C_{olor}$-hypergraphs}

\begin{defn} \label{defn:coloring-graph-set-hypergraph}
Suppose that a graph $G$ admits colorings $f_1,f_2,\dots ,f_n$ with \textbf{Independence} $f_i\neq f_j$ if $i\neq j$, and each coloring $f_i$ corresponds to another coloring $f_j$ with $i\neq j$ such that there is a transformation $\theta_{i,j}$ holding \textbf{Operation} $f_j=\theta_{i,j}(f_i)$. Let $\Lambda_{C}=\{f_1,f_2,\dots ,f_n\}$, and each IOI-hyperedge set $\mathcal{E}_{C}\in \mathcal{E}(\Lambda^2_{C})$ holds \textbf{Integrity} $\Lambda_{C}=\bigcup_{e\in \mathcal{E}_{C}}e$, we get a \emph{$C_{olor}$-hypergraph} $H_{yper}=(\Lambda_{C},\mathcal{E}_{C})$ if a hyperedge $e\in \mathcal{E}_{C}$ corresponds to another hyperedge $e\,'\in \mathcal{E}_{C}\setminus e$ such that $f\in e$ and $f\,'\in e\,'$ hold $f\,'=\theta(f\,')$.

The Topcode-matrix set $\Lambda_{matrix}=\{T_{code}(G,f_i):i\in [1,n]\}$ forms a Topcode-matrix hypergraph $H^{matrix}_{yper}=(\Lambda_{matrix},\mathcal{E})$ for an IOI-hyperedge set $\mathcal{E}\in \mathcal{E}(\Lambda^2_{matrix})$.

The Topcode-matrix graph set $\Lambda_{graph}=\{G_{raph}(T_{code}(G,f_i)):i\in [1,n]\}$ forms a graph-set hypergraph $H^{graph}_{yper}=(\Lambda_{graph},\mathcal{E})$ for an IOI-hyperedge set $\mathcal{E}\in \mathcal{E}\big (\Lambda^2_{graph}\big )$.\qqed
\end{defn}

\begin{thm}\label{thm:666666}
Suppose that $\Lambda=\{g_i:i\in [1,m]\}$ is a coloring set, and there is a transformation $\theta_{i,j}$ holding $g_j=\theta_{i,j}(g_i)$ for any pair of colorings $g_i,g_j\in \Lambda$ (Ref. Definition \ref{defn:coloring-graph-set-hypergraph}). Then each connected graph $G$ admits a set-coloring $F:V(G)\rightarrow \mathcal{E}$ for an IOI-hyperedge set $\mathcal{E}\in \mathcal{E}\big (\Lambda^2\big )$, such that $F(u)=g_i$, $F(v)=g_j$ and $F(v)=\theta_{i,j}(F(u))$ for each edge $uv\in E(G)$.
\end{thm}

\begin{problem}\label{question:assembl-graphs-hypergraphs}
A total coloring partition the vertices and edges of the graph into \emph{total independent sets}. For a graph $G$, we

(i) let $V_i=\{u_{i,1},u_{i,2},\dots ,u_{i,a_i}\}$ be an independent vertex set of $V(G)$, such that each vertex $u_{i,j}\in V_i$ with $j\in [1,a_i]$ is colored with color $i$.

(ii) let $E_i=\{e_{i,1},e_{i,2},\dots ,e_{i,b_i}\}$ be an independent edge set of $E(G)$, such that each edge $e_{i,j}\in E_i$ with $j\in [1,b_i]$ is colored with $i$th color, but two ends of the edge $e_{i,j}$ are not colored.

For the $i$-color ve-set $S_i=V_i\cup E_i$ with $i\in [1,k]$ defined above, we get a hypergraph $\mathcal{H}_{yper}=(\Lambda(G),\mathcal{E})$ with its own vertex set $\Lambda(G)=V(G)\cup E(G)$ and its own hyperedge set $\mathcal{E}=\bigcup^k_{i=1}S_i$ holding $\Lambda(G)=\bigcup_{S_i\in \mathcal{E}}S_i$. \textbf{Dose} $G$ admit a proper total coloring $f:V(G)\cup E(G)\rightarrow [1,k]$, such that the $i$-color ve-set $S_i=\{w:~w\in V(G)\cup E(G)\textrm{ and }f(w)=i\}$ for each $i\in [1,k]$?

Without doubt, there are some hyperedge sets $\mathcal{E}^a\in \mathcal{E}_{gen}\big (\Lambda^2(G)\big )$ with $a\in [1,m]$, such that each hypergraph $\mathcal{H}^a_{yper}=(\Lambda(G),\mathcal{E}^a)$ produces a proper total coloring of the graph $G$, where each hyperedge set $\mathcal{E}^a=\bigcup^k_{i=1} S^a_i=\bigcup^k_{i=1} (V^a_i\cup E^a_i)$ with $a\in [1,m]$, and moreover subsets $V^a_i\subset V(G)$ and $E^a_i\subset E(G)$ are independent sets in the graph $G$.
\end{problem}

\begin{problem}\label{qeu:444444}
A $k$-level $T$-tree $H$ is a tree, where the tree $T$ is the root tree, such that each \emph{leaf-removing tree} $H_{i-1}=H_{i}-L(H_{i})$ for $i\in [1,k]$ with $k\geq 1$, and $H_{0}=T$, \textbf{characterize} $k$-level $T$-trees with $k\geq 1$.
\end{problem}

\subsubsection{$D_{nei}$-hypergraphs}

\begin{defn} \label{defn:dist-nei-hypergraphs}
Suppose that a $(p,q)$-graph $G$ admits a total coloring $f:V(G)\cup E(G)\rightarrow S_{thing}$ holding the total color set $f(V(G)\cup E(G))=S_{thing}$, such that $f(u)$, $f(uv)$ and $f(v)$ for each edge $uv\in E(G)$ satisfy a constraint set $R_{est}(c_1,c_2,\dots ,c_m)$. There are four neighbor sets for each vertex $u$ of the graph $G$ as follows:
\begin{equation}\label{eqa:555555}
{
\begin{split}
C_v(u)&=\{f(v):v\in N_{ei}(u)\},~C_v[u]=C_v(u)\cup \{f(u)\}\\
C_e(u)&=\{f(uv):v\in N_{ei}(u)\},~C_e[u]=C_e(u)\cup \{f(u)\}
\end{split}}
\end{equation}
We have a $3$I-hyperedge set $\mathcal{E}_{ei}\in \mathcal{E}\big (S^2_{thing}\big )$ as follows
\begin{equation}\label{eqa:555555}
\mathcal{E}_{ei}=\{C_v(u),C_v[u],C_e(u),C_e[u]:u\in V(G)\}
\end{equation} Clearly, the total color set $f(V(G)\cup E(G))=S_{thing}=\bigcup_{e\in \mathcal{E}_{ei}}e$. Then the graph $G$ admits a total hyperedge-set coloring $F:V(G)\rightarrow \mathcal{E}_{ei}$ holding one of the following neighbor cases
\begin{asparaenum}[\textbf{\textrm{Nei}}-1.]
\item (v-neighbor) $F(u)=C_v[u]$ for each vertex $u\in V(G)$
\item (e-neighbor) $F(u)=C_e(u)$ for each vertex $u\in V(G)$
\item (ve-neighbor) $F(u)=C_e(u)\cup C_v(u)$ for each vertex $u\in V(G)$
\item (all-neighbor) $F(u)=C_e(u)\cup C_v(u)\cup \{f(u)\}$ for each vertex $u\in V(G)$
\end{asparaenum}
and obey the constraint set $R^*_{est}(a_1,a_2,\dots ,a_n)$, then we say that $G$ is \emph{neighbor-set graph} of the \emph{$D_{nei}$-hypergraph} $H_{yper}=(S_{thing},\mathcal{E}_{ei})$.\qqed
\end{defn}

\begin{thm}\label{thm:666666}
Each $W$-constraint coloring $f$ of a graph $G$ corresponds to a $D_{nei}$-hypergraph $H_{yper}=(\bigwedge_f,\mathcal{E}_{f})$ with its own vertex set $\bigwedge_f=f(S)$ for $S\subseteq V(G)\cup E(G)$ and a hyperedge set $\mathcal{E}_{f}\in \mathcal{E}_{gen}\big (\bigwedge^2_f\big )$, such that each hyperedge $e\in \mathcal{E}_{f}$ corresponds to another hyperedge $e\,'\in \mathcal{E}_{f}\setminus e$ holding $e\cap e\,'\neq\emptyset$.
\end{thm}

\begin{example}\label{exa:8888888888}
Suppose that the $(p,q)$-graph $G$ is a graph appeared in Definition \ref{defn:dist-nei-hypergraphs}, we present the following examples for constructing $D_{nei}$-hypergraphs:
\begin{asparaenum}[\textbf{\textrm{Dnei}}-1.]
\item $[1,M]$ is a consecutive integer set. About the constraint set $R_{est}(c_1,c_2$, $\dots $, $c_m)$, there are:

\qquad $c_1:f(u)\neq f(v)$ for each edge $uv\in E(G)$;

\qquad $c_2:f(uv)\neq f(uw)$ for neighbors $v,w\in N_{ei}(u)$ and $u\in V(G)$.

\qquad About the constraint set $R^*_{est}(a_1,a_2,\dots ,a_n)$, there are:

\qquad $a_1:F(x)\supset C_v(x)\neq C_v(y)\subset F(y)$ for each edge $xy\in E(G)$ with degrees $\textrm{deg}_G(u)\geq 2$ and $\textrm{deg}_G(v)\geq 2$;

\qquad $a_2:F(x)\supset C_e(x)\neq C_e(y)\subset F(y)$ for each edge $xy\in E(G)$ with degrees $\textrm{deg}_G(u)\geq 2$ and $\textrm{deg}_G(v)\geq 2$.

\qquad Hence, the graph $G$ is called \emph{adjacent ve-neighbor distinguishing graph} of the \emph{$D_{nei}$-hypergraph} $H_{yper}=([1,M]$, $\mathcal{E}_{ei})$ based on the hyperedge-set coloring $F$ defined in Definition \ref{defn:dist-nei-hypergraphs}. As $M=\chi\,''(G)$ is the total chromatic number, it is not easy to determine the total coloring $f$ in Definition \ref{defn:dist-nei-hypergraphs}.

\item $[1,p+q]$ is a consecutive integer set, and $G$ is a $(p,q)$-graph. About the constraint set $R_{est}(c_1,c_2,\dots ,c_m)$, there are:

\qquad $c_1:f(u)\neq f(x)$ for distinct vertices $u,x\in V(G)$;

\qquad $c_1:f(uv)\neq f(xy)$ for distinct edges $uv,xy\in E(G)$;

\qquad $c_3:f(u)+f(uv)+f(v)=k$ for each edge $uv\in E(G)$.

\qquad About the constraint set $R^*_{est}(a_1,a_2,\dots ,a_n)$, there are:

\qquad $a_1:F(u)\cap F(v)\neq \emptyset$ for each edge $uv\in E(G)$.

\qquad $a_2:C_e(u)=F(u)\neq F(v)=C_e(v)$ for each edge $uv\in E(G)$ with degrees $\textrm{deg}_G(u)\geq 2$ and $\textrm{deg}_G(v)\geq 2$.

\qquad So, the $(p,q)$-graph $G$ is called \emph{adjacent e-neighbor distinguishing edge-magic graph} of the \emph{$D_{nei}$-hypergraph} $H_{yper}=([1,p+q]$, $\mathcal{E}_{ei})$ based on the hyperedge-set coloring $F$ defined in Definition \ref{defn:dist-nei-hypergraphs}.

\item $[0,q]$ is a consecutive integer set, and $G$ is a $(p,q)$-graph. About the constraint set $R_{est}(c_1,c_2,\dots ,c_m)$, there are:

\qquad $c_1:f(u)\neq f(x)$ for distinct vertices $u,x\in V(G)$;

\qquad $c_2:f(uv)\neq f(xy)$ for distinct edges $uv,xy\in E(G)$;

\qquad $c_3:f(uv)=|f(u)-f(v)|$ for each edge $uv\in E(G)$;

\qquad $c_4:f(E(G))=\{f(uv):uv\in E(G)\}=[1,q]$.

\qquad About the constraint set $R^*_{est}(a_1,a_2,\dots ,a_n)$, there are:

\qquad $a_1:F(u)\cap F(v)\neq \emptyset$ for each edge $uv\in E(G)$.

\qquad $a_2:C_e(u)\cup C_v(u)\cup \{f(u)\}=F(u)\neq F(v)=C_e(v)\cup C_v(v)\cup \{f(v)\}$ for each edge $uv\in E(G)$ with degrees $\textrm{deg}_G(u)\geq 2$ and $\textrm{deg}_G(v)\geq 2$.

\qquad Thereby, we say the $(p,q)$-graph $G$ to be \emph{adjacent all-neighbor distinguishing graceful graph} of the \emph{$D_{nei}$-hypergraph} $H_{yper}=([0,q]$, $\mathcal{E}_{ei})$ based on the hyperedge-set coloring $F$ defined in Definition \ref{defn:dist-nei-hypergraphs}.

\item $[0,q]$ is a consecutive integer set, and $G$ is a bipartite $(p,q)$-graph with $V(G)=X\cup Y$ and $X\cup Y=\emptyset$. About the constraint set $R_{est}(c_1,c_2,\dots ,c_m)$, there are:

\qquad $c_1:u\in X$ and $v\in Y$ for each edge $uv\in E(G)$;

\qquad $c_2:f(u)\neq f(x)$ for distinct vertices $u,x\in V(G)$;

\qquad $c_3:f(uv)\neq f(xy)$ for distinct edges $uv,xy\in E(G)$;

\qquad $c_4:f(uv)=|f(u)-f(v)|$ for each edge $uv\in E(G)$;

\qquad $c_5:f(E(G))=\{f(uv):uv\in E(G)\}=[1,q]$;

\qquad $c_6:$ the set-ordered constraint $\max f(X)<\min f(Y)$ holds true.

\qquad About the constraint set $R^*_{est}(a_1,a_2,\dots ,a_n)$, there are:

\qquad $a_1:F(u)\cap F(v)\neq \emptyset$ for each edge $uv\in E(G)$.

\qquad $a_2:C_e(u)\cup C_v(u)\cup \{f(u)\}=F(u)\neq F(v)=C_e(v)\cup C_v(v)\cup \{f(v)\}$ for each edge $uv\in E(G)$ with degrees $\textrm{deg}_G(u)\geq 2$ and $\textrm{deg}_G(v)\geq 2$.

\qquad Thereby, the bipartite $(p,q)$-graph $G$ is called \emph{adjacent all-neighbor distinguishing set-ordered graceful graph} of the \emph{$D_{nei}$-hypergraph} $H_{yper}=([0,q]$, $\mathcal{E}_{ei})$ based on the hyperedge-set coloring $F$ defined in Definition \ref{defn:dist-nei-hypergraphs}.\qqed
\end{asparaenum}
\end{example}

\begin{cor}\label{cor:666666}
Since each tree $T$ of $q$ edges admits a set-ordered gracefully total coloring $f$ by Theorem \ref{thm:tree-graceful-total-coloringss}, then the tree $T$ induces a $D_{nei}$-hypergraph $H_{yper}=([0,q],\mathcal{E}_{ei})$ by Definition \ref{defn:dist-nei-hypergraphs}.
\end{cor}

\begin{thm}\label{thm:graceful-total-sequence-coloring}
\cite{Yao-Su-Wang-Hui-Sun-ITAIC2020} Every tree $T$ with diameter $D(T)\geq 3$ and $s+1=\left \lceil \frac{D(T)}{2}\right \rceil $ admits at least $2^{s}$ different \emph{gracefully total sequence colorings} if two sequences $A_M, B_q$ holding $0<b_j-a_i\in B_q$ for $a_i\in A_M$ and $b_j\in B_q$.
\end{thm}

By Theorem \ref{thm:graceful-total-sequence-coloring} and Definition \ref{defn:dist-nei-hypergraphs}, we have:

\begin{cor}\label{cor:666666}
Each tree $T$ with diameter $D(T)\geq 3$ and $s+1=\left \lceil \frac{D(T)}{2}\right \rceil $ induces at least $2^{s}$ different $D_{nei}$-hypergraphs $H^i_{yper}=(A_M\cup B_q,\mathcal{E}^i_{ei})$ for $i\in [1,2^{s}]$.
\end{cor}

\subsubsection{$W$-constraint hyperedge sets}

We construct $W$-constraint hyperedge sets by some parameterized colorings in this subsection.

\begin{example}\label{exa:8888888888}
Suppose that a bipartite $(p,q)$-graph $H$ admits a total coloring $f:V(H)\cup E(H)\rightarrow [0,q]$ subject to a constraint set $R_{est}(c_1,c_2,c_3,c_4)$. Since $V(H)=X\cup Y$ and $X\cap Y=\emptyset $, the coloring $f$ holds the constraints of the constraint set $R_{est}(c_1,c_2,c_3,c_4)$ as follows:

$c_1:$ The \emph{labeling constraint} $f(u)\neq f(w)$ for distinct vertices $u,w\in V(H)$, and $|f(V(H))|=p$;

$c_2:$ the \emph{set-ordered constraint} $\max f(X)<\min f(X)$;

$c_3:$ the \emph{$W$-constraint} $W[f(x),f(xy),f(y)]=\big [f(y)-f(x)\big ]-f(xy)=0$ for each edge $xy\in E(H)$ with $x\in X$ and $y\in Y$; and

$c_4:$ the edge color set holds the graceful constraint
$$
f(E(H))=[1,q]=\{f(xy)=f(y)-f(x):x\in X,y\in Y,xy\in E(H)\}
$$ true.

Also, the coloring $f$ is called \emph{set-ordered graceful labeling}. \qqed
\end{example}

\begin{thm}\label{thm:graph-admits-6-set-colorings}
\cite{Bing-Yao-arXiv:2207-03381} Each connected graph $G$ admits each one of the following $W$-constraint $(k,d)$-total set-colorings for $W$-constraint $\in \{$graceful, harmonious, edge-difference, edge-magic, felicitous-difference, graceful-difference$\}$.
\end{thm}
\begin{proof} Since a connected graph $G$ can be vertex-split into a tree $T$, such that we have a graph homomorphism $T\rightarrow G$ under a mapping $\theta:V(T)\rightarrow V(G)$, and each tree admits a graceful $(k,d)$-total coloring $g_1$, a harmonious $(k,d)$-total coloring $g_2$, an edge-difference $(k,d)$-total coloring $g_3$, an edge-magic $(k,d)$-total coloring $g_4$, a felicitous-difference $(k,d)$-total coloring $g_5$ and a graceful-difference $(k,d)$-total coloring $g_6$. Then the connected graph $G$ admits a $(k,d)$-total set-coloring $F$ defined by $F(u)=\{g_i(u^*):i\in [1,6],u^*\in V(T)\}$ for each vertex $u\in V(G)$ with $u=\theta(u^*)$ for $u^*\in V(T)$, and $F(uv)=\{g_i(u^*v^*):i\in [1,6],u^*v^*\in E(T)\}$ for each edge $uv\in E(G)$ with $uv=\theta(u^*)\theta(v^*)$ for each edge $u^*v^*\in E(T)$.

The proof of the theorem is complete.
\end{proof}

The following Definition \ref{defn:odd-even-separable-6C-labeling} shows us a constraint set $R_{est}(c_1,c_2,\dots, c_8)$:

\begin{defn}\label{defn:odd-even-separable-6C-labeling}
\cite{Yao-Sun-Zhang-Mu-Sun-Wang-Su-Zhang-Yang-Yang-2018arXiv} A total labeling $f:V(G)\cup E(G)\rightarrow [1,p+q]$ for a bipartite $(p,q)$-graph $G$ is a bijection and holds the following constraints:
\begin{asparaenum}[(i) ]
\item (total color set) $f(V(G)\cup E(G))=[1,p+q]$;
\item (e-magic) $f(uv)+|f(u)-f(v)|=k$;

\item (ee-difference) each edge $uv$ matches with another edge $xy$ holding one of $f(uv)=|f(x)-f(y)|$ and $f(uv)=2(p+q)-|f(x)-f(y)|$ true;

\item (ee-balanced) let $s(uv)=|f(u)-f(v)|-f(uv)$ for $uv\in E(G)$, then there exists a constant $k\,'$ such that each edge $uv$ matches with another edge $u\,'v\,'$ holding one of $s(uv)+s(u\,'v\,')=k\,'$ and $2(p+q)+s(uv)+s(u\,'v\,')=k\,'$ true;

\item (EV-ordered) $\min f(V(G))>\max f(E(G))$ (resp. $\max f(V(G))<\min f(E(G))$, or $f(V(G))$ $\subseteq f(E(G))$, or $f(E(G))$ $\subseteq f(V(G))$, or $f(V(G))$ is an odd-set and $f(E(G))$ is an even-set);

\item (ve-matching) there exists a constant $k\,''$ such that each edge $uv$ matches with one vertex $w$ such that $f(uv)+f(w)=k\,''$, and each vertex $z$ matches with one edge $xy$ such that $f(z)+f(xy)=k\,''$, except the \emph{singularity} $f(x_0)=\lfloor \frac{p+q+1}{2}\rfloor $;

\item (set-ordered constraint) $\max f(X)<\min f(Y)$ (resp. $\min f(X)>\max f(Y)$) for the bipartition $(X,Y)$ of $V(G)$.

\item (odd-even separable) $f(V(G))$ is an odd-set containing only odd numbers, as well as $f(E(G))$ is an even-set containing only even numbers.
\end{asparaenum}

We call $f$ \emph{odd-even separable 6C-labeling}.\qqed
\end{defn}

\begin{defn} \label{defn:kd-w-type-colorings}
\cite{Yao-Su-Wang-Hui-Sun-ITAIC2020} Let $G$ be a bipartite and connected $(p,q)$-graph, then its vertex set $V(G)=X\cup Y$ with $X\cap Y=\emptyset$ such that each edge $uv\in E(G)$ holds $u\in X$ and $v\in Y$. Let integers $a,k,m\geq 0$, $d\geq 1$ and $q\geq 1$. We have two parameterized sets
\begin{equation}\label{eqa:555555}
{
\begin{split}
S_{m,k,a,d}=& \big \{k+ad,k+(a+1)d,\dots ,k+(a+m)d\big \},\\
O_{2q-1,k,d}=& \big \{k+d,k+3d,\dots ,k+(2q-1)d\big \}
\end{split}}
\end{equation} with two cardinalities $|S_{m,k,a,d}|=m+1$ and $|O_{2q-1,k,d}|=q$. Suppose that the bipartite and connected $(p,q)$-graph $G$ admits a coloring
\begin{equation}\label{eqa:555555}
f:X\rightarrow S_{m,0,0,d}=\big \{0,d,\dots ,md\big \},~f:Y\cup E(G)\rightarrow S_{n,k,0,d}=\big \{k,k+d,\dots ,k+nd\big \}
\end{equation} with integers $k\geq 0$ and $d\geq 1$, here it is allowed $f(x)=f(y)$ for some distinct vertices $x,y\in V(G)$. Let $c$ be a non-negative integer. We define the following \emph{parameterized colorings}:
\begin{asparaenum}[\textbf{\textrm{Ptol}}-1. ]
\item If the edge color $f(uv)=|f(u)-f(v)|$ for each edge $uv\in E(G)$, and two color sets
\begin{equation}\label{eqa:555555}
f(E(G))=S_{q-1,k,0,d},\quad f(V(G)\cup E(G))\subseteq S_{m,0,0,d}\cup S_{q-1,k,0,d}
\end{equation} then $f$ is called a \emph{$(k,d)$-gracefully total coloring}; and moreover $f$ is called a \emph{strongly $(k,d)$-graceful total coloring} if $f(x)+f(y)=k+(q-1)d$ for each matching edge $xy$ of a matching $M$ of the graph $G$.
\item If the edge color $f(uv)=|f(u)-f(v)|$ for each edge $uv\in E(G)$, and two color sets
$$f(E(G))=O_{2q-1,k,d},~f(V(G)\cup E(G))\subseteq S_{m,0,0,d}\cup S_{2q-1,k,0,d}$$ then $f$ is called a \emph{$(k,d)$-odd-gracefully total coloring}; and moreover $f$ is called a \emph{strongly $(k,d)$-odd-graceful total coloring} if $f(x)+f(y)=k+(2q-1)d$ for each matching edge $xy$ of a matching $M$ of the graph $G$.
\item If there is a color set
$$\big \{f(u)+f(uv)+f(v):uv\in E(G)\big \}=\big \{2k+2ad,2k+2(a+1)d,\dots ,2k+2(a+q-1)d\big \}
$$ with $a\geq 0$ and the total color set $f(V(G)\cup E(G))\subseteq S_{m,0,0,d}\cup S_{2(a+q-1),k,a,d}$, then $f$ is called a \emph{$(k,d)$-edge antimagic total coloring}.
\item If the edge color $f(uv)=f(u)+f(v)~(\bmod^*qd)$ defined by
\begin{equation}\label{eqa:555555}
f(uv)-k=\big [f(u)+f(v)-k\big ]~(\bmod ~qd),~uv\in E(G)
\end{equation} and the edge color set $f(E(G))=S_{q-1,k,0,d}$, then we call $f$ \emph{$(k,d)$-harmonious total coloring}.
\item If the edge color $f(uv)=f(u)+f(v)~(\bmod^*qd)$ defined by $f(uv)-k=[f(u)+f(v)-k](\bmod ~qd)$ for each edge $uv\in E(G)$, and the edge color set $f(E(G))=O_{2q-1,k,d}$, then we call $f$ \emph{$(k,d)$-odd-elegant total coloring}.
\item If the \emph{edge-magic constraint} $f(u)+f(uv)+f(v)=c$ for each edge $uv\in E(G)$, the edge color set $f(E(G))=S_{q-1,k,0,d}$, and the vertex color set $f(V(G))\subseteq S_{m,0,0,d}\cup S_{q-1,k,0,d}$, then $f$ is called \emph{strongly edge-magic $(k,d)$-total coloring}; and moreover $f$ is called \emph{edge-magic $(k,d)$-total coloring} if the cardinality $|f(E(G))|\leq q$ and $f(u)+f(uv)+f(v)=c$ for each edge $uv\in E(G)$.
\item If the \emph{edge-difference constraint} $f(uv)+|f(u)-f(v)|=c$ for each edge $uv\in E(G)$ and the edge color set $f(E(G))=S_{q-1,k,0,d}$, then $f$ is called \emph{strongly edge-difference $(k,d)$-total coloring}; and moreover $f$ is called \emph{edge-difference $(k,d)$-total coloring} if the cardinality $|f(E(G))|\leq q$ and $f(uv)+|f(u)-f(v)|=c$ for each edge $uv\in E(G)$.
\item If the \emph{felicitous-difference constraint} $|f(u)+f(v)-f(uv)|=c$ for each edge $uv\in E(G)$ and the edge color set $f(E(G))=S_{q-1,k,0,d}$, then $f$ is called \emph{strongly felicitous-difference $(k,d)$-total coloring}; and moreover we call $f$ \emph{felicitous-difference $(k,d)$-total coloring} if the cardinality $|f(E(G))|\leq q$ and $\big |f(u)+f(v)-f(uv)\big |=c$ for each edge $uv\in E(G)$.
\item If the \emph{graceful-difference constraint} $\big ||f(u)-f(v)|-f(uv)\big |=c$ for each edge $uv\in E(G)$ and the edge color set $f(E(G))=S_{q-1,k,0,d}$, then we call $f$ to be \emph{strongly graceful-difference $(k,d)$-total coloring}; and we call $f$ \emph{graceful-difference $(k,d)$-total coloring} if the cardinality $|f(E(G))|\leq q$ and $\big ||f(u)-f(v)|-f(uv)\big |=c$ for each edge $uv\in E(G)$.\qqed
\end{asparaenum}
\end{defn}

\begin{thm}\label{thm:tree-graceful-total-coloringss}
\cite{Yao-Su-Wang-Hui-Sun-ITAIC2020} Each tree admits a $(k,d)$-gracefully total coloring defined in Definition \ref{defn:kd-w-type-colorings}, also, a set-ordered gracefully total coloring as $(k,d)=(1,1)$, and a set-ordered odd-gracefully total coloring as $(k,d)=(1,2)$.
\end{thm}

\begin{defn} \label{defn:set-orderedw-hyperedge-set}
A \emph{set-ordered $W$-constraint hyperedge set}
\begin{equation}\label{eqa:555555}
\mathcal{E}_i=\big \{e^x_{i,j}:j\in [1,a_x]\big\}\bigcup \big \{e^E_{i,j}:j\in [1,b_E]\big\}\bigcup \big \{e^y_{i,j}:j\in [1,c_y]\big\}
\end{equation} holds:
\begin{asparaenum}[\textbf{\textrm{Sochs}}-1.]
\item \textbf{(Hyperedge set)} $\mathcal{E}_i\in \mathcal{E}\big([a,b]^2\big)$ and $[a,b]=\bigcup_{e\in \mathcal{E}_i}e$;
\item \textbf{(Set-ordered constraint)} $\max \big \{\max e^x_{i,j}:j\in [1,a_x]\big \}<\min \big \{\min e^y_{i,j}:j\in [1,c_y]\big \}$;
\item \textbf{($W$-constraint)} each $\gamma \in e^E_{i,k}$ with $k\in [1,b_E]$ corresponds to $\alpha\in e^x_{i,s}$ for some $s\in [1,a_x]$ and $\beta\in e^y_{i,t}$ for some $t\in [1,c_y]$ holding the $W$-constraint $W[\alpha,\gamma,\beta]=0$.
\end{asparaenum}

And moreover, we say a set-ordered $W$-constraint hyperedge set $\mathcal{E}_i$ to be \emph{\textbf{full}} if

(i) each $\alpha\in e^x_{i,s}$ for $s\in [1,a_x]$ corresponds to $\gamma \in e^E_{i,k}$ for some $k\in [1,b_E]$ and $\beta\in e^y_{i,t}$ for some $t\in [1,c_y]$ holding the $W$-constraint $W[\alpha,\gamma,\beta]=0$; and

(ii) each $\beta\in e^y_{i,t}$ with $t\in [1,c_y]$ corresponds to $\alpha\in e^x_{i,s}$ for some $s\in [1,a_x]$ and $\gamma \in e^E_{i,k}$ for some $k\in [1,b_E]$ holding the $W$-constraint $W[\alpha,\gamma,\beta]=0$.\qqed
\end{defn}

Suppose that a bipartite $(p,q)$-graph $H$ admits a $W$-constraint total coloring $f:V(H)\cup E(H)\rightarrow [a,b]$, such that the total color set $f(V(H)\cup E(H))=[a,b]$. Since $V(H)=X\cup Y$ with $X\cap Y=\emptyset$, then we have a hyperedge set $\mathcal{E}^*=\{f(X), f(E(H)),f(Y)\}$ with
\begin{equation}\label{eqa:makk-hyperedge-sets}
[a,b]=\bigcup_{e\in \mathcal{E}^*}e=f(X)\bigcup f(E(H))\bigcup f(Y)
\end{equation} such that each $f(x_i)\in f(X)$ corresponds to some $f(e_i)\in f(E(H))$ and $f(y_i)\in f(Y)$ with $e_i=x_iy_i\in E(H)$ holding a $W$-constraint $W[f(x_i),f(e_i),f(y_i)]=0$ and some other constraints of a constraint set $R_{est}(c_1,c_2,\dots, c_m)$.

\begin{example}\label{exa:8888888888}
In the hyperedge set $\mathcal{E}^*=\{f(X), f(E(H)),f(Y)\}$, three color sets $f(X)$, $f(E(H))$ and $f(Y)$ defined in Eq.(\ref{eqa:makk-hyperedge-sets}) can form small color subsets, so there are many hyperedge sets $\mathcal{E}\in \mathcal{E}([a,b]^2)$, like the above hyperedge set $\mathcal{E}^*=\{f(X), f(E(H))$, $f(Y)\}$.

In Fig.\ref{fig:W-constraint-hypergraphs}, the bipartite Hanzi-graph $H_1$ admits a total set-ordered graceful labeling $h:V(H_1)\cup E(H_1)\rightarrow [0,9]$, such that the color set $h(V(H_1)\cup E(H_1))=[0,9]$.

(i) Notice that $V(H_1)=X\cup Y$ with $X\cap Y=\emptyset$, we have a hyperedge set $\mathcal{E}_1=\{e_{1,1},e_{1,2},e_{1,3}\}$, where $h(X)=e_{1,1}=\{0,2,3,4\}$, $h(Y)=e_{1,2}=\{5,7,8,9\}$ and $h(E(H_1))=e_{1,3}=[1,9]$, such that each $\alpha\in e_{1,1}$ corresponds to $\beta\in e_{1,2}$ and $\gamma\in e_{1,3}$ holding the set-ordered graceful-constraint
\begin{equation}\label{eqa:set-ordered-graceful-constraints}
\{\gamma=\beta-\alpha:\alpha\in e_{1,1},\beta\in e_{1,2},\gamma\in e_{1,3}\}=[1,9]
\end{equation} and vice versa. So, $\mathcal{E}_1$ is full, and moreover, $[0,9]=\bigcup_{e_{1,i}\in \mathcal{E}_1}e_{1,i}=\bigcup ^3_{i=1}e_{1,i}$.

(ii) The second hyperedge set is $\mathcal{E}_2=\{e_{2,1},e_{2,2},e_{2,3}, e_{2,4},e_{2,5}\}$, where $e_{2,1}=\{0,2\}$, $e_{2,2}=\{3,4\}$, $e_{2,3}=\{5,7\}$, $e_{2,4}=\{8,9\}$ and $e_{2,5}=[1,9]$. Clearly, the hyperedge set $\mathcal{E}_2$ is full and holds the set-ordered graceful-constraint like that shown in Eq.(\ref{eqa:set-ordered-graceful-constraints}), as well as $[0,9]=\bigcup_{e_{2,i}\in \mathcal{E}_2}e_{2,i}=\bigcup ^5_{i=1}e_{2,i}$.

(iii) The third hyperedge set $\mathcal{E}_3=\{e_{3,1},e_{3,2},e_{3,3}, e_{3,4},e_{3,5}\}$ (see $H_2$ shown in Fig.\ref{fig:W-constraint-hypergraphs}), where $e_{3,1}=\{0,2\}$, $e_{3,2}=\{0,3\}$, $e_{3,3}=\{0,4\}$, $e_{3,4}=\{5,7\}$, $e_{3,5}=\{5,8\}$, $e_{3,6}=\{5,9\}$, $e_{3,7}=[1,5]$ and $e_{3,8}=[6,9]$. It is not hard to verify that the hyperedge set $\mathcal{E}_3$ is full and holds the set-ordered graceful-constraint like that shown in Eq.(\ref{eqa:set-ordered-graceful-constraints}), as well as $[0,9]=\bigcup_{e_{3,i}\in \mathcal{E}_3}e_{3,i}=\bigcup ^8_{i=1}e_{3,i}$.

The above three hyperedge sets $\mathcal{E}_1,\mathcal{E}_2,\mathcal{E}_3$ are the full set-ordered $W$-constraint hyperedge sets according to Definition \ref{defn:set-orderedw-hyperedge-set}. \qqed
\end{example}

\begin{figure}[h]
\centering
\includegraphics[width=15cm]{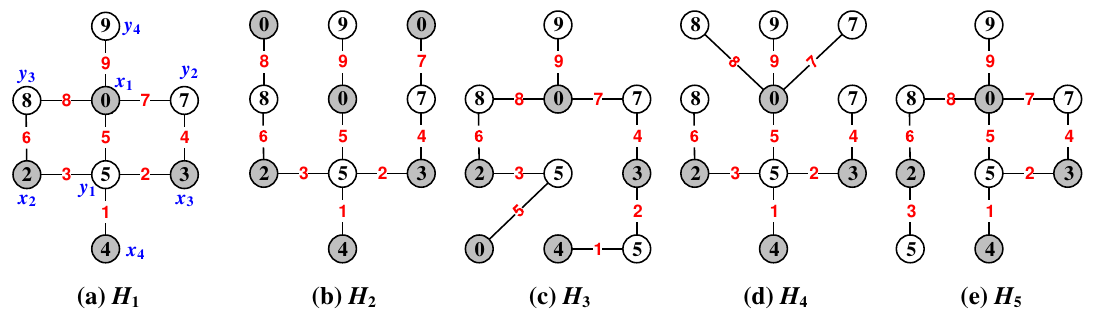}\\
\caption{\label{fig:W-constraint-hypergraphs}{\small A scheme for illustrating the $W$-constraint hyperedge sets.}}
\end{figure}

\begin{defn} \label{defn:hyperedge-set-generator}
Suppose that a bipartite graph $G$ admits a set-ordered $W$-constraint total coloring $f:V(G)\cup E(G)\rightarrow [a,b]$, and $V(G)=X\cup Y$ with $X\cap Y=\emptyset$ holding the set-ordered constraint $\max f(X)<\min f(Y)$. If a set-ordered $W$-constraint hyperedge set $\mathcal{E}_i$ defined in Definition \ref{defn:set-orderedw-hyperedge-set} holds $f(X)=\big \{e^x_{i,j}:j\in [1,a_x]\big \}$, $f(E)=\big \{e^E_{i,j}:j\in [1,b_E]\big \}$ and $f(Y)=\big \{e^y_{i,j}:j\in [1,c_y]\big \}$, then the bipartite graph $G$ is called \emph{topological generator} of the set-ordered $W$-constraint hyperedge set $\mathcal{E}_i$, and the symbol $G_{ener}(G,\mathcal{E})$ denotes the set containing all set-ordered $W$-constraint hyperedge sets generated topologically by the bipartite graph $G$.\qqed
\end{defn}

\begin{thm}\label{thm:666666}
If a bipartite graph $G$ admits a set-ordered $W$-constraint total coloring $f:V(G)\cup E(G)\rightarrow [a,b]$, or a $(k,d)$-$W$-constraint total coloring defined in Definition \ref{defn:kd-w-type-colorings}, then there exists a topological generator $G_{ener}(G,\mathcal{E})$ defined in Definition \ref{defn:hyperedge-set-generator}.
\end{thm}

\begin{defn} \label{defn:w-constraint-hypergraphs}
Let $[a,b]$ be a consecutive integer set. For a 3I-hyperedge set $\mathcal{E}\in \mathcal{E}([a,b]^2)$, if each hyperedge $e\in \mathcal{E}$ corresponds to two hyperedges $e\,'\in \mathcal{E}\setminus e$ and $e\,''\in \mathcal{E}\setminus \{e,e\,'\}$, such that there are numbers $\alpha\in e$, $\beta\in e\,'$ and $\gamma\in e\,''$ holding the $W$-constraint $W[\alpha,\beta,\gamma]=0$ and some other constraints of a constraint set $R_{est}(c_1,c_2,\dots, c_m)$, we call the hypergraph $\mathcal{H}_{yper}=([a,b],\mathcal{E})$ to be \emph{$W$-constraint hypergraph}, and the 3I-hyperedge set $\mathcal{E}$ to be \emph{$W$-constraint 3I-hyperedge set}.

Moreover, if each number $\alpha\in e\in \mathcal{E}$ corresponds to two numbers $\beta\in e\,'\in \mathcal{E}\setminus e$ and $\gamma\in e\,''\in \mathcal{E}\setminus \{e,e\,'\}$ holding the $W$-constraint $W[\alpha,\beta,\gamma]=0$ and some other constraints of the constraint set $R_{est}(c_1,c_2,\dots, c_m)$, then we say the $W$-constraint 3I-hyperedge set $\mathcal{E}$ to be \emph{full}, and the hypergraph $\mathcal{H}_{yper}=([a,b],\mathcal{E})$ to be \emph{full $W$-constraint hypergraph}.\qqed
\end{defn}

\begin{problem}\label{question:444444}
By Definition \ref{defn:set-orderedw-hyperedge-set} and Definition \ref{defn:w-constraint-hypergraphs}: \textbf{(A)} \textbf{Make} $W$-constraint hyperedge sets based on a consecutive integer set $[a,b]$. For example, three numbers $\alpha\in e\in \mathcal{E}$, $\beta\in e\,'\in \mathcal{E}\setminus e$ and $\gamma\in e\,''\in \mathcal{E}\setminus \{e,e\,'\}$ hold the following \emph{magic-constraints}:

(A-i) The \emph{edge-magic constraint} $\alpha+\gamma+\beta=k$ for a non-negative constant $k$.

(A-ii) The \emph{edge-difference constraint} $\gamma+|\alpha-\beta|=k$ for a non-negative constant $k$.

(A-iii) The \emph{graceful-difference constraint} $\big||\alpha-\beta|-\gamma \big|=k$ for a non-negative constant $k$.

(A-iv) The \emph{felicitous-difference constraint} $\big|\alpha+\beta-\gamma \big|=k$ for a non-negative constant $k$.

\textbf{(B)} \textbf{Determine} the topological generator $G_{ener}(G,\mathcal{E})$ defined in Definition \ref{defn:hyperedge-set-generator} for a bipartite graph $G$ and one $W$-constraint from the edge-magic constraint, the edge-difference constraint, the graceful-difference constraint and the felicitous-difference constraint.
\end{problem}

\begin{defn} \label{defn:Edge-w-constraint-graphs}
\textbf{The edge-$W$-constraint graphs of $W$-constraint hypergraphs.} For a $W$-constraint hypergraph $\mathcal{H}_{yper}=([a,b],\mathcal{E})$ based on a $3$I-hyperedge set $\mathcal{E}\in \in \mathcal{E}([a,b]^2)$, a graph $G$ admits a \emph{total hyperedge-set coloring} $F:V(G)\cup E(G)\rightarrow \mathcal{E}$, such that there are $c\in F(uv)$, $a\in F(u)$ and $b\in F(v)$ holding the $W$-constraint $W[a,c,b]=0$ and some other constraints of a constraint set $R_{est}(c_1,c_2,\dots, c_m)$.

Conversely, three numbers $\alpha\in e\in \mathcal{E}$, $\beta\in e\,'\in \mathcal{E}$ and $\gamma\in e\,''\in \mathcal{E}$ holding the $W$-constraint $W[\alpha,\beta,\gamma]=0$ and some other constraints of the constraint set $R_{est}(c_1,c_2,\dots, c_m)$ correspond to always an edge $xy\in E(G)$ such that $\gamma\in F(xy)$, $\alpha\in F(x)$ and $\beta\in F(y)$, then we call the graph $G$ \emph{edge-$W$-constraint graph} of the $W$-constraint hypergraph $\mathcal{H}_{yper}=([a,b],\mathcal{E})$.\qqed
\end{defn}

\begin{rem}\label{rem:333333}
In Definition \ref{defn:Edge-w-constraint-graphs}, an edge-$W$-constraint graph of the $W$-constraint hypergraph $\mathcal{H}_{yper}=([a,b],\mathcal{E})$ is like some ve-intersected graph of a hypergraph defined in Definition \ref{defn:vertex-intersected-graph-hypergraph}.\qqed
\end{rem}

\subsubsection{Parameterized hypergraphs}

We will use the following terminology and natation:

\begin{asparaenum}[$\bullet$~]
\item For a \emph{parameterized set} $\Lambda_{(m,b,n,k,a,d)}=S_{m,0,b,d}\cup S_{n,k,a,d}$ with
\begin{equation}\label{eqa:555555}
S_{m,0,b,d}=\big \{bd,(b+1)d\dots, md\big \},~ S_{n,k,a,d}=\big \{k+ad,k+(a+1)d,\dots ,k+(a+n)d\big \}
\end{equation} the power set $\Lambda^2_{(m,b,n,k,a,d)}$ collects all subsets of the parameterized set $\Lambda_{(m,b,n,k,a,d)}$.

\item A \emph{parameterized hyperedge set} $\mathcal{E}^P\in \mathcal{E}_{gen}(\Lambda^2_{(m,b,n,k,a,d)})$ holds $\bigcup _{e\in \mathcal{E}^P}e=\Lambda_{(m,b,n,k,a,d)}$ true.
\end{asparaenum}

Motivated from the hypergraph definition, we present the parameterized hypergraph as follows:

\begin{defn}\label{defn:parameterized-hypergraph-basic-definition}
\cite{Bing-Yao-arXiv:2207-03381} A \emph{parameterized hypergraph} $\mathcal{P}_{hyper}=(\Lambda_{(m,b,n,k,a,d)},\mathcal{E}^P)$ defined on a parameterized hypervertex set $\Lambda_{(m,b,n,k,a,d)}$ holds:

(i) Each element of $\mathcal{E}^P$ is not empty and called a \emph{parameterized hyperedge};

(ii) $\Lambda_{(m,b,n,k,a,d)}=\bigcup _{e\in \mathcal{E}^P}e$, where each element of $\Lambda_{(m,b,n,k,a,d)}$ is called a \emph{vertex}, and the cardinality $|\Lambda_{(m,b,n,k,a,d)}|$ is \emph{order} of $\mathcal{P}_{hyper}$;

(iii) $\mathcal{E}^P$ is called \emph{parameterized hyperedge set}, and the cardinality $|\mathcal{E}^P|$ is \emph{size} of $\mathcal{P}_{hyper}$. \qqed
\end{defn}

\begin{problem}\label{problem:99999}
By Definition \ref{defn:parameterized-hypergraph-basic-definition}, let $S(\Lambda^P)=\big \{\mathcal{E}^P_1,\mathcal{E}^P_2,\dots ,\mathcal{E}^P_{M}\big \}$ be the set of parameterized hyperedge sets defined on a parameterized set $\Lambda_{(m,b,n,k,a,d)}$, so that each parameterized hyperedge set $\mathcal{E}^P_i$ with $i\in [1,M]$ holds $\Lambda_{(m,b,n,k,a,d)}=\bigcup _{e\in \mathcal{E}^P_i}e$ true. \textbf{Estimate} the number $M$ of parameterized hypergraphs defined on the parameterized set $\Lambda_{(m,b,n,k,a,d)}$.
\end{problem}

\begin{defn} \label{defn:operation-graphs-pa-hypergraph}
\cite{Bing-Yao-arXiv:2207-03381} Let $\textbf{\textrm{O}}=(O_1,O_2,\dots ,O_m)$ be an operation set with $m\geq 1$, and if an element $c$ is obtained by implementing an operation $O_i\in \textbf{\textrm{O}}$ to other two elements $a,b$, we write this case as $c=a[O_i]b$. Suppose that a $(p,q)$-graph $G$ admits a proper $(k,d)$-total set-coloring $F: V(G)\cup E(G)\rightarrow \mathcal{E}^P$, where $\mathcal{E}^P$ is the parameterized hyperedge set of a parameterized hypergraph $\mathcal{P}_{hyper}=(\Lambda_{(m,b,n,k,a,d)},\mathcal{E}^P)$ defined in Definition \ref{defn:parameterized-hypergraph-basic-definition}. There are the following constraints:
\begin{asparaenum}[\textrm{\textbf{Pahy}}-1.]
\item \label{pahyper:parameterized-graph-3} Only one operation $O_k\in \textbf{\textrm{O}}$ holds $c_{uv}=a_u[O_k]b_v$ for each edge $uv\in E(G)$, where $a_u\in F(u)$, $b_v\in F(v)$ and $c_{uv}\in F(uv)$.
\item \label{pahyper:parameterized-graph-1} For each operation $O_i\in \textbf{\textrm{O}}$, each edge $uv\in E(G)$ holds $c_{uv}=a_u[O_i]b_v$ with some elements $a_u\in F(u)$, $b_v\in F(v)$ and $c_{uv}\in F(uv)$.
\item \label{pahyper:parameterized-graph-2} Each element $z\in F(uv)$ for each edge $uv\in E(G)$ corresponds to an operation $O_j\in \textbf{\textrm{O}}$, such that $z=x[O_j]y$ for some $x\in F(u)$ and $y\in F(v)$.
\item \label{pahyper:parameterized-graph-v} Each element $a_x\in F(x)$ for any vertex $x\in V(G)$ corresponds to an operation $O_s\in \textbf{\textrm{O}}$ and an adjacent vertex $y\in N_{ei}(x)$, such that $z_{xy}=a_x[O_s]b_y$ for some $z_{xy}\in F(xy)$ and $b_y\in F(y)$.
\item \label{pahyper:parameterized-graph-4} Each operation $O_t\in \textbf{\textrm{O}}$ corresponds to some edge $xy\in E(G)$ holding $c_{xy}=a_x[O_t]b_y$ for $a_x\in F(x)$, $b_y\in F(y)$ and $c_{xy}\in F(xy)$.
\item \label{pahyper:parameterized-graph-must} If there are three different sets $e_i,e_j,e_k\in F(V(G))$ and an operation $O_r\in \textbf{\textrm{O}}$ holding $e_i[O_r]e_j\subseteq e_k$, then there exists an edge $xy\in E(G)$, such that $F(x)=e_i$, $F(y)=e_j$ and $F(xy)=e_k$.
\end{asparaenum}
\noindent \textbf{Then, the $(p,q)$-graph $G$ is}:
\begin{asparaenum}[\textrm{\textbf{Ograph}}-1.]
\item a \emph{$(k,d)$-$\textbf{\textrm{O}}$-vertex operation graph} of the parameterized hypergraph $\mathcal{P}_{hyper}$ if Pahy-\ref{pahyper:parameterized-graph-1} and Pahy-\ref{pahyper:parameterized-graph-must} hold true.
\item a \emph{$(k,d)$-$\textbf{\textrm{O}}$-edge operation graph} of the parameterized hypergraph $\mathcal{P}_{hyper}$ if Pahy-\ref{pahyper:parameterized-graph-1}, Pahy-\ref{pahyper:parameterized-graph-2} and Pahy-\ref{pahyper:parameterized-graph-must} hold true.
\item a \emph{$(k,d)$-$\textbf{\textrm{O}}$-total operation graph} of the parameterized hypergraph $\mathcal{P}_{hyper}$ if Pahy-\ref{pahyper:parameterized-graph-1}, Pahy-\ref{pahyper:parameterized-graph-2}, Pahy-\ref{pahyper:parameterized-graph-v} and Pahy-\ref{pahyper:parameterized-graph-must} hold true.
\item a \emph{$(k,d)$-edge operation graph} of the parameterized hypergraph $\mathcal{P}_{hyper}$ if Pahy-\ref{pahyper:parameterized-graph-2} and Pahy-\ref{pahyper:parameterized-graph-must} hold true.
\item a \emph{$(k,d)$-vertex operation graph} of the parameterized hypergraph $\mathcal{P}_{hyper}$ if Pahy-\ref{pahyper:parameterized-graph-v} and Pahy-\ref{pahyper:parameterized-graph-must} hold true.
\item a \emph{$(k,d)$-total operation graph} of the parameterized hypergraph $\mathcal{P}_{hyper}$ if Pahy-\ref{pahyper:parameterized-graph-2}, Pahy-\ref{pahyper:parameterized-graph-v} and Pahy-\ref{pahyper:parameterized-graph-must} hold true.
\item a \emph{$(k,d)$-non-uniform operation graph} of the parameterized hypergraph $\mathcal{P}_{hyper}$ if Pahy-\ref{pahyper:parameterized-graph-4} and Pahy-\ref{pahyper:parameterized-graph-must} hold true.
\item a \emph{$(k,d)$-one operation graph} of the parameterized hypergraph $\mathcal{P}_{hyper}$ if Pahy-\ref{pahyper:parameterized-graph-3} and Pahy-\ref{pahyper:parameterized-graph-must} hold true.\qqed
\end{asparaenum}
\end{defn}

\begin{example}\label{exa:8888888888}
About Definition \ref{defn:operation-graphs-pa-hypergraph}, if the operation set $\textbf{\textrm{O}}$ contains only one operation ``$\bigcap$'', which is the \emph{intersection operation} on sets, and the $(p,q)$-graph $G$ satisfies the constraints Pahy-\ref{pahyper:parameterized-graph-1} and Pahy-\ref{pahyper:parameterized-graph-must} of Definition \ref{defn:operation-graphs-pa-hypergraph}, then $G$ is called \emph{$(k,d)$-ve-intersected graph} of the parameterized hypergraph $\mathcal{P}_{hyper}$. As $(k,d)=(1,1)$, the graph $G$ is just a \emph{ve-intersected graph} of a hypergraph $H_{hyper}$ defined in \cite{Jianfang-Wang-Hypergraphs-2008}.\qqed
\end{example}

\subsection{Operation-hypergraphs}

\subsubsection{Pan-operation graphs of non-ordinary hypergraphs}

In particular cases, pan-operation graphs of non-ordinary hypergraphs are v-intersected graphs and ve-intersected graphs of ordinary hypergraphs introduced in Definition \ref{defn:new-intersected-hypergraphss} and Definition \ref{defn:vertex-intersected-graph-hypergraph}, respectively.

\begin{defn} \label{defn:pan-operation-graph-a-hypergraph}
\textbf{The pan-operation graphs of non-ordinary hypergraphs.} Let $[\bullet]$ be an operation on sets, and let $\mathcal{E}$ be an IOI-hyperedge set of $\mathcal{E}(\Lambda^2)$ defined on a finite set $\Lambda$ such that each hyperedge $e\in \mathcal{E}$ corresponds to another hyperedge $e\,'\in \mathcal{E}$ holding $e[\bullet] e\,'$ to be one subset of the power set $\Lambda^2$, and $\Lambda=\bigcup _{e\in \mathcal{E}}e$, as well as $R_{est}(c_0,c_1,c_2,\dots ,c_m)$ be a constraint set with $m\geq 0$, in which the first constraint $c_0:= e[\bullet] e\,'\in \Lambda^2$. If a graph $G$ admits a total set-coloring $\theta: V(H)\cup E(H)\rightarrow \mathcal{E}$ holding $\theta(x)\neq \theta(y)$ for each edge $uv\in E(G)$, and

(i) the first constraint $c_0: ~\theta(uv)\supseteq \theta(u)[\bullet]\theta(v)\neq \emptyset$.

(ii) For some $k\in [1,m]$, there is a function $\pi_k$, such that the $k$th constraint $c_k:=\pi_k (b_{u},b_{uv},b_{v})=0$ with $k\in [1,m]$ for $b_{uv}\in \theta(uv)$, $b_{u}\in \theta(u)$ and $b_{v}\in \theta(v)$.

(iii) If any pair of distinct hyperedges $e_i,e_j\in \mathcal{E}$ holding $e_i[\bullet] e_j\in \Lambda^2$, then there is an edge $x_iy_j$ of the graph $G$, such that $\theta(x_i)=e_i$ and $\theta(y_j)=e_j$.

Then we call $G$ \emph{pan-operation graph} of the non-ordinary hypergraph $\mathcal{H}_{yper}=(\Lambda,\mathcal{E})$. \qqed
\end{defn}

\begin{defn} \label{defn:pan-groups-pan-element-sets}
\textbf{The every-zero pan-group.} An \emph{every-zero pan-group} $\{P_{an}(G);[+][-]\}$ has its own \emph{pan-element set} $P_{an}(G)=\{G_1,G_2$, $\dots$, $G_m\}$ with each pan-element $G_i$ admitting a total pan-coloring $h_i$, and $G_i\cong G_1$. The finite module abelian additive operation $G_i[+]G_j[-]G_k=G_\lambda $ based on the pan-element set $P_{an}(G)$ is defined by
\begin{equation}\label{eqa:3333333}
h_i(w)[+]h_j(w)[-]h_k(w)=h_\lambda(w),~w\in V(G_1)\cup E(G_1)
\end{equation} with $\lambda=i+j-k~(\bmod~M)$ for any preappointed \emph{zero} $G_k\in \{P_{an}(G);[+][-]\}$.\qqed
\end{defn}

\subsubsection{Compound hypergraphs}

From studying relationship between communities in networks, which is the topological structure between hypergraphs, we propose the following two concepts of the hyperedge-set coloring and the compound set-coloring:

\begin{defn}\label{defn:hypergraph-intersected-hypergraph}
\textbf{The hypergraph-intersected graphs.} Let $S_{\mathcal{E}}=\{\mathcal{E}_1,\mathcal{E}_1,\dots, \mathcal{E}_m\}$ with $m\geq 2$ be a \emph{3I-hyperedge set-set} with the 3I-hyperedge sets $\mathcal{E}_i\in\mathcal{E}\big (\Lambda^2\big )$ based on a finite set $\Lambda=\{x_1,x_2$, $\dots$, $x_m\}$ (Ref. Definition \ref{defn:new-hypergraphs-sets}), such that each 3I-hyperedge set $\mathcal{E}_i\in S_{\mathcal{E}}$ corresponds to another 3I-hyperedge set $\mathcal{E}_i\in S_{\mathcal{E}}$ with $i\neq j$ holding $\mathcal{E}_i\cap \mathcal{E}_j\neq \emptyset$. Suppose that a $(p,q)$-graph $G$ admits a \emph{hyperedge-set coloring} $\theta: V(G)\rightarrow S_{\mathcal{E}}$, such that each edge $uv$ of $E(G)$ holds $\theta(u)\cap \theta(v)\neq \emptyset$. Conversely, each pair of 3I-hyperedge sets $\mathcal{E}_i,\mathcal{E}_j\in S_{\mathcal{E}}$ with the intersected operation $\mathcal{E}_i\cap \mathcal{E}_j\neq \emptyset$ corresponds to an edge $xy\in E(G)$, such that $\theta(x)=\mathcal{E}_i$ and $\theta(y)=\mathcal{E}_j$. Then we call $G$ \emph{hypergraph-intersected graph} of the \emph{hyperedge-compound hypergraph} $\mathcal{H}_{comp}=(\Lambda, S_{\mathcal{E}} )$.\qqed
\end{defn}

\begin{rem}\label{rem:333333}
We generalize Definition \ref{defn:new-intersected-hypergraphss}, Definition \ref{defn:vertex-intersected-graph-hypergraph} and Definition \ref{defn:hypergraph-intersected-hypergraph} to the following $[\bullet]$-operation graphs:

(i) In Definition \ref{defn:new-intersected-hypergraphss}, if each edge $uv\in E(H)$ holds $\varphi(u)[\bullet]\varphi(v)$ true for an operation $[\bullet]$ on the sets of a hyperedge set $\mathcal{E}$, then we call $H$ \emph{v-$[\bullet]$-operation graph} of the hypergraph $\mathcal{H}_{yper}=(\Lambda,\mathcal{E})$.

(ii) In Definition \ref{defn:vertex-intersected-graph-hypergraph}, if each edge $uv\in E(H)$ holds $\varphi(u)[\bullet]\varphi(v)$ true for an operation $[\bullet]$ on a group of hyperedge sets, then we call $H$ \emph{ve-$[\bullet]$-operation graph}.

(iii) In Definition \ref{defn:hypergraph-intersected-hypergraph}, if each edge $uv\in E(G)$ holds $\theta(u)[\bullet]\theta(v)$ true for an operation $[\bullet]$ on the sets of a \emph{hyperedge set-set} $S_{\mathcal{E}}$, then we call $G$ \emph{hypergraph-$[\bullet]$-operation graph}.\qqed
\end{rem}

\begin{example}\label{exa:8888888888}
Let $K_{2n}$ be a complete graph of $2n$ vertices. Then $K_{2n}$ has perfect matching groups $\textbf{\textrm{M}}_i=\{M_{i,1},M_{i,2}$, $\dots $, $M_{i,2n-1}\}$ for $i\in [1,m]$. There is a graph $G$ admitting a set-coloring $F:V(G)\rightarrow \mathcal{E}_i$ with each hyperedge set $\mathcal{E}_i=\textbf{\textrm{M}}_i$ such that $F(x)\neq F(y)$ for distinct vertices $x,y\in V(G)$, and each edge $uv\in E(G)$ is colored with an induced set $F(uv)=F(u)\cup F(v)$ if $E(C_k)=F(u)\cup F(v)$, where $C_k$ is a hamiltonian cycle of $K_{2n}$. As the union operation ``$\cup $'' is the abstract operation ``$[\bullet ]$'' appeared in Definition \ref{defn:compound-set-coloring-hypergraph} and $G$ is a complete graph, then we have:
\begin{conj} \label{conj:perfect-1-factorization-conjecture}
\textbf{Perfect 1-Factorization Conjecture} (Anton Kotzig, 1964): For integer $n \geq 2$, $K_{2n}$ can be decomposed into $2n-1$ perfect matchings such that the union ``$\cup $'' of any two matchings forms a hamiltonian cycle of $K_{2n}$.
\end{conj}
\end{example}

\begin{conj} \label{conj:c2-KT-conjecture}
$K$-$T$ \textbf{conjecture} (Gy\'{a}r\'{a}s and Lehel, 1978; B\'{e}la Bollob\'{a}s, 1995): For integer $n\geq 3$, given $n$ mutually vertex disjoint trees $T_k$ of $k$ vertices with respect to $1\leq k\leq n$. Then the complete graph $K_n$ can be decomposed into the union of $n$ mutually edge disjoint trees $H_k$, namely $K_n=\bigcup ^{n}_{k=1}H_k$, such that $T_k\cong H_k$ whenever $1\leq k\leq n$, write this case as $\langle T_1,T_2,\dots, T_n\mid K_n\rangle$.
\end{conj}

\textbf{The tree-group homomorphism.} By Conjecture \ref{conj:c2-KT-conjecture}, there is a tree-group homomorphism $(T_1,T_2,\dots, T_n)\rightarrow K_n$ in the view of homomorphism. Each $[1,n]$-tree $e_{i,j}=\{T_{i,j,1},T_{i,j,2},\dots, T_{i,j,n}\}$ with $|V(T_{i,j,s})|=s$ for $s\in [1,n]$ holds $\langle e_{i,j}\mid K_n\rangle$ and $e_{i,j}\rightarrow K_n$ true, or the complete graph $K_n$ can be vertex-splitting into trees $T_1,T_2,\dots, T_n$. A hyperedge set $S_{i}=\{e_{i,1},e_{i,2},\dots, e_{i,m}\}$ is a $[1,n]$-tree set, such that each $[1,n]$-tree $e_{i,j}\in S_{i}$ corresponds to another $[1,n]$-tree $e_{i,k}\in S_{i}\setminus e_{i,j}$ holding $e_{i,j}\cap e_{i,k}\neq \emptyset$. We get a \emph{$[1,n]$-tree homomorphism general hypergraph} $\mathcal{H}_{yper}^{gen}=(T_{ree}(\leq n),S_{i})$, where $T_{ree}(\leq n)$ is the set of trees of $p$ vertices with $1\leq p\leq n$. We have a $[1,n]$-tree set $G_{tree}(K_n)=\{S_{i}=\{e_{i,1},e_{i,2},\dots, e_{i,m}\}:i\in [1,A_n]\}$, where $A_n$ is the number of $[1,n]$-trees.

A graph $G$ admitting a $[1,n]$-tree set-coloring $f:V(G)\rightarrow S_i$, such that each edge $uv\in E(G)$ is colored with $f(uv)=f(u)\cap f(u)\neq \emptyset$. So, the graph $G$ is a v-intersected graph of the $[1,n]$-tree homomorphism general hypergraph $\mathcal{H}_{yper}^{gen}=(T_{ree}(\leq n),S_{i})$ if each pair of $[1,n]$-tree $e_{i,j}$ and $e_{i,k}$ holding $e_{i,j}\cap e_{i,k}\neq \emptyset$ corresponds to an edge of the graph $G$. Notice that the hyperedge set $S_{i}\not \in \mathcal{E}(T^2_{ree}(\leq n))$, since $T_{ree}(\leq n)\neq \bigcup_{e_{i,s}\in S_{i}}e_{i,s}$.

\begin{defn} \label{defn:compound-set-coloring-hypergraph}
Suppose that a graph $G$ admits a proper \emph{compound set-coloring} $\Gamma:V(G)\rightarrow S_{\mathcal{E}}=\{\mathcal{E}_1,\mathcal{E}_1,\dots, \mathcal{E}_m\}$ with $\Gamma(x)\neq \Gamma(y)$ for each edge $xy\in E(G)$, where each $\mathcal{E}_i$ is a set of subsets of a finite set $\Lambda$, and $\Lambda=\bigcup^n_{i=1}\Lambda_i$ with $\Lambda_i=\bigcup_{e_{i,j}\in \mathcal{E}_i}e_{i,j}$, and each $(\Lambda_i, \mathcal{E}_i)$ is a hypergraph. If there are a function $\psi$ and a constraint set $R_{est}(c_1,c_2,\dots, c_m)$ such that each edge $u_iv_j$ is colored with an induced edge color
\begin{equation}\label{eqa:compound-set-coloring-hypergraph-11}
\Gamma(u_iv_j)=\psi(\Gamma(u_i), \Gamma(v_j))\supseteq \Gamma(u_i)\cap \Gamma(v_j)=\mathcal{E}_{u_i}\cap \mathcal{E}_{u_j}\neq\emptyset
\end{equation} or there is an abstract operation ``$[\bullet ]$'' such that each induced edge color set
\begin{equation}\label{eqa:compound-set-coloring-hypergraph-22}
\Gamma(u_iv_j)\supseteq \Gamma(u_i)[\bullet ]\Gamma(v_j)=\mathcal{E}_{u_i}[\bullet ] \mathcal{E}_{u_j}\neq \emptyset
\end{equation} subject to the constraint set $R_{est}(c_1,c_2,\dots, c_m)$, then we call $\mathcal{H}_{comp}=\big (\bigcup^n_{i=1}\Lambda_i, \bigcup^n_{i=1}\mathcal{E}_i\big )$ \emph{compound hypergraph}, and call $G$ \emph{compound ve-intersected graph} of the compound hypergraph $\mathcal{H}_{comp}$ if each intersection $\mathcal{E}_{u_i}\cap \mathcal{E}_{u_j}\neq\emptyset$ corresponds to some edge $u_iv_j$ of the graph $G$. \qqed
\end{defn}

\begin{defn} \label{defn:hypergraph-type-topcode-matrix}
According to Definition \ref{defn:compound-set-coloring-hypergraph}, the compound ve-intersected graph $G$ admits a compound set-coloring $\Gamma:V(G)\rightarrow S_{\mathcal{E}}=\{\mathcal{E}_1,\mathcal{E}_1,\dots, \mathcal{E}_n\}$ holding Eq.(\ref{eqa:compound-set-coloring-hypergraph-11}), so the compound ve-intersected graph $G$ corresponds to its own \emph{hypergraph Topcode-matrix}
\begin{equation}\label{eqa:hypergraph-type-Topcode-matrix}
{
\begin{split}
H^{comp}_{yper}(G)&=\left(
\begin{array}{ccccc}
\Gamma(x_1) & \Gamma(x_2) & \cdots & \Gamma(x_q)\\
\Gamma(x_1y_1) & \Gamma(x_2y_2) & \cdots & \Gamma(x_qy_q)\\
\Gamma(y_1) & \Gamma(y_2) & \cdots & \Gamma(y_q)
\end{array}
\right)=\left(
\begin{array}{ccccc}
\mathcal{E}_{x_1} & \mathcal{E}_{x_2} & \cdots & \mathcal{E}_{x_q}\\
\mathcal{E}_{x_1y_1} & \mathcal{E}_{x_2y_2} & \cdots & \mathcal{E}_{x_qy_q}\\
\mathcal{E}_{y_1} & \mathcal{E}_{y_2} & \cdots & \mathcal{E}_{y_q}
\end{array}
\right)\\
&=(X(S_{\mathcal{E}}),~E(S_{\mathcal{E}}),~Y(S_{\mathcal{E}}))^{T}
\end{split}}
\end{equation}
where $E(G)=\{x_iy_i:i\in [1,q]\}$, $X(S_{\mathcal{E}})=(\mathcal{E}_{x_1},\mathcal{E}_{x_2}, \cdots ,\mathcal{E}_{x_q})$ and $Y(S_{\mathcal{E}})=(\mathcal{E}_{y_1},\mathcal{E}_{y_2}, \cdots ,\mathcal{E}_{y_q})$ are called \emph{v-hypergraph vectors}, and $E(S_{\mathcal{E}})=(\mathcal{E}_{x_1y_1}, \mathcal{E}_{x_2y_2}, \cdots ,\mathcal{E}_{x_qy_q})$ is called \emph{e-hypergraph intersection vector}, as well as $\mathcal{E}_{x_iy_i}=\psi(\mathcal{E}_{x_i},\mathcal{E}_{y_i})$ for each edge $x_iy_i\in E(G)$.\qqed
\end{defn}

\begin{rem}\label{rem:333333}
The hypergraph Topcode-matrix $H^{comp}_{yper}(G)$ defined in Eq.(\ref{eqa:hypergraph-type-Topcode-matrix}) of Definition \ref{defn:hypergraph-type-topcode-matrix} is a three dimensional matrix, and the compound ve-intersected graph $G$ has its own vertices as hypergraphs, its own edges as intersections of the hypergraphs.\qqed
\end{rem}

\subsubsection{Miscellaneous hypergraphs}

\begin{defn} \label{defn:hyper-hypergraphs}
\textbf{The hyper-hypergraph}. By Definition \ref{defn:new-hypergraphs-sets} and a finite set $\Lambda=\{x_1,x_2,\dots ,x_n\}$, we defined a \emph{hyper-hypergraph} as follows: Let $\Phi=\mathcal{E}\big (\Lambda^2\big )=\{\mathcal{E}_i:~i\in [1,n(\Lambda)]\}$ with $n(\Lambda)=|\mathcal{E}\big (\Lambda^2\big )|$, each hyper-hyperedge set $P_i\in \mathcal{E}(\Phi^2)$ is a hyperedge-set set $P_i=\{S_{i,j}:j\in [1,a_i]\}$ with
\begin{equation}\label{eqa:555555}
S_{i,j}=\big \{\mathcal{E}_{i,j,s}\in \mathcal{E}\big (\Lambda^2\big ):s\in [1,b_{i,j}],~j\in [1,a_i]\big \}
\end{equation} and holds $\Phi=\bigcup _{S_{i,j}\in P_i}S_{i,j}$. We get hypergraphs $H^i_{yper}=(\Phi,P_i)$ for $i\in [1,n(\Phi)]$ with $n(\Phi)=|\mathcal{E}\big (\Phi^2\big )|$, and each hypergraph $H^i_{yper}=(\Phi,P_i)$ is called \emph{hyper-hypergraph}, or \emph{hyper$(2)$hypergraph}.\qqed
\end{defn}

We have defined hyper$(n)$hypergraphs for $n\geq 1$ by Definition \ref{defn:hyper-hypergraphs}.

\begin{defn} \label{defn:hyperedge-set-coloring-set-coloring}
\textbf{The vertex coloring of hypergraphs.} Let $H$ be a v-intersected graph of the hypergraph $\mathcal{H}_{yper}=(\Lambda,\mathcal{E})$ based on a 3I-hyperedge set $\mathcal{E}\in \mathcal{E}\big (\Lambda^2\big )$, so we have a hyperedge-set coloring $\varphi: V(H)\rightarrow \mathcal{E}$ defined in Definition \ref{defn:new-intersected-hypergraphss}, such that $\varphi(u)=e_u\in \mathcal{E}$ for each vertex $u\in V(H)$.

Since $H$ admits a proper vertex coloring $\theta:V(H)\rightarrow [1,k]$ holding $\theta(V(H))=[1,k]$, so each vertex $x\in e_u\cap \Lambda$ is colored with the color $\theta(u)\in [1,k]$ for each vertex $u\in V(H)$. For the vertex set $\Lambda$ of the hypergraph $\mathcal{H}_{yper}=(\Lambda,\mathcal{E})$, we get an ordinary set-coloring $\beta:\Lambda\rightarrow [1,k]^2$, such that each vertex $x\in \Lambda$ is colored with a subset $\beta(x)$ of the power set $[1,k]^2$, so $\max\{|\beta(x)|:x\in \Lambda\}\leq k$. We get a hyperedge set $\mathcal{E}^*=\{\beta(x):x\in \Lambda\}$, and the set-coloring $\beta$ is induced by the hyperedge-set coloring $\varphi$ and the proper vertex coloring $\theta$ of the v-intersected graph $H$, denoted as $\beta\prec (\varphi, \theta)$.\qqed
\end{defn}

By Definition \ref{defn:hyperedge-set-coloring-set-coloring}, we construct hyperedge $e_i=\{x:\beta(x)=i,x\in \Lambda\}$ for $i\in [1,k]$, and get a hyperedge set $\mathcal{E}^*=\{e_1,e_2,\dots ,e_k\}$. Clearly, $e_i\not \subseteq e_j$ fort $i\neq j$, and $\Lambda=\bigcup_{e_i\in \mathcal{E}^*}e_i$, as well as each hyperedge $e_i\in \mathcal{E}^*$ corresponds to another hyperedge $e_s\in \mathcal{E}^*\setminus e_i$ holding $e_i\cap e_s\neq \emptyset $. Hence, the hyperedge set $\mathcal{E}^*\in \mathcal{E}\big (\Lambda^2\big )$ is really a 3I-hyperedge set.

\begin{defn} \label{defn:111111}
Let $\Lambda(t)=\{x_{1}(t),x_{2}(t),\dots ,x_{m(t)}(t)\}$, and let $\mathcal{E}(t)$ be a hyperedge set of some subsets of the power set $\Lambda^2(t)$ with $t\in [a,b]$. If the hyperedge set $\mathcal{E}(t)$ is a 3I-hyperedge (IOI-hyperedge) set, namely, $\mathcal{E}(t)\in \mathcal{E}\big (\Lambda^2(t)\big )$, we get a \emph{dynamic hypergraph} $\mathcal{H}_{yper}(t)=(\Lambda(t),\mathcal{E}(t))$ and a \emph{dynamic v-intersected graph} $G(t)$ determined by the dynamic hypergraph $\mathcal{H}_{yper}(t)$. If the hyperedge set $\mathcal{E}(t)$ holds $e\not \in e\,'\in \mathcal{E}(t)\setminus e$ for each hyperedge $e\in \mathcal{E}(t)$, and moreover each hyperedge $e_i\in \mathcal{E}(t)$ corresponds to another hyperedge $e_j\in \mathcal{E}(t)\setminus e_i$ holding $e_i\cap e_j\neq \emptyset$, but $\Lambda(t)\neq \bigcup _{e\in \mathcal{E}(t)}e$, then we get a \emph{dynamic general hypergraph} $\mathcal{H}^{gen}_{yper}(t)=(\Lambda(t),\mathcal{E}(t))$.\qqed
\end{defn}

\vskip 0.4cm

At the end of this article, we recall the \emph{graph network} proposed by 27 scientists from DeepMind, GoogleBrain, MIT and University
of Edinburgh \cite{Peter-Battaglia-Jessica-et-al-arXiv-2018}, they argue that \emph{combinatorial generalization} must be a top priority for AI to achieve human-like abilities, and that structured representations and computations are key to realizing this objective. And they present a new building block for the AI toolkit with a \emph{strong relational inductive bias} --- the \emph{graph network} (or \emph{graph framework}, also, the v-intersected graph $H$ defined in Definition \ref{defn:p-property-hypergraphs}), which generalizes and extends various approaches for Neural Networks that operate on graphs, and provides a straightforward interface for manipulating \emph{structured knowledge} and producing \emph{structured behaviors}. For example, the coloring set $\Lambda_{C}=\{f_1,f_2,\dots ,f_n\}$ defined in Definition \ref{defn:coloring-graph-set-hypergraph} can be considered as a \emph{reasoning function set} introduced in \cite{Peter-Battaglia-Jessica-et-al-arXiv-2018}. The $K_{tree}$-spanning lattice shown in Eq.(\ref{eqa:K-tree-spanning-lattice}), the vertex-coinciding tree-lattice shown in Eq.(\ref{eqa:vertex-coinciding-tree-lattice}) and the edge-joining tree-lattice shown in Eq.(\ref{eqa:edge-joining-tree-lattice}) are \emph{structured knowledge}.

\section*{Acknowledgment}

The author, \emph{Fei Ma}, was supported by the National Natural Science Foundation of China No. 62403381 and the Key Research and Development Plan of Shaanxi Province No.2024GX-YBXM-021. The author, \emph{Bing Yao}, was supported by the National Natural Science Foundation of China under grants No. 61163054, No. 61363060 and No. 61662066.

\end{document}